\documentclass[11pt]{article}

\usepackage[a4paper,margin=1in]{geometry}

\usepackage{amsmath,amssymb,amsthm,mathtools,bm,mathrsfs}

\usepackage{enumitem}
\usepackage{booktabs}
\usepackage{multirow}

\usepackage{graphicx}
\usepackage{epstopdf}
\usepackage[dvipsnames]{xcolor}
\usepackage{hyperref}
\hypersetup{colorlinks=true, linkcolor=blue!50!black, citecolor=blue!50!black, urlcolor=blue!50!black}

\usepackage{algorithm}
\usepackage{algpseudocode}

\usepackage[title]{appendix}
\usepackage{todonotes}
\usepackage{manyfoot}
\usepackage{textcomp}
\usepackage{listings}

\theoremstyle{plain}
\newtheorem{theorem}{Theorem}[section]
\newtheorem{proposition}[theorem]{Proposition}
\newtheorem{lemma}[theorem]{Lemma}

\theoremstyle{definition}
\newtheorem{definition}[theorem]{Definition}

\newtheorem{remark}[theorem]{Remark}

\title{Accelerated Inertial Gradient Algorithms with Vanishing Tikhonov Regularization}

\author{
\begin{tabular}{ccc}
  Samir Adly\thanks{Email: \texttt{samir.adly@unilim.fr}} &
  Vinh Thanh Ho\thanks{Email: \texttt{vinh-thanh.ho@unilim.fr}} &
  Huu Nhan Nguyen \thanks{Email: \texttt{huu-nhan.nguyen@etu.unilim.fr}} 
\end{tabular}
\\[2mm]
\small Laboratoire XLIM, Universit\'e de Limoges,\\
\small 123 Avenue Albert Thomas, 87060 Limoges CEDEX, France
}

\date{}
\begin{document}
\maketitle

{\bf Abstract.}
In this paper, we study an explicit Tikhonov-regularized inertial gradient algorithm for smooth convex minimization with Lipschitz continuous gradient. The method is derived via an explicit time discretization of a damped inertial system with vanishing Tikhonov regularization.
Under appropriate control of the decay rate of the Tikhonov term, we establish accelerated convergence of the objective values to the minimum together with strong convergence of the iterates to the minimum-norm minimizer.
In particular, for polynomial schedules $\varepsilon_k = k^{-p}$ with $0<p<2$, we prove strong convergence to the minimum-norm solution while preserving fast objective decay.
In the critical case $p=2$, we still obtain fast rates for the objective values, while our analysis does not guarantee strong convergence to the minimum-norm minimizer.
Furthermore, we provide a thorough theoretical analysis for several choices of Tikhonov schedules.
Numerical experiments on synthetic, benchmark, and real datasets illustrate the practical performance of the proposed algorithm.

\vskip 2mm


\tableofcontents

\section{Introduction}

Let $\mathcal{H}$ be a real Hilbert space endowed with the scalar product $\langle \cdot, \cdot \rangle$ and the associated norm $\|x\|=\sqrt{\langle x,x\rangle}$.
In this paper, we consider the convex minimization problem
\begin{equation} \label{prob:ori_cvx_prob}
    {\min}\{ f(x) : x \in \mathcal{H}\},
\end{equation}
where $f \colon \mathcal{H}\to\mathbb{R}$ is continuously differentiable, convex, and has $L$-Lipschitz continuous gradient.
We assume that the set of minimizers $
\mathcal{S}:=\arg\min_{x\in\mathcal{H}} f(x)
$
is nonempty, and we denote by $x^*$ its minimum-norm element, that is, the projection of $0$ onto $\mathcal{S}$. When $\mathcal{S}$ is not a singleton, it is well known that, along the classical Tikhonov regularization path, the minimizers of $f+\frac{\varepsilon}{2}\|\cdot\|^2$ converge to $x^*$ as $\varepsilon\to 0^+$. However, preserving this selection principle in accelerated and inertial schemes, especially in the vanishing regularization regime, is far from automatic. At the continuous-time level, vanishing Tikhonov regularization can enforce convergence to $x^*$, and recent second-order models can in addition yield fast decay of the objective values. At the discrete level, transferring simultaneously acceleration and minimum-norm selection is more delicate: several schemes rely on implicit proximal steps, or introduce multiple regularization terms and intricate coefficient rules. This motivates the search for explicit, Nesterov-type discretizations with a vanishing Tikhonov term that preserve both fast value decay and strong convergence to $x^*$. In this work, we prove that a single vanishing Tikhonov term can enforce minimum-norm selection while preserving accelerated decay of the objective values, within a Nesterov-type inertial scheme.

Our aim is to analyze an accelerated first-order inertial gradient algorithm with vanishing Tikhonov regularization and to establish its strong convergence to the minimum-norm minimizer $x^*$.
Our approach is based on an explicit time discretization of the Tikhonov regularized inertial gradient system \eqref{trigs} \cite{Attouch:jde:22, Attouch:mmor:24}:
\begin{equation}\label{trigs}\tag{TRIGS}
    \ddot{x}(t) + \delta \sqrt{\varepsilon(t)}\,\dot{x}(t) + \nabla f(x(t)) + \varepsilon(t)\,x(t) = 0 .
\end{equation}
More precisely, our contributions can be summarized as follows:
\begin{itemize}
\item we derive an explicit Nesterov-type inertial gradient algorithm with a single vanishing Tikhonov regularization term;
\item we establish fast convergence rates for the objective values and strong convergence of the iterates to the minimum-norm minimizer under general parameter schedules, via a discrete Lyapunov analysis;
\item we corroborate the theory with numerical experiments on synthetic and real datasets, and we compare our method with \ref{algo:nag} and \ref{algo:nadtr}  in terms of solution quality and runtime.
\end{itemize}
From a theoretical viewpoint, this problem is challenging because it requires combining three competing mechanisms: inertia, acceleration, and a vanishing regularization term. While inertia and acceleration tend to amplify oscillations, the Tikhonov term acts as a selection bias toward $x^*$, and calibrating their interplay at the discrete level is delicate. In particular, ensuring both fast decay of the objective values and strong convergence of the iterates requires a careful Lyapunov-based analysis and a precise tuning of the parameters.

\subsection{Related works}

\subsubsection{Continuous-time dynamics}

First-order dynamical systems with a Tikhonov regularization term whose coefficient vanishes asymptotically have been extensively studied (see, e.g., \cite{Attouch:sjo:96, Attouch:jde:96, Cominetti:jde:08, Attouch:jde:10, Bot:ana:21}). In parallel, the second-order setting has also been addressed. The first studies \cite{Attouch:jde:02, Attouch:jde:17} by Attouch and Czarnecki considered the coupling of a second-order dynamical system with Tikhonov regularization,
\begin{align} \label{system:fixed-damping}
    \ddot{x}(t) + \gamma \dot{x}(t) + \nabla f(x(t)) + \varepsilon(t) x(t) = 0.
\end{align}
If the Tikhonov regularization term $\varepsilon$ satisfies the slow parametrization condition
$\int_0^{+\infty} \varepsilon(t)\,dt = +\infty$,
then the unique trajectory of \eqref{system:fixed-damping} strongly converges to the minimum-norm element of $\arg\min_{\mathcal{H}} f$. In contrast, without the Tikhonov regularization term, the convergence remains only weak. Heuristically, including the Tikhonov regularization term introduces a time-dependent strong convexity effect that diminishes over time.

To enhance the convergence rate, Attouch, Chbani, and Riahi \cite{Attouch:jmaa:18} explored the following system with asymptotically vanishing damping:
\begin{align} \label{system:avd-tikhonov}
    \ddot{x}(t) + \dfrac{\alpha}{t}\dot{x}(t) + \nabla f(x(t)) + \varepsilon(t)x(t) = 0.
\end{align}
The system \eqref{system:avd-tikhonov} can be considered as a Tikhonov regularization of the low-resolution generalized ODE modeling the Nesterov accelerated gradient (NAG) method \cite{Su:jmlr:16}:
\begin{align*}
    \ddot{x}(t) + \dfrac{\alpha}{t} \dot{x}(t) + \nabla f(x(t)) = 0.
\end{align*}
As a natural extension of \eqref{system:avd-tikhonov} incorporating Hessian-driven damping, the authors in \cite{Bot:mp:19} investigated a Tikhonov regularization of the high-resolution ODE \cite{Shi:mp:18} defined as follows:
\begin{align} \label{system:avd-hessian-tikhonov}
    \ddot{x}(t) + \dfrac{\alpha}{t}\dot{x}(t) + \beta \nabla^2 f(x(t)) \dot{x}(t) + \nabla f(x(t)) + \varepsilon(t) x(t) = 0.
\end{align}
Both second-order dynamical systems \eqref{system:avd-tikhonov} and \eqref{system:avd-hessian-tikhonov} provide fast convergence rates for the objective values. However, the available results on the convergence of the trajectories toward the minimum-norm solution remain limited, as they only guarantee that
$\liminf_{t \to +\infty} \|x(t) - x^*\| = 0.$

The works \cite{Attouch:mmor:24, Attouch:jde:22} are closely related to ours. The authors introduced new inertial dynamics designed to generate trajectories that converge rapidly to the minimizer of $f$ with minimum norm. They started from the Polyak heavy ball system with friction dynamics for strongly convex functions:
\begin{align} \label{system:polyak}
     \ddot{x}(t) + 2\sqrt{\mu}\dot{x}(t) + \nabla f(x(t)) = 0.
\end{align}
Their main idea for handling general convex functions is based on a Tikhonov regularization technique, which naturally leads to a dynamical system governed by the gradient of a strongly convex function
\begin{align*}
    \varphi_t \colon \mathcal{H} \to \mathbb{R}, \quad \varphi_t(x) := f(x) + \dfrac{\varepsilon(t)}{2}\|x\|^2.
\end{align*}
Replacing $\mu$, $\nabla f$, and the coefficient $2$ in \eqref{system:polyak} by $\varepsilon(t)$, $\nabla \varphi_t$, and a positive parameter $\delta$, respectively, amounts to the aforementioned dynamical system \eqref{trigs}.

The system \eqref{trigs} is the first second-order dynamical system that guarantees both fast convergence of the values and strong convergence of the trajectories to the minimum-norm solution. The authors in \cite{Attouch:mmor:24, Attouch:jde:22} obtained these results in the case of a general $\varepsilon(\cdot)$ and gave an in-depth analysis in the particular case $\varepsilon(t) = c/t^2$. The main results in \cite{Attouch:mmor:24, Attouch:jde:22} are extended and improved in \cite{Laszlo:jde:23} with the study of the following dynamical system:
\begin{align} \label{system:correlated-tikhonov}
    \ddot{x}(t) + \dfrac{\alpha}{t^q} \dot{x}(t) + \nabla f(x(t))  + \dfrac{a}{t^p} x(t) = 0.
\end{align}
L\'{a}szl\'{o} showed that the asymptotically vanishing damping coefficient $\frac{\alpha}{t^q}$ is correlated with the Tikhonov regularization coefficient $\frac{a}{t^p}$ \cite{Laszlo:jde:23}. He noted that his results significantly extend those in \cite{Attouch:jde:22} which only considered the case $q = \frac{p}{2}$. Moreover, his analysis indicated that the choice $p = 2q$ in the dynamical system \eqref{system:correlated-tikhonov}, which coincides with \eqref{trigs}, is not necessarily optimal.

By introducing geometric damping driven by the Hessian of the function $f$ into the system \eqref{trigs}, the work \cite{Attouch:amo:23} presented the Tikhonov regularized inertial system with Hessian-driven damping \eqref{trish} as follows:
\begin{equation} \label{trish} \tag{TRISH}
    \ddot{x}(t) + \delta\sqrt{\varepsilon(t)} \dot{x}(t) + \beta \nabla^2 f(x(t)) \dot{x}(t) + \nabla f(x(t)) + \varepsilon(t)x(t)=0.
\end{equation}
The role of the Hessian-driven damping has been thoroughly analyzed in \cite{Attouch:amo:23}. This new dynamical system (TRISH) greatly reduced the oscillations of the trajectory $x(t)$. Further recent studies on Tikhonov regularized second-order dynamical systems are actively being investigated: for example, Tikhonov regularized dynamical systems with implicit Hessian-driven damping \cite{Laszlo:coap:25}, a perturbed heavy ball system with vanishing damping \cite{Alecsa:sjo:21}, a Newton-like inertial dynamical system with a Tikhonov regularization term governed by a general maximally monotone operator \cite{Bot:jmaa:24}, the Tikhonov regularized second-order dynamical system for solving a monotone equation with single-valued, monotone and continuous operator \cite {Csetnek:arxiv:24}, the second-order dynamics combining viscous and Hessian-driven damping with a Tikhonov regularization term and the incorporation of the Moreau envelope of the nonsmooth objective function \cite{Karapetyants:coap:24}.

\subsubsection{Discrete algorithms}

Thanks to the development of second-order dynamical systems, various inertial algorithms with Tikhonov regularization have been recently developed by discretizing these systems.
One of the first Tikhonov regularized inertial algorithms is the inertial forward-backward algorithm with vanishing Tikhonov regularization \eqref{algo:ifb} \cite{Attouch:jota:18}:
\begin{align*} \label{algo:ifb} \tag{$\mathrm{(IFB)_{Tikh}}$}
    \begin{cases}
        y_k = x_k + \alpha_k(x_k - x_{k-1}), \\
        x_{k+1} = \mathrm{prox}_{s g} (y_k - s \nabla f(y_k) - s\varepsilon_k y_k),
    \end{cases}
\end{align*}
for solving the structured convex minimization problem
\begin{equation} \label{prob:composite_fg}
    \min\{f(x) + g(x) : x \in \mathcal{H}\}
\end{equation}
where $f \colon \mathcal{H} \to \mathbb{R}$ is a continuously differentiable convex function whose gradient is Lipschitz continuous and $g \colon \mathcal{H} \to \mathbb{R} \cup \{+\infty\}$ is a proper lower semi-continuous convex function.
The original idea of \ref{algo:ifb} is linked to second-order dynamical systems by performing a time discretization, implicit with respect to the nonsmooth function $g$ and explicit with respect to the smooth function $f$, of the second-order differential inclusion
\begin{align*}
    \ddot{x}(t) + \gamma(t) \dot{x}(t) + \nabla f(x(t))  + \partial g(x(t)) \ni h(t).
\end{align*}
The damping parameter $\gamma(t)$ approaches zero. The second member $h(t)$ plays the role of a perturbation term, and Algorithm \ref{algo:ifb} considers its discretization term $h_k = -\varepsilon_k y_k$ where a sequence of positive numbers $(\varepsilon_k)_k$ tends to zero. Attouch et al. \cite{Attouch:jota:18} showed that when the Tikhonov regularization coefficient $\varepsilon_k$ goes sufficiently fast to zero, Algorithm \ref{algo:ifb} achieves an accelerated convergence rate of the objective function values. On the other hand, under the slow decay of the Tikhonov regularization term, the authors obtained a partial strong convergence of the iterates $(x_k)_k$ toward the minimum-norm solution, i.e., $\liminf_{k \to +\infty} \|x_k -  x^*\| = 0$.

In \cite{Attouch:mmor:24}, the authors developed an Inertial Proximal Algorithm with Tikhonov REgularization \eqref{algo:ipatre} by applying implicit temporal discretization of \eqref{trigs} in the special case $\varepsilon_k = \frac{c}{k^2}$. The algorithm is given by
\begin{align} \tag{IPATRE} \label{algo:ipatre}
    \begin{cases}
        y_k = x_k + \left(1 - \dfrac{\alpha}{k} \right)(x_k - x_{k-1}), \\
        x_{k+1} = \mathrm{prox}_{f}\left( y_k - \dfrac{c}{k^2}x_k \right),
    \end{cases}
\end{align}
where $\alpha > 3$ and $c>0$. To the best of our knowledge, this is the first inertial gradient algorithm which ensures both the fast convergence rate of the values and a partial strong convergence to the minimum norm solution: $\liminf_{k\to +\infty} \|x_k - x^*\| = 0$. The authors extended Algorithm \ref{algo:ipatre} to the nonsmooth case with a proper lower semicontinuous and convex function $f$ by relying on the Moreau envelope $f_\lambda : \mathcal{H} \to \mathbb{R}$ (where $\lambda$ is a positive real parameter) (see more details in Section 5.3 in \cite{Attouch:mmor:24}).

L{\'a}szl{\'o} in \cite{Laszlo:cnsns:25} proposed an improved version of \ref{algo:ipatre} defined as follows:
\begin{align}
    \label{algo:piatr} \tag{PIATR}
    \begin{cases}
        y_k = x_k + \left( 1 - \dfrac{\alpha}{k^q}\right)(x_k - x_{k-1}), \\
        x_{k+1} = \mathrm{prox}_{\lambda k^{\delta}f}\left(y_k - \dfrac{c}{k^p}x_k \right),
    \end{cases}
\end{align}
where $\alpha, c, \lambda >0$, $0< q \leq 1$, $p > 1$, $\delta \in \mathbb{R}$. The weak convergence properties of the sequence generated by Algorithm~\ref{algo:piatr} and fast convergence of the objective function values were analyzed for the case $q+1 \leq p$ and $\delta \geq 0$. On the other hand, for $0 < q < 1$, $1 < p \leq q+1$, $\delta \leq 0$, further if $\delta = 0$, then $\lambda \in (0,1)$, the author provided the convergence rates for the values:
$f(x_k) - \underset{\mathcal{H}}{\min} f = \mathcal{O}\left( \frac{1}{k^{p+\delta}} \right)$ if $p < q+1$ or $\mathcal{O}\left( \frac{\ln{k}}{k^{p+\delta}} \right)$ if $p = q+1$ as $k \to +\infty$. Moreover, under the assumptions that $1 < p < q+1 < 2$, $p-q-1<\delta \leq 0$, and if $\delta = 0$, also suppose that $\lambda \in (0,1)$, he showed a partial strong convergence of the sequences generated by \ref{algo:piatr} to a minimum-norm solution: $\liminf_{k \to+\infty} \|x_k - x^*\| = 0$. For the special case $1<p<q+1<2$, $\delta = 0$ and $\lambda = 1$, he obtained both the fast convergence rates of the values (more precisely, $\mathcal{O}\left(k^{2p-4q-2}\right)$ if $(4q+2)/3 \leq p < q+1$ and $\mathcal{O}\left(k^{-p}\right)$ otherwise) and the full strong convergence, i.e., $\lim_{k \to+\infty} x_k = x^*$.

For solving the optimization problem~\eqref{prob:ori_cvx_prob}, Karapetyants and L{\'a}szl{\'o} \cite{Karapetyants:amo:24} introduced an inertial gradient algorithm with Tikhonov regularization term \eqref{algo:nadtr}, which is of the following form:
\begin{align}
    \begin{cases} \label{algo:nadtr} \tag{NADTR}
        y_k = x_k + b_{k-1}(x_k - x_{k-1}) - c_k x_k, \\
        x_{k+1} = y_k - s ( \nabla f(y_k) + \varepsilon_k y_k),
    \end{cases}
\end{align}
where the step-size $s$ satisfies $0 < s < \frac{1}{L}$.
By taking $b_{k-1} := 1 - \dfrac{\alpha}{k}$ for $k \geq 1$ and $c_k \equiv 0$, $\varepsilon_k \equiv 0$ (i.e., without Tikhonov regularization terms), Algorithm \ref{algo:nadtr} coincides with the famous Nesterov accelerated gradient \eqref{algo:nag} algorithm:
\begin{align} \label{algo:nag} \tag{NAG}
    \begin{cases}
        y_k = x_k + \left( 1 - \dfrac{\alpha}{k} \right)(x_k - x_{k-1}), \\
        x_{k+1} = y_k - s \nabla f(y_k).
    \end{cases}
\end{align}
The last two terms $c_k x_k$ and $\varepsilon_k y_k$ in \ref{algo:nadtr} serve as Tikhonov regularization terms. These terms may ensure the strong convergence of the generated sequence $(x_k)_k$ to the minimum-norm solution $x^*$ of \eqref{prob:ori_cvx_prob}, which is not valid for the algorithm~\ref{algo:nag}. After choosing appropriate coefficients $(b_k)_k$, $(c_k)_k$ and $(\varepsilon_k)_k$ to guarantee strong convergence, Algorithm \ref{algo:nadtr} is detailed as follows \cite{Karapetyants:amo:24}:
\begin{align*}
    \begin{cases}
        y_k = x_k ~ \text{if } k \in \{1, (cs)^\frac{1}{p}, (cs)^\frac{1}{p} + 1\},  \\
        y_k = x_k + \frac{k^p (a(k-1)^q - s)(a((k-1)^p -cs)^2(k-1)^q - 2s(k-1)^{2p})}{a^2(k-1)^{p+q}k^q ((k-1)^p -cs)(k^p -cs)}(x_k - x_{k-1}) \\
        \qquad - \frac{2s^2k^p ((k-1)^pk^p- c(k-1)^p - ac(k-1)^q k^p+ ac(k-1)^{q+p})}{a^2(k-1)^q k^q((k-1)^p-cs)(k^p-cs)^2}x_k,\text{otherwise,} \\
        x_{k+1} = y_k - s\nabla f(y_k) - \dfrac{cs}{k^p} y_k.
    \end{cases}
\end{align*}
Here $a$ and $c$ are positive real numbers.
This algorithm accelerates the convergence of the function values $f(x_k) - \min_{\mathcal{H}} f = \mathcal{O}(k^{-p})$ as $k \to +\infty$ and assures the strong convergence of the generated sequence $(x_k)_k$ to the minimum-norm minimizer $x^*$. It is essential to require both Tikhonov regularization terms in Algorithm~\ref{algo:nadtr} in order to ensure the strong convergence of $(x_k)_k$. Following the same idea of two Tikhonov regularization terms, L{\'a}szl{\'o} has recently proposed an inertial proximal-gradient algorithm associated with the optimization problem~\eqref{prob:composite_fg}, named TIREPROG (see \cite{Laszlo:oms:25}). Under appropriate parameters, he showed the strong convergence of the sequence $(x_k)_k$ generated by TIREPROG to the minimum-norm minimizer of $f + g$, and provided not only a fast convergence to zero of the objective function values, but also for the discrete velocity and the sub-gradient of the objective function.

Similarly to Algorithm \ref{algo:ipatre}, Attouch et al. \cite{Attouch:amo:23} proposed a first-order inertial proximal algorithm with Tikhonov regularization and Hessian-driven damping (IPATH) by using an implicit temporal discretization of the continuous dynamic \eqref{trish}. In the nonsmooth case of $f$, the authors replace $f$ by the Moreau-Yosida regularization $f_{\theta}$ and leverage the continuous differentiability of $f_{\theta}$ for which the dynamic \eqref{trish} is applied. Furthermore, \cite{Attouch:amo:23} also presented two inertial gradient algorithms with Tikhonov regularization and Hessian damping (IGATH) associated with the dynamical system \eqref{trish}, denoted as (IGATH)-1 and (IGATH)-2. Specifically, the algorithm (IGATH)-1 relies on an explicit discretization of \eqref{trish} relevant to Polyak's heavy ball in the strongly convex setting, while the other (IGATH)-2 is closely related to the method \ref{algo:nag} where a similar discretization is considered but the gradient $\nabla f$ is evaluated at an auxiliary point $y_k$ determined according to the Nesterov scheme. The algorithm in our work can be viewed as a particular instance of (IGATH)-2 corresponding to the dynamic \eqref{trigs}. However, a rigorous analysis of the convergence properties of both IGATH algorithms remains an open problem. In addition, a study of their numerical performance has not yet been explored, and such an investigation constitutes a relevant direction for our work. As an extension, the authors in \cite{Attouch:amo:23} gave an introduction and preliminary numerical illustrations of the forward-backward algorithm IPGATH (Inertial Proximal Gradient Algorithm with Tikhonov regularization and Hessian damping) for solving the composite problem \eqref{prob:composite_fg}. The theoretical study of IPGATH remains an interesting subject for further research.

\subsection{Motivation and main contributions}

Our work is motivated by the development of explicit accelerated inertial algorithms with vanishing Tikhonov regularization that ensure strong convergence to the minimum-norm minimizer.
More precisely, we consider an explicit time discretization of the dynamical system \eqref{trigs}, in the spirit of Nesterov-type schemes, and we derive the Tikhonov Regularized Inertial Gradient Algorithm \eqref{algo:triga}, presented in Section \ref{sec:algo}.

With this explicit discretization, our method differs from the inertial proximal algorithms \ref{algo:ipatre} and \ref{algo:piatr}, which rely on implicit discretizations and proximal evaluations of the objective function.
Compared to \ref{algo:nadtr}, our algorithm uses a single Tikhonov regularization term and has a simpler update structure.

Our contributions focus on a rigorous convergence analysis of \ref{algo:triga}. The main results, stated in Theorems \ref{theorem:grad-algo}, \ref{theorem:conv_rate_grad_algo}, and \ref{theorem:conv_rate_grad_algo_k^2}, mirror at the discrete level the qualitative behavior established for \eqref{trigs} in \cite{Attouch:jde:22}. Our analysis relies on a discrete Lyapunov approach and yields convergence under general assumptions on the sequence $(\varepsilon_k)_k$. We provide explicit and consistent parameter conditions ensuring these assumptions, independently of the particular choice of $(\varepsilon_k)_k$, and representative choices are summarized in Table~\ref{table:condition_necessary_parameters}.

For polynomial schedules $\varepsilon_k = k^{-p}$ with $p \in (0,2)$, we prove the value estimate
\[
f(x_k) - \min_{\mathcal{H}} f = \mathcal{O}\Big(\frac{1}{k^{p}}\Big)\quad \text{as } k\to+\infty,
\]
together with strong convergence of $(x_k)_k$ to the minimum-norm minimizer $x^*$.
Moreover, the discrete velocity satisfies
\[
\|x_k - x_{k-1}\| = \mathcal{O}\Big(\frac{1}{k^{\frac{2+p}{4}}}\Big)\quad \text{as } k\to+\infty.
\]
In the critical case $p=2$, namely $\varepsilon_k = c/k^2$, we obtain Nesterov-type rates for both values and velocities,
\[
f(x_k) - \min_{\mathcal{H}} f = \mathcal{O}\Big(\frac{1}{k^{2}}\Big),
\qquad
\|x_k - x_{k-1}\| = \mathcal{O}\Big(\frac{1}{k}\Big),
\quad \text{as } k\to+\infty,
\]
while strong convergence of $(x_k)_k$ to $x^*$ is not guaranteed by our analysis.

On the numerical side, we assess the performance of Algorithm~\ref{algo:triga} on a range of experiments, from synthetic examples to moderately large-scale real datasets with up to $32{,}000$ samples. We compare \ref{algo:triga} with \ref{algo:nag} and \ref{algo:nadtr} in terms of solution quality and runtime.

\subsection{Outline of the paper}

The rest of this paper is organized as follows. In Section~\ref{sec:algo_and_general_case}, we present Algorithm \ref{algo:triga} and analyze the asymptotic convergence of the generated iterates via a discrete Lyapunov approach under general assumptions on the Tikhonov sequence $(\varepsilon_k)_k$. Section~\ref{sec:particular_case} focuses on polynomial schedules $\varepsilon_k = k^{-p}$ with $p\in(0,2)$, where we establish fast value convergence together with strong convergence to the minimum-norm minimizer. We also discuss the critical regime $p=2$, in particular $\varepsilon_k = c/k^2$, for which fast rates are obtained while strong convergence is not guaranteed by our analysis. Section~\ref{sec:numerical} reports numerical experiments illustrating the practical behavior of the proposed method and comparing it with existing algorithms. Finally, Section~\ref{sec:conclusion} concludes the paper and outlines directions for future work.

\section{Tikhonov-regularized inertial gradient algorithm and general convergence analysis}
\label{sec:algo_and_general_case}

\subsection{Tikhonov Regularized Inertial Gradient Algorithm (TRIGA)}
\label{sec:algo}

Throughout the paper, we assume that the objective function $f$ and the sequence of Tikhonov regularization parameters $(\varepsilon_k)_{k}$ satisfy the following assumptions \eqref{assumption}:
\begin{align} \label{assumption} \tag{$\mathcal{A}$}
    \begin{cases}
        f \colon \mathcal{H} \to \mathbb{R} \text{ is convex and differentiable, and } \nabla f \text{ is } L\text{-Lipschitz continuous}, \\
        \mathcal{S} = \underset{x \in \mathcal{H}}{\arg\min} f \neq \emptyset \text{. We denote by } x^* \text{ the minimum-norm element of } \mathcal{S}, \\
        (\varepsilon_k)_{k} \text{ is a nonincreasing positive sequence and } \lim\limits_{k \to +\infty} \varepsilon_k = 0.
    \end{cases}
\end{align}

In this subsection, we derive a first-order inertial gradient algorithm by applying an explicit discretization in time to the second-order dynamical system \eqref{trigs}. For ease of reading, let us recall the system:
\begin{equation*}
    \ddot{x}(t) + \delta \sqrt{\varepsilon(t)} \dot{x}(t) + \nabla f(x(t)) + \varepsilon(t) x(t) = 0. \tag{TRIGS}
\end{equation*}
We then show that the resulting discrete-time scheme enjoys convergence properties that parallel those established for the continuous dynamics \eqref{trigs}. 

Specifically, consider the following explicit temporal discretization with a fixed step $h > 0$ and define the step-size $s := h^2$. For all $k \geq 1$, the discrete-time scheme is given by:
\begin{equation*} 
    \dfrac{x_{k+1} - 2x_k + x_{k-1}}{s} + \delta \sqrt{\varepsilon_k} \dfrac{x_k - x_{k-1}}{\sqrt{s}} + \nabla f(y_k) + \varepsilon_k y_k = 0,
\end{equation*} 
where the auxiliary sequence $(y_k)_k$ is designed in the spirit of Nesterov's accelerated scheme \cite{Nesterov:smd:83}.
This is equivalent to
\begin{equation*}
    x_{k+1} - x_k - (1 - \delta \sqrt{s\varepsilon_k}) (x_k - x_{k-1}) + s \big( \nabla f(y_k) + \varepsilon_k y_k \big) = 0.
\end{equation*}
By choosing 
\begin{equation*}
    y_k := x_k + (1-\delta\sqrt{s\varepsilon_k})(x_k - x_{k-1}),
\end{equation*}
we obtain the following Tikhonov Regularized Inertial Gradient Algorithm \eqref{algo:triga}.

\begin{algorithm}
\caption{Tikhonov Regularized Inertial Gradient Algorithm \eqref{algo:triga}}
\begin{algorithmic}[1]
\State \textbf{Initialization}: $x_0, x_1 \in \mathcal{H}$ and $k \leftarrow 1$.
\While{stopping criteria are not satisfied}
\State Compute $y_k$ and $x_{k+1}$ as follows:
\begin{equation} \label{algo:triga} \tag{TRIGA}
        \begin{cases}
            y_k =  x_k + (1 - \delta \sqrt{s\varepsilon_k})(x_k-x_{k-1}), \\
            x_{k+1} =  y_k - s \big(\nabla f(y_k) + \varepsilon_k y_k\big).
        \end{cases}
\end{equation}
\State $k \leftarrow k + 1$.
\EndWhile
\end{algorithmic}
\end{algorithm}
\subsection{Preparatory results}
For $k \geq 1$, let us define the function $\varphi_k \colon \mathcal{H} \to \mathbb{R}$ by 
\begin{align} \label{def:function_varphi_k}
    \varphi_k(x) := f(x) + \dfrac{\varepsilon_k}{2} \|x\|^2,
\end{align}
and denote
\[
x_{\varepsilon_k} := \underset{x \in \mathcal{H}}{\arg\min} ~ \varphi_k(x),
\]
which is the unique minimizer of the $\varepsilon_k$-strongly convex function $\varphi_k$. The sequence $(x_{\varepsilon_k})_k$ is called the discrete viscosity curve. We recall some useful results of Tikhonov regularization for $(x_{\varepsilon_k})_k$ in the following lemmas, which will be used in the next sections.

\begin{lemma}[\cite{Attouch:sjo:96,Attouch:amo:23}]
\label{lemma:visco-curve}
    The sequence $(x_{\varepsilon_k})_{k \geq 1}$ satisfies the following classical properties of Tikhonov regularization. \\
    \noindent
    {\rm (i)}\; $\|x_{\varepsilon_k}\| \le \|x^*\|$ for all $k\ge 1$,
    \qquad
    {\rm (ii)}\; $\lim\limits_{k\to +\infty}\|x_{\varepsilon_k}-x^*\|=0$.
\end{lemma}

\begin{lemma}[Lemma 4, \cite{Attouch:amo:23}]
\label{lemma:difference-tikhonov-value}
    For all $k \geq 1$, the following relations are satisfied:
    \def\labelenumi{\rm (\roman{enumi})}
    \def\theenumi{\roman{enumi}}
    \begin{enumerate}        
        \item $\varphi_k(x_{\varepsilon_k}) - \varphi_{k+1}(x_{\varepsilon_{k+1}}) \leq \dfrac{\varepsilon_k - \varepsilon_{k+1}}{2} \|x_{\varepsilon_{k+1}}\|^2$,
        \item $\| x_{\varepsilon_{k+1}} - x_{\varepsilon_k}\|^2 \leq \dfrac{\varepsilon_k - \varepsilon_{k+1}}{\varepsilon_{k}} \langle x_{\varepsilon_{k+1}}, x_{\varepsilon_{k+1}} - x_{\varepsilon_k} \rangle$ and therefore
        \begin{align*}
            \| x_{\varepsilon_{k+1}} - x_{\varepsilon_k}\| \leq \dfrac{\varepsilon_k - \varepsilon_{k+1}}{\varepsilon_{k}} \|x_{\varepsilon_{k+1}}\|.
        \end{align*}
    \end{enumerate}
\end{lemma}

\subsection{Convergence analysis for the general sequence \texorpdfstring{$(\varepsilon_k)_k$}{TEXT}}

To proceed with the Lyapunov analysis, we first introduce the energy sequence $(E_k)_k$ defined by
\begin{equation} \label{def:energy-function-grad-algo}
    E_k := E_{\mathrm{pot},k} + E_{\mathrm{mix},k} + E_{\mathrm{kin},k},
\end{equation}
where $E_{\mathrm{pot},k}$, $E_{\mathrm{mix},k}$, and $E_{\mathrm{kin},k}$ are respectively the potential, mixed, and kinetic components,
\begin{align*}
    & E_{\mathrm{pot},k} = \alpha_k \big(\varphi_k(x_k) - \varphi_k(x_{\varepsilon_k})\big), \; E_{\mathrm{mix},k} = \dfrac{1}{2}\big\| \tau_k (x_k - x_{\varepsilon_{k-1}}) + (x_k - x_{k-1})\big\|^2, \\
    & E_{\mathrm{kin},k} = \dfrac{\beta_k}{2}\|x_k - x_{k-1}\|^2.
\end{align*}
Here the positive parameters $\alpha_k$, $\beta_k$, and $\tau_k$ will be specified later in the proof of Theorem~\ref{theorem:grad-algo}.

The next lemma provides elementary estimates relating the function values and the iterates to the energy level $E_k$.

\begin{lemma} \label{lemma:convergence_rate_grad_algo}
    Let $(x_k)_k$ denote the sequence generated by Algorithm~\ref{algo:triga}, and let $(E_k)_k$ be the associated energy sequence defined in \eqref{def:energy-function-grad-algo}. Then, for all $k \geq 1$, the following inequalities hold:
   \\
    \noindent
    {\rm (i)}\; $f(x_k) - \min_{\mathcal{H}} f \le \dfrac{E_k}{\alpha_k} + \dfrac{\varepsilon_k}{2}\|x^*\|^2,$
    \qquad
    {\rm (ii)}\; $\|x_k - x_{\varepsilon_k}\|^2 \le \dfrac{2E_k}{\alpha_k \varepsilon_k},$
    \\
    {\rm (iii)}\; $\|x_k - x_{k-1}\|^2 \le \dfrac{2E_k}{\beta_k}.$

\noindent In particular, the sequence $(x_k)_k$ converges strongly to $x^*$ provided that
    \[
    \lim_{k\to+\infty}\frac{E_k}{\alpha_k\varepsilon_k}=0.
    \]
\end{lemma}

\begin{proof}
    The proof follows the same lines as \cite[Lemma 3]{Attouch:amo:23}. Using the definition of $\varphi_k$ in \eqref{def:function_varphi_k}, we obtain
    \[
        f(x_k) - \underset{\mathcal{H}}{\min} f
        \leq \varphi_k(x_k) - \varphi_k(x_{\varepsilon_k}) + \dfrac{\varepsilon_k}{2}\|x^*\|^2.
    \]
    By definition of $E_{\mathrm{pot},k}$, we have
    \[
    \varphi_k(x_k) - \varphi_k(x_{\varepsilon_k})
    =\frac{E_{\mathrm{pot},k}}{\alpha_k}
    \leq \frac{E_k}{\alpha_k},
    \]
    since $E_k-E_{\mathrm{pot},k}=E_{\mathrm{mix},k}+E_{\mathrm{kin},k}\ge 0$.
    This proves (i).

    By the $\varepsilon_k$-strong convexity of $\varphi_k$, we also have
    \[
    \varphi_k(x_k) - \varphi_k(x_{\varepsilon_k}) \geq \dfrac{\varepsilon_k}{2}\|x_k-x_{\varepsilon_k}\|^2.
    \]
    Combining this with
    $\varphi_k(x_k) - \varphi_k(x_{\varepsilon_k})=\dfrac{E_{\mathrm{pot},k}}{\alpha_k}\le \dfrac{E_k}{\alpha_k}$ yields (ii).

    Finally, by definition of $E_{\mathrm{kin},k}$,
    \[
    \|x_k-x_{k-1}\|^2=\frac{2E_{\mathrm{kin},k}}{\beta_k}\le \frac{2E_k}{\beta_k},
    \]
   which proves (iii).
\end{proof}

We now impose quantitative decay conditions on the Tikhonov parameters $(\varepsilon_k)_k$, tailored to the Lyapunov analysis.
\begin{definition}
    Given $\delta>0$, we say that the sequence $(\varepsilon_k)_k$ satisfies the conditions $(\mathcal{K}_0)$ if there exist $k_0 \in \mathbb{N}$ and positive parameters $a, b, q, \lambda$ such that $\lambda < \delta$ and, for all $k \geq k_0$,
    \begin{align} \label{condition:tikhonov-sequence}\tag{$\mathcal{K}_0$}
        \begin{cases}
            (i) & (1 + q) \left(\dfrac{1}{\sqrt{s\varepsilon_{k+1}}} -\dfrac{1}{\sqrt{s\varepsilon_{k}}}  + \dfrac{\delta}{1 - \delta \sqrt{s \varepsilon_k}} - \lambda \right) - \lambda\sqrt{\dfrac{\varepsilon_{k+1}}{\varepsilon_{k}}} \leq 0,\\
            (ii) &  \dfrac{1-b}{b}\lambda^2  + \dfrac{\delta \lambda}{1 - \delta \sqrt{s \varepsilon_k}} - 1 \leq 0, \\
            (iii) &   \dfrac{(1+a(1-q))\lambda}{a} -\delta   + (1+q)\left(\dfrac{1}{\sqrt{s\varepsilon_{k+1}}} -\dfrac{1}{\sqrt{s\varepsilon_{k}}} \right) +\dfrac{q\delta^2 \sqrt{s \varepsilon_{k+1}}}{1 - \delta \sqrt{s \varepsilon_k}} \leq 0, \\
            (iv) & Ls\left(1+\lambda\sqrt{s\varepsilon_{k+1}}\right) + q((L + \varepsilon_k)s-1) \leq 0.  \\
        \end{cases}
    \end{align}
\end{definition}

For readability, introduce the forward differences (for $k\ge 0$)
\[
\dot{x}_k := x_{k+1}-x_k,\qquad \dot{x}_{\varepsilon_k}:=x_{\varepsilon_{k+1}}-x_{\varepsilon_k},\qquad \dot{\tau}_k:=\tau_{k+1}-\tau_k.
\]
Then the update rules of Algorithm~\ref{algo:triga} can be rewritten in the phase-space form: for $k\ge 1$,
\begin{equation} \label{phase-space}
    \begin{cases}
        \dot{x}_k - \dot{x}_{k-1} = -\delta \sqrt{s\varepsilon_k} \dot{x}_{k-1} - s\nabla \varphi_k(y_k),\\
        y_k - x_k = (1-\delta\sqrt{s\varepsilon_k}) \dot{x}_{k-1} = \dot{x}_k + s\nabla \varphi_k(y_k).
    \end{cases}
\end{equation}

\begin{theorem} \label{theorem:grad-algo}
    Suppose $f \colon \mathcal{H} \to \mathbb{R}$ is convex and differentiable, with $L$-Lipschitz continuous gradient $\nabla f$. Let $(x_k)_k$ be generated by Algorithm~\ref{algo:triga} and assume that $(\varepsilon_k)_k$ satisfies \eqref{condition:tikhonov-sequence} for all $k \geq k_0$. Then, for all $k \geq k_0$,
    \begin{align} \label{theorem:grad-algo_main_inequa}
        E_{k+1} - E_k + \mu_{k+1} E_{k+1} \leq \dfrac{\theta_k}{2} \| x^* \|^2, 
    \end{align}
    where $(\mu_k)_k$ and $(\theta_k)_k$ are given by
    \begin{align*}
        \mu_{k+1} & = \dfrac{\sqrt{\varepsilon_k} - \sqrt{\varepsilon_{k+1}}}{\sqrt{\varepsilon_{k+1}}} + \left( \dfrac{\delta}{1 - \delta\sqrt{s \varepsilon_k}} - \lambda \right)\sqrt{s \varepsilon_k},\\
        \theta_k & = (a+b)\lambda \sqrt{s\varepsilon_k} \dfrac{(\dot{\varepsilon}_{k-1})^2}{\varepsilon_{k-1}^2} - \dfrac{(1+q)s \dot{\varepsilon}_k}{(1 - \delta \sqrt{s \varepsilon_k})(1 -s\varepsilon_k)} \left(1 + \mu_{k+1}\right). 
    \end{align*}
    Define $\displaystyle \gamma_k = \prod_{j=k_0}^k (1 + \mu_j)$ for $k \geq k_0$. Then
    \begin{align*}
        E_{k+1} \leq \dfrac{\gamma_{k_0}}{\gamma_{k+1}}E_{k_0} + \dfrac{1}{2\gamma_{k+1}} \left( \sum_{j = k_0}^k \gamma_{j} \theta_j \right) \|x^*\|^2.
    \end{align*}
\end{theorem}

\begin{proof}
First, we estimate the mixed energy component in the left member of \eqref{theorem:grad-algo_main_inequa}:
\begin{align*}
    E_{\mathrm{mix},k+1} - E_{\mathrm{mix},k} + \mu_{k+1} E_{\mathrm{mix},k+1}.
\end{align*}
Let us denote $p_k := \tau_k (x_k - x_{\varepsilon_{k-1}}) + (x_k - x_{k-1})$. From the equalities \eqref{phase-space}, we obtain
\begin{align*}
    p_{k+1} - p_k &= (\tau_{k+1} ( x_{k+1} - x_{\varepsilon_k}) + \dot{x}_k) - (\tau_k ( x_k - x_{\varepsilon_{k-1}}) + \dot{x}_{k-1}) \\
    & = \dot{\tau}_k(x_{k+1} - x_{\varepsilon_k}) + \tau_k ( \dot{x}_{k} - \dot{x}_{\varepsilon_{k-1}}) + (\dot{x}_k - \dot{x}_{k-1}) \\
    & = \dot{\tau}_k(x_{k+1} - x_{\varepsilon_k}) + \tau_k \dot{x}_k - \delta\sqrt{s\varepsilon_k} \dot{x}_{k-1} - \tau_k \dot{x}_{\varepsilon_{k-1}} - s\nabla \varphi_k(y_k).
\end{align*}
Using the inequality $\dfrac{1}{2}\|x\|^2 - \dfrac{1}{2}\|y\|^2 = \langle x - y, x \rangle - \dfrac{1}{2}\|x - y\|^2 \leq \langle x - y, x \rangle$, we have
\begin{align*}
    & E_{\mathrm{mix},k+1} - E_{\mathrm{mix},k} = \dfrac{1}{2}\|p_{k+1}\|^2 - \dfrac{1}{2} \|p_k\|^2 \leq \langle p_{k+1} - p_k, p_{k+1} \rangle \\
    & = \langle \dot{\tau}_k(x_{k+1} - x_{\varepsilon_k}) + \tau_k \dot{x}_k - \delta\sqrt{s\varepsilon_k} \dot{x}_{k-1} - \tau_k \dot{x}_{\varepsilon_{k-1}} - s\nabla \varphi_k(y_k), \tau_{k+1}(x_{k+1} - x_{\varepsilon_k}) + \dot{x}_k \rangle \\
    & = \tau_{k+1} \dot{\tau}_k \|x_{k+1} - x_{\varepsilon_k}\|^2 + (\dot{\tau}_k + \tau_{k+1}\tau_k) \langle x_{k+1} - x_{\varepsilon_k}, \dot{x}_k \rangle + \tau_{k} \|\dot{x}_k\|^2  \\
    & ~ ~ -\tau_{k+1} \delta\sqrt{s \varepsilon_k} \langle x_{k+1} - x_{\varepsilon_k}, \dot{x}_{k-1} \rangle - \delta \sqrt{s\varepsilon_k} \langle \dot{x}_k, \dot{x}_{k-1} \rangle - \tau_{k+1} \tau_k \langle x_{k+1} - x_{\varepsilon_k}, \dot{x}_{\varepsilon_{k-1}} \rangle  \\ 
    & ~ ~ - \tau_k \langle \dot{x}_k, \dot{x}_{\varepsilon_{k-1}} \rangle - s\tau_{k+1} \langle \nabla \varphi_k(y_k), x_{k+1} - x_{\varepsilon_k} \rangle - s\langle \nabla \varphi_k(y_k), \dot{x}_k \rangle.
\end{align*}
We evaluate now the term $\langle x_{k+1} - x_{\varepsilon_k}, \dot{x}_{k-1} \rangle$. By the first equality in \eqref{phase-space}, we have
\begin{align*}
    \langle x_{k+1} - x_{\varepsilon_k}, \dot{x}_{k-1} \rangle & = \left \langle x_{k+1} - x_{\varepsilon_k}, \dfrac{1}{1 -\delta \sqrt{s\varepsilon_k}}(\dot{x}_k + s\nabla \varphi_k(y_k)) \right\rangle \\
    & = \dfrac{1}{1- \delta \sqrt{s \varepsilon_k}} \langle x_{k+1} - x_{\varepsilon_k}, \dot{x}_k \rangle + \dfrac{s}{1 - \delta\sqrt{s\varepsilon_k}} \langle \nabla \varphi_k (y_k), x_{k+1} - x_{\varepsilon_k} \rangle.
\end{align*}
Similarly, the term $\langle \dot{x}_k , \dot{x}_{k-1} \rangle$ is calculated as follows:
\begin{align*}
    \langle \dot{x}_k , \dot{x}_{k-1} \rangle = \dfrac{1}{1 -\delta \sqrt{s\varepsilon_k}} \|\dot{x}_k\|^2 + \dfrac{s}{1 - \delta \sqrt{s\varepsilon_k}} \langle \nabla \varphi_k(y_k), \dot{x}_k \rangle.
\end{align*}
Then, we obtain
\begin{align}
     & E_{\mathrm{mix},k+1} - E_{\mathrm{mix},k}  \leq  \tau_{k+1} \dot{\tau}_k\|x_{k+1} - x_{\varepsilon_k}\|^2  + \left( \tau_k - \dfrac{\delta\sqrt{s \varepsilon_k}}{1 - \delta\sqrt{s \varepsilon_k}} \right)\|\dot{x}_k\|^2 \notag \\
     & + \left(\dot{\tau}_k + \tau_{k+1} \tau_k  - \dfrac{\tau_{k+1} \delta \sqrt{s\varepsilon_k}}{1 - \delta \sqrt{s\varepsilon_k}} \right) \langle x_{k+1} - x_{\varepsilon_k}, \dot{x}_k \rangle  - \dfrac{s\tau_{k+1}}{1 - \delta \sqrt{s\varepsilon_k}} \langle \nabla \varphi_k(y_k), x_{k+1} - x_{\varepsilon_k} \rangle \notag \\
     &    - \dfrac{s}{1 -\delta \sqrt{s \varepsilon_k}} \langle \nabla \varphi_k(y_k) , \dot{x}_k \rangle  - \tau_{k+1} \tau_k \langle x_{k+1} - x_{\varepsilon_k}, \dot{x}_{\varepsilon_{k-1}} \rangle - \tau_k \langle \dot{x}_k, \dot{x}_{\varepsilon_{k-1}} \rangle. \label{difference-mixed-energy}
\end{align}
We are going to estimate the term $ \langle \nabla\varphi_k(y_k), x_{k+1} - x_{\varepsilon_k} \rangle$. Since $\lim_{k \to +\infty}\varepsilon_k = 0$ and the parameters $\lambda, s, \delta$ are positive, there exists a positive integer $N$ (without loss of generality, assume that $k_0 \geq N$) such that the following inequalities hold for all $k \geq N$:
\begin{equation}
    1-s\varepsilon_k > 0;~1-\lambda\sqrt{s\varepsilon_k} > 0;~ 1-\delta\sqrt{s\varepsilon_k} > 0. \label{condition:positive_1minusvarepsilon}
\end{equation}
Applying Lemma \ref{lemma:extended-descent-lemma} in Appendix~\ref{sec:appendix} for the $\varepsilon_k$-strongly convex function $\varphi_k$ with $y = y_k$, $x = x_{\varepsilon_k}$, and using the facts that $x_{k+1} = y_k - s \nabla \varphi_k(y_k)$ and the gradient $\nabla \varphi_k$ is $L_k$-Lipschitz continuous with $L_k := L + \varepsilon_k$, we obtain
\begin{align} \label{descent_x_epsilon}
    \varphi_k(x_{k+1}) - \varphi_k(x_{\varepsilon_k}) \leq \langle \nabla \varphi_k(y_k), y_k - x_{\varepsilon_k} \rangle + \left( \dfrac{L_k s^2}{2} - s \right) \| \nabla \varphi_k(y_k)\|^2 - \dfrac{\varepsilon_k}{2} \|y_k - x_{\varepsilon_k}\|^2.
\end{align}
We have
\begin{eqnarray*}
    &&  \|x_{k+1} - x_{\varepsilon_k}\|^2 = \| y_k - x_{\varepsilon_k} - s\nabla \varphi_k(y_k)\|^2  \\
    & =& \|y_k - x_{\varepsilon_k}\|^2 -2s \langle \nabla\varphi_k(y_k), y_k - x_{\varepsilon_k} \rangle + s^2 \|\nabla \varphi_k(y_k)\|^2 \leq \|y_k - x_{\varepsilon_k}\|^2 + s^2 \|\nabla \varphi_k(y_k)\|^2   \\
    && + 2s \left[ (\varphi_k(x_{\varepsilon_k})-\varphi_k(x_{k+1})) + \left( \dfrac{L_k s^2}{2} - s \right) \| \nabla \varphi_k(y_k)\|^2 - \dfrac{\varepsilon_k}{2} \|y_k - x_{\varepsilon_k}\|^2 \right]  \\
    & =& (1-s\varepsilon_k)\|y_k - x_{\varepsilon_k}\|^2 +  s^2\left( L_k s - 1 \right) \| \nabla \varphi_k(y_k)\|^2 + 2s(\varphi_k(x_{\varepsilon_k})-\varphi_k(x_{k+1})).
\end{eqnarray*}
This is equivalent to  
\begin{align} \label{x_k_x_epsilon}
    \|y_k - x_{\varepsilon_k}\|^2  \geq   \dfrac{s^2\left(1-L_k s\right) \| \nabla \varphi_k(y_k)\|^2  -2s(\varphi_k(x_{\varepsilon_k})-\varphi_k(x_{k+1}))+\|x_{k+1} - x_{\varepsilon_k}\|^2}{(1-s\varepsilon_k)}.
\end{align}
From \eqref{descent_x_epsilon} and \eqref{x_k_x_epsilon}, we derive
\begin{align}
    & - \langle \nabla\varphi_k(y_k), x_{k+1} - x_{\varepsilon_k} \rangle =  - \langle \nabla \varphi_k(y_k), y_k - x_{\varepsilon_k} \rangle - \langle \nabla \varphi_k(y_k), x_{k+1} - y_k \rangle \notag \\
    & = - \langle \nabla \varphi_k(y_k), y_k - x_{\varepsilon_k} \rangle - \langle \nabla \varphi_k(y_k), - s\nabla \varphi_k(y_k) \rangle \notag\\
     & \leq \left( \varphi(x_{\varepsilon_k}) - \varphi_k(x_{k+1}) \right) + \left(\dfrac{L_k s^2}{2} -s \right) \| \nabla \varphi_k(y_k)\|^2  - \dfrac{\varepsilon_k}{2}\|y_k - x_{\varepsilon_k}\|^2 + s \|\nabla\varphi_k(y_k)\|^2 \notag \\
     & \leq \left( \varphi(x_{\varepsilon_k}) - \varphi_k(x_{k+1}) \right) + \left( \dfrac{L_k s^2}{2} + \dfrac{\varepsilon_ks^2\left( L_k s - 1 \right)}{2(1-s\varepsilon_k)}\right) \| \nabla \varphi_k(y_k)\|^2 \notag \\
     & -  \dfrac{\varepsilon_k}{2(1-s\varepsilon_k)} \|x_{k+1} - x_{\varepsilon_k}\|^2 + \dfrac{s\varepsilon_k}{(1-s\varepsilon_k)}  \left( \varphi(x_{\varepsilon_k}) - \varphi_k(x_{k+1}) \right) \notag \\
     & = \dfrac{1}{1-s\varepsilon_k} \left[\left( \varphi(x_{\varepsilon_k}) - \varphi_k(x_{k+1}) \right) + \dfrac{Ls^2}{2}  \| \nabla \varphi_k(y_k)\|^2 -  \dfrac{\varepsilon_k}{2} \|x_{k+1} - x_{\varepsilon_k}\|^2 \right]. \label{grad-phi-x-xepsilon}
\end{align}
We continue to estimate the terms $\langle \dot{x}_k, \dot{x}_{\varepsilon_{k-1}} \rangle$ and $\langle x_{k+1} - x_{\varepsilon_k}, \dot{x}_{\varepsilon_{k-1}} \rangle$.
Using Lemma \ref{lemma:visco-curve}(i) and Lemma \ref{lemma:difference-tikhonov-value}(ii), we have
\begin{align} \label{xdot_inequality-1}
    -\tau_k \langle \dot{x}_k, \dot{x}_{\varepsilon_{k-1}} \rangle \leq \dfrac{a\tau_k}{2} \| \dot{x}_{\varepsilon_{k-1}}\|^2 + \dfrac{\tau_k}{2a} \|\dot{x}_k\|^2 \leq \dfrac{a\tau_k}{2} \dfrac{(\dot{\varepsilon}_{k-1})^2}{\varepsilon_{k-1}^2} \|x^*\|^2 + \dfrac{\tau_k}{2a} \|\dot{x}_k\|^2,
\end{align}
and
\begin{align} \label{xdot_inequality-2}
    -\tau_{k+1} \tau_k \langle x_{k+1} - x_{\varepsilon_k}, \dot{x}_{\varepsilon_{k-1}} \rangle \leq \dfrac{b\tau_k}{2} \dfrac{(\dot{\varepsilon}_{k-1})^2}{\varepsilon_{k-1}^2} \|x^*\|^2 + \dfrac{\tau_k \tau_{k+1}^2}{2b} \|x_{k+1} - x_{\varepsilon_k}\|^2.
\end{align}
We evaluate the last term in the mixed energy component as follows:
\begin{align}
    & \mu_{k+1} E_{\mathrm{mix},k+1}  = \dfrac{\mu_{k+1}}{2} \| \tau_{k+1} (x_{k+1} - x_{\varepsilon_k}) + \dot{x}_k\|^2 \notag \\
    & = \dfrac{\mu_{k+1} \tau_{k+1}^2}{2} \|x_{k+1} - x_{\varepsilon_k}\|^2 + \dfrac{\mu_{k+1}}{2}\| \dot{x}_k\|^2 + \mu_{k+1} \tau_{k+1} \langle x_{k+1} -x_{\varepsilon_k}, \dot{x}_k \rangle. \label{mu-mixed-energy}
\end{align}
It follows from \eqref{difference-mixed-energy}, \eqref{grad-phi-x-xepsilon}, \eqref{xdot_inequality-1}, \eqref{xdot_inequality-2} and \eqref{mu-mixed-energy} that
\begin{align}
    &\quad E_{\mathrm{mix},k+1} - E_{\mathrm{mix},k} + \mu_{k+1} E_{\mathrm{mix},k+1} 
     \leq -\dfrac{s\tau_{k+1}}{(1 - \delta \sqrt{s\varepsilon_k})(1-s\varepsilon_k)} \left( \varphi_k(x_{k+1}) - \varphi_k(x_{\varepsilon_k}) \right) \notag \\
    & + \left( \tau_{k+1} \dot{\tau}_k + \dfrac{\tau_k \tau_{k+1}^2}{2b} + \dfrac{\mu_{k+1}\tau_{k+1}^2}{2} - \dfrac{s\varepsilon_k\tau_{k+1}}{2(1 - \delta \sqrt{s \varepsilon_k})(1-s\varepsilon_k)}\right) \| x_{k+1} - x_{\varepsilon_k}\|^2 \notag \\
    & + \left( \tau_k - \dfrac{\delta \sqrt{s \varepsilon_k}}{1 - \delta \sqrt{s \varepsilon_k}} + \dfrac{\tau_k}{2a} + \dfrac{\mu_{k+1}}{2}\right) \| \dot{x}_k\|^2 + (a+b) \dfrac{\tau_k}{2} \dfrac{(\dot{\varepsilon}_{k-1})^2}{\varepsilon_{k-1}^2}  \|x^*\|^2 \notag \\
    & + \dfrac{L s^3 \tau_{k+1}}{2(1- \delta \sqrt{s \varepsilon_k})(1-s\varepsilon_k)} \| \nabla \varphi_k(y_k)\|^2 - \dfrac{s}{1 -\delta \sqrt{s\varepsilon_k}} \langle \nabla \varphi_k(y_k) , \dot{x}_k \rangle \notag\\
    & + \left( \tau_{k+1} \tau_k + \dot{\tau}_k - \dfrac{\tau_{k+1} \delta \sqrt{s \varepsilon_k}}{1 - \delta \sqrt{s \varepsilon_k}} + \mu_{k+1} \tau_{k+1} \right) \langle x_{k+1} - x_{\varepsilon_k}, \dot{x}_k \rangle. \label{mixed_energy_estimate}
\end{align}
It is essential to choose the parameters $\mu_{k}$ and $\tau_k$ such that the term $\langle x_{k+1} - x_{\varepsilon_k}, \dot{x}_k \rangle$ vanishes, i.e., its coefficient must be zero:
\begin{align*}
    \tau_{k+1} \tau_k + \dot{\tau}_k - \dfrac{\tau_{k+1} \delta \sqrt{s \varepsilon_k}}{1 - \delta \sqrt{s \varepsilon_k}} + \mu_{k+1} \tau_{k+1}  = 0 
    \Longleftrightarrow  \mu_{k+1} = - \dfrac{\dot{\tau}_k}{\tau_{k+1}} - \tau_k + \dfrac{\delta \sqrt{s \varepsilon_k}}{1 - \delta \sqrt{s \varepsilon_k}}.
\end{align*}
Taking $\tau_k := \lambda\sqrt{s\varepsilon_k}$ where $0< \lambda < \delta$, we yield
\begin{align*}
    \mu_{k+1} = \left(\sqrt{\dfrac{\varepsilon_k}{\varepsilon_{k+1}}}-1 \right) + \left( \dfrac{\delta}{1 - \delta \sqrt{s \varepsilon_k}} -\lambda \right) \sqrt{s \varepsilon_k}.
\end{align*}
For all $k \geq k_0$, we have $$\frac{\delta}{1 - \delta \sqrt{s \varepsilon_k}} -\lambda > \delta - \lambda > 0.$$
Therefore, our choice of $\tau_k$ guarantees that $\mu_k$ is positive for all $k \geq k_0$.

Next, let us estimate the kinetic energy part 
\begin{align*}
    E_{\mathrm{kin},k+1} - E_{\mathrm{kin},k} + \mu_{k+1} E_{\mathrm{kin},k+1}.
\end{align*}
Once again, using the equalities \eqref{phase-space}, we get
\begin{align*}
    \| \dot{x}_{k-1}\|^2 &= \dfrac{1}{(1-\delta \sqrt{s \varepsilon_k})^2} \|\dot{x}_k + s \nabla \varphi_k(y_k) \|^2 \\
    &= \dfrac{1}{(1-\delta \sqrt{s \varepsilon_k})^2} \left( \|\dot{x}_k\|^2 + s^2 \|\nabla \varphi_k(y_k)\|^2 +2s \langle \nabla \varphi_k(y_k) , \dot{x}_k \rangle\right).
\end{align*}
Therefore,
\begin{align}
    & E_{\mathrm{kin},k+1} - E_{\mathrm{kin},k} + \mu_{k+1} E_{\mathrm{kin},k+1} = \dfrac{\beta_{k+1}}{2} \|\dot{x}_k\|^2 - \dfrac{\beta_{k}}{2} \|\dot{x}_{k-1}\|^2 + \dfrac{\mu_{k+1} \beta_{k+1}}{2} \|\dot{x}_k\|^2 \notag \\ &= \left( \dfrac{\beta_{k+1}}{2} - \dfrac{\beta_k}{2(1 - \delta \sqrt{s \varepsilon_k})^2} + \dfrac{\mu_{k+1} \beta_{k+1}}{2} \right) \|\dot{x}_k\|^2 - \dfrac{\beta_k s^2}{2(1- \delta \sqrt{s \varepsilon_k})^2} \|\nabla \varphi_k(y_k)\|^2 \notag \\
    & - \dfrac{\beta_k s}{(1 - \delta \sqrt{s \varepsilon_k})^2} \langle \nabla \varphi_k(y_k), \dot{x}_k \rangle . \label{kinetic-energy-estimate}
\end{align}

Finally, we estimate the potential energy part 
\begin{align*}
    E_{\mathrm{pot},k+1} - E_{\mathrm{pot},k} + \mu_{k+1} E_{\mathrm{pot},k+1}.
\end{align*}
We assume that the sequence $(\alpha_k)_k$ is positive and nonincreasing, with the specific form to be determined later. We now proceed with the following estimation
\begin{align*}
    & E_{\mathrm{pot},k+1} - E_{\mathrm{pot},k} =  \alpha_{k+1} \left( \varphi_{k+1}(x_{k+1}) - \varphi_{k+1}(x_{\varepsilon_{k+1}}) \right) - \alpha_k \left( \varphi_k(x_k) - \varphi_k(x_{\varepsilon_k}) \right)\\
    \leq & \alpha_k( \varphi_{k+1}(x_{k+1}) - \varphi_k(x_k)) + \alpha_k( \varphi_{k}(x_{\varepsilon_k}) - \varphi_{k+1}(x_{\varepsilon_{k+1}})).
\end{align*}    
From the definition of $\varphi_k$, Lemma \ref{lemma:visco-curve}(i), and Lemma \ref{lemma:difference-tikhonov-value}(i), we yield
$$\varphi_{k+1}(x_{k+1}) = f(x_{k+1}) + \dfrac{\varepsilon_{k+1}}{2}\|x_{k+1}\|^2 =  \varphi_{k}(x_{k+1}) + \dfrac{\varepsilon_{k+1} - \varepsilon_k}{2} \|x_{k+1}\|^2 \leq \varphi_{k}(x_{k+1})$$
and
$$\varphi_{k}(x_{\varepsilon_k}) - \varphi_{k+1}(x_{\varepsilon_{k+1}}) \leq \dfrac{\varepsilon_k - \varepsilon_{k+1}}{2} \|x_{\varepsilon_{k+1}}\|^2 \leq -\dfrac{\dot{\varepsilon}_k}{2} \|x^*\|^2.$$
Therefore, 
$$E_{\mathrm{pot},k+1} - E_{\mathrm{pot},k} \leq \alpha_k ( \varphi_k(x_{k+1}) - \varphi_k(x_k)) -\dfrac{\alpha_k\dot{\varepsilon}_k}{2} \|x^*\|^2.$$
Similarly to \eqref{descent_x_epsilon} but with $x=x_k$ and using \eqref{phase-space}, we obtain 
\begin{align*}
    & \varphi_k(x_{k+1}) - \varphi_k(x_k) \leq \langle \nabla \varphi_k(y_k), y_k - x_k \rangle + \left( \dfrac{L_ks^2}{2} -s \right) \|\nabla \varphi_k(y_k)\|^2 - \dfrac{\varepsilon_k}{2} \|y_k - x_k\|^2 \\
    &= \langle \nabla\varphi_k(y_k), \dot{x}_k \rangle + \dfrac{L_k s^2}{2} \|\nabla \varphi_k(y_k)\|^2 - \dfrac{\varepsilon_k}{2} \| \dot{x}_k + s\nabla \varphi_k(y_k)\|^2\\
    & = (1 - s\varepsilon_k) \langle \nabla \varphi_k(y_k), \dot{x}_k \rangle - \dfrac{\varepsilon_k}{2} \| \dot{x}_k\|^2 + \dfrac{Ls^2}{2} \| \nabla \varphi_k(y_k)\|^2.
\end{align*}
Therefore,
\begin{align*}
    & E_{\mathrm{pot},k+1} - E_{\mathrm{pot},k} \notag \\
    \leq & ~ \alpha_k(1-s\varepsilon_k) \langle \nabla \varphi_k(y_k), \dot{x}_k \rangle - \dfrac{\alpha_k \varepsilon_k}{2}\|\dot{x}_k\|^2 + \dfrac{Ls^2 \alpha_k}{2} \| \nabla \varphi_k(y_k)\|^2 - \dfrac{\alpha_k \dot{\varepsilon}_k}{2}\|x^*\|^2.
\end{align*}
We evaluate the term $\varphi_{k+1}(x_{k+1}) - \varphi_{k+1}(x_{\varepsilon_{k+1}})$ in $E_{\mathrm{pot},k+1}$ as follows:
\begin{align*}
    & \varphi_{k+1}(x_{k+1}) - \varphi_{k+1}(x_{\varepsilon_{k+1}}) \\
     = & \left( \varphi_k(x_{k+1}) - \varphi_k(x_{\varepsilon_k}) \right) +  \left( \varphi_{k+1}(x_{k+1}) - \varphi_{k}(x_{k+1})\right) + \left( \varphi_{k}(x_{\varepsilon_k}) - \varphi_{k+1}(x_{\varepsilon_{k+1}}) \right) \\
     \leq & \left( \varphi_k(x_{k+1}) - \varphi_k(x_{\varepsilon_k}) \right) - \dfrac{\dot{\varepsilon}_k}{2} \|x^*\|^2.
\end{align*}
Hence, 
\begin{align}
    & E_{\mathrm{pot},k+1} - E_{\mathrm{pot},k} + \mu_{k+1} E_{\mathrm{pot},k+1} \leq  \mu_{k+1} \alpha_{k+1} \left( \varphi_k(x_{k+1}) - \varphi_k(x_{\varepsilon_k}) \right)   - \dfrac{\alpha_k \varepsilon_k}{2}\|\dot{x}_k\|^2 \notag \\
    & + \dfrac{Ls^2 \alpha_k}{2} \| \nabla \varphi_k(y_k)\|^2  + \alpha_k(1-s\varepsilon_k) \langle \nabla \varphi_k(y_k), \dot{x}_k \rangle - \dfrac{\alpha_k(1+\mu_{k+1}) \dot{\varepsilon}_k}{2}\|x^*\|^2. \label{potential-energy-estimate}
\end{align}

From all the component estimates of $E_k$ given in \eqref{mixed_energy_estimate}, \eqref{kinetic-energy-estimate}, and \eqref{potential-energy-estimate}, we have 
\begin{align}
    &E_{k+1} - E_k + \mu_{k+1} E_{k+1} \notag \\
    \leq & \left( \mu_{k+1} \alpha_{k+1} -\dfrac{s\tau_{k+1}}{(1 - \delta \sqrt{s\varepsilon_k})(1-s\varepsilon_k)}\right) \left( \varphi_k(x_{k+1}) - \varphi_k(x_{\varepsilon_k}) \right) \notag \\
    & + \left( \tau_{k+1} \dot{\tau}_k + \dfrac{\tau_k \tau_{k+1}^2}{2b} + \dfrac{\mu_{k+1}\tau_{k+1}^2}{2} - \dfrac{s\varepsilon_k\tau_{k+1}}{2(1 - \delta \sqrt{s \varepsilon_k})(1-s\varepsilon_k)} \right) \|x_{k+1} - x_{\varepsilon_k}\|^2 \notag\\
    &+ \left( \tau_k - \dfrac{\delta \sqrt{s \varepsilon_k}}{1 - \delta \sqrt{s \varepsilon_k}} + \dfrac{\tau_k}{2a} + \dfrac{\mu_{k+1}+\beta_{k+1}}{2} - \dfrac{\beta_{k}}{2(1 - \delta \sqrt{s \varepsilon_k})^2} + \dfrac{\mu_{k+1} \beta_{k+1}}{2} - \dfrac{\alpha_k \varepsilon_k}{2} \right) \| \dot{x}_k\|^2 \notag \\
    &+ \left( \dfrac{Ls^2 \alpha_k}{2} + \dfrac{L s^3 \tau_{k+1}}{2(1- \delta \sqrt{s \varepsilon_k})(1-s\varepsilon_k)} - \dfrac{\beta_k s^2}{2(1 - \delta \sqrt{s \varepsilon_k})^2}\right)\|\nabla \varphi_k(y_k)\|^2 \notag \\
    &+ \left( (a+b)\tau_k \dfrac{(\dot{\varepsilon}_{k-1})^2}{\varepsilon_{k-1}^2}  - \alpha_k(1+\mu_{k+1}) \dot{\varepsilon}_{k} \right) \dfrac{\|x^*\|^2}{2} \notag \\
    &+ \left( - \dfrac{s}{1 - \delta \sqrt{s \varepsilon_k}} - \dfrac{\beta_k s}{(1 - \delta \sqrt{s \varepsilon_k})^2} + \alpha_k(1 - s\varepsilon_k)\right) \langle \nabla \varphi_k(y_k), \dot{x}_k \rangle. \label{ineq:estimate_Ek}
\end{align}
In order to eliminate the term $\langle \nabla \varphi_k(y_k), \dot{x}_k \rangle$, we select the parameters $\alpha_k$ and $\beta_k$ such that 
\begin{eqnarray*}
     && - \dfrac{s}{1 - \delta \sqrt{s \varepsilon_k}} - \dfrac{\beta_k s}{(1 - \delta \sqrt{s \varepsilon_k})^2} + \alpha_k(1 - s\varepsilon_k) = 0 \\
    \Longleftrightarrow && \alpha_k = \dfrac{s}{(1 - \delta \sqrt{s \varepsilon_k})(1 - s\varepsilon_k)} \left(1 + \dfrac{\beta_k}{1 - \delta \sqrt{s \varepsilon_k}} \right).
\end{eqnarray*}
By setting $\beta_k := q ( 1- \delta \sqrt{s \varepsilon_k})$ where $q > 0$, we obtain 
\begin{align*}
    \alpha_k = \dfrac{(1 + q)s}{(1 - \delta \sqrt{s \varepsilon_k})(1 - s\varepsilon_k)}.
\end{align*}
Since $(\varepsilon_k)_k$ is a nonincreasing sequence, so is $(\alpha_k)_k$.

Let  $\zeta_k$, $\eta_k$, $\nu_k$, $\xi_k$, and $\theta_k$ denote, respectively, the coefficients of the terms $\varphi_k(x_{k+1}) - \varphi_k(x_{\varepsilon_k})$, $\|x_{k+1} - x_{\varepsilon_k}\|^2$, $\|\dot{x}_k\|^2$, $\| \nabla \varphi_k(y_k)\|^2$, and $\|x^*\|^2/2$, the inequality \eqref{ineq:estimate_Ek} simplifies to 
\begin{eqnarray*}
    && E_{k+1} - E_k + \mu_{k+1} E_{k+1} \\
    &\leq & \zeta_k \left( \varphi_k(x_{k+1}) - \varphi_k(x_{\varepsilon_k}) \right) + \eta_k \|x_{k+1} - x_{\varepsilon_k}\|^2 + \nu_k \|\dot{x}_k\|^2 + \xi_k \| \nabla \varphi_k(y_k)\|^2 + \dfrac{\theta_k}{2} \|x^*\|^2.
\end{eqnarray*}
As $\dot{\varepsilon}_{k} \leq 0$ for all $k$, the coefficient $\theta_k = (a+b)\tau_k \frac{(\dot{\varepsilon}_{k-1})^2}{\varepsilon_{k-1}^2}  - \alpha_k(1+\mu_{k+1}) \dot{\varepsilon}_{k} > 0, \forall k \geq k_0$.
Now we show that $\zeta_k$, $\eta_k$, $\nu_k$, and $\xi_k$ are nonpositive under the conditions of $(\varepsilon_k)_k$ in \eqref{condition:tikhonov-sequence}.
Before starting, let us note two inequalities used in the following: for all $k \geq k_0$, 
\begin{equation*}
    \dot{\tau}_k = \tau_{k+1} - \tau_{k} = \lambda\sqrt{s}(\sqrt{\varepsilon_{k+1}} - \sqrt{\varepsilon_{k}}) \leq 0, \quad \beta_k = q (1 - \delta\sqrt{s \varepsilon_k}) < q.
\end{equation*}

For brevity, set $\sigma_k := (1-\delta\sqrt{s\varepsilon_k})(1-s\varepsilon_k),$ then we have

\begin{flalign*}
\zeta_k
&:= \mu_{k+1}\alpha_{k+1}-\frac{s\tau_{k+1}}{\sigma_k}
\le \mu_{k+1}\alpha_k-\frac{s\tau_{k+1}}{\sigma_k}
= \frac{s\left[(1+q)\mu_{k+1}-\tau_{k+1}\right]}{\sigma_k}
= \frac{s\sqrt{s\varepsilon_k}}{\sigma_k}\,\Theta_k \le 0,&
\end{flalign*}
where
\[
\Theta_k :=
(1+q)\Bigl(\frac{1}{\sqrt{s\varepsilon_{k+1}}}-\frac{1}{\sqrt{s\varepsilon_k}}
+\frac{\delta}{1-\delta\sqrt{s\varepsilon_k}}-\lambda\Bigr)
-\lambda\sqrt{\frac{\varepsilon_{k+1}}{\varepsilon_k}}.
\]

\begin{flalign*}
\eta_k
&:= \tau_{k+1}\dot{\tau}_k+\frac{\tau_k\tau_{k+1}^2}{2b}
+\frac{\mu_{k+1}\tau_{k+1}^2}{2}
-\frac{s\varepsilon_k\tau_{k+1}}{2\sigma_k} &\\
&= \tau_{k+1}^2 \left[\dfrac{\dot{\tau}_k}{\tau_{k+1}} + \dfrac{\tau_k}{2b} + \dfrac{1}{2} \left(- \dfrac{\dot{\tau}_k}{\tau_{k+1}} - \tau_k  + \dfrac{\delta \sqrt{s \varepsilon_k}}{1 - \delta \sqrt{s \varepsilon_k}} \right) - \dfrac{s\varepsilon_k}{2\tau_{k+1}\sigma_k} \right] &\\
&\le \frac{\tau_{k+1}^2}{2}\left(\frac{1-b}{b}\tau_k
+\frac{\delta\sqrt{s\varepsilon_k}}{1-\delta\sqrt{s\varepsilon_k}}
-\frac{s\varepsilon_k}{\lambda\sqrt{s\varepsilon_{k+1}}}\right)
(\mbox{since } \dot{\tau}_k<0)&\\
&\le \frac{\tau_{k+1}^2\sqrt{s\varepsilon_k}}{2\lambda}
\left(\frac{1-b}{b}\lambda^2+\frac{\delta\lambda}{1-\delta\sqrt{s\varepsilon_k}}-1\right)
\le 0.&
\end{flalign*}

\begin{flalign*}
\nu_k
&:= \tau_k-\frac{\delta\sqrt{s\varepsilon_k}}{1-\delta\sqrt{s\varepsilon_k}}
+\frac{\tau_k}{2a}+\frac{\mu_{k+1}}{2}
+\left(\frac{\beta_{k+1}}{2}-\frac{\beta_k}{2(1-\delta\sqrt{s\varepsilon_k})^2}\right)
+\frac{\mu_{k+1}\beta_{k+1}}{2}-\frac{\alpha_k\varepsilon_k}{2} &\\
&\le \dfrac{1+a(1-q)}{2a}\lambda\sqrt{s \varepsilon_k} -\dfrac{\delta \sqrt{s \varepsilon_k}}{2(1 - \delta \sqrt{s \varepsilon_k})}   + \dfrac{1+q}{2}\left(\sqrt{\dfrac{\varepsilon_k}{\varepsilon_{k+1}}}-1 \right) +\dfrac{q\sqrt{s \varepsilon_k}\delta^2 \sqrt{s \varepsilon_{k+1}}}{2(1 - \delta \sqrt{s \varepsilon_k})} & \\
&\le \frac{\sqrt{s\varepsilon_k}}{2}\,\Psi_k \le 0,&
\end{flalign*}
where
\[
\Psi_k =
\frac{(1+a(1-q))\lambda}{a}
-\delta
+(1+q)\left(\frac{1}{\sqrt{s\varepsilon_{k+1}}}-\frac{1}{\sqrt{s\varepsilon_k}}\right)
+\frac{q\delta^2\sqrt{s\varepsilon_{k+1}}}{1-\delta\sqrt{s\varepsilon_k}}.
\]

\begin{flalign*}
\xi_k
&:= \frac{Ls^2\alpha_k}{2}
+\frac{Ls^3\tau_{k+1}}{2\sigma_k}
-\frac{\beta_k s^2}{2(1-\delta\sqrt{s\varepsilon_k})^2}
= \frac{s^2}{2\sigma_k}
\Bigl[Ls\bigl(1+\lambda\sqrt{s\varepsilon_{k+1}}\bigr)+q(L_k s-1)\Bigr]\le 0.&
\end{flalign*}

Since the sequence $(\varepsilon_k)_k$ satisfies the system of conditions \eqref{condition:tikhonov-sequence}, it follows that all the sequences $\zeta_k$, $\eta_k$, $\nu_k$, $\xi_k$ are not positive for all $k \geq k_0$. Therefore,
\begin{equation}
    E_{k+1} - E_k + \mu_{k+1} E_{k+1} \leq \dfrac{\theta_k}{2}  \|x^*\|^2,~ \forall k \geq k_0.
\end{equation}
 Multiplying both sides by $\displaystyle \gamma_{k} := \prod_{j=k_0}^k  (1 + \mu_j) > 0$, we obtain
    \begin{align*}
        \gamma_k (1 + \mu_{k+1}) E_{k+1} - \gamma_k E_k \leq  \dfrac{\gamma_k\theta_k}{2} \|x^*\|^2
        \Leftrightarrow
        \gamma_{k+1} E_{k+1} - \gamma_k E_k \leq \dfrac{\gamma_k\theta_k}{2} \|x^*\|^2.
    \end{align*}
    Telescoping the sum from $k_0$ to $k$, we yield
    \begin{align*}
        \gamma_{k+1} E_{k+1} - \gamma_{k_0} E_{k_0} \leq \dfrac{1}{2}\left( \sum_{j=k_0}^k \gamma_j \theta_j \right) \|x^*\|^2.
    \end{align*}
    Dividing both sides by $\gamma_{k+1}$, we conclude
    \begin{align*}
        E_{k+1} \leq \dfrac{\gamma_{k_0}}{\gamma_{k+1}}E_{k_0} + \dfrac{1}{2\gamma_{k+1}} \left( \sum_{j = k_0}^k \gamma_{j} \theta_j \right) \|x^*\|^2.
    \end{align*}
    The proof is complete.
\end{proof}

The conditions \eqref{condition:tikhonov-sequence} are stated for an arbitrary Tikhonov schedule $(\varepsilon_k)_k$. A natural question is whether the constants $a,b,\lambda,q$ can be chosen independently of the particular form of $(\varepsilon_k)_k$. This is clarified in Proposition~\ref{proposition:constant-parameter}, and representative admissible choices are summarized in Table~\ref{table:condition_necessary_parameters}.

\begin{proposition} \label{proposition:constant-parameter}
    Assume that a sequence $(\varepsilon_k)_k$ satisfies \eqref{condition:tikhonov-sequence}. Then the existence of positive constants $a,b,\lambda,q$ fulfilling \eqref{condition:tikhonov-sequence} necessarily requires the following constraints:
    \begin{align} \label{condition:constant-parameter} \tag{$\mathcal{K}_1$}
        \begin{cases}
            (i) & \dfrac{1 + q}{2 + q} \delta < \lambda < \delta, \\
            (ii) & (1 + a(1 - q)) \lambda - a \delta < 0, \\
            (iii) & \dfrac{1-b}{b}\lambda^2  + \delta \lambda - 1 < 0, \\
            (iv) & Ls < \dfrac{q}{q + 1}.
        \end{cases}
    \end{align}
\end{proposition}
  
\begin{proof}
    Since the sequence $(\varepsilon_k)_k$ is positive and nonincreasing, we note
    \begin{align*}
        1 - s\varepsilon_k < 1; ~ 1 -\delta\sqrt{s\varepsilon_k}<1; ~\dfrac{1}{\sqrt{s\varepsilon_{k+1}}} - \dfrac{1}{\sqrt{s\varepsilon_k}} \geq 0; ~ \sqrt{\dfrac{\varepsilon_{k+1}}{\varepsilon_k}} \leq 1.
    \end{align*}
    Firstly, we identify from condition \eqref{condition:tikhonov-sequence}$(iv)$ the following inequality
    \begin{equation*}
        Ls + q(Ls-1) < 0 \Longleftrightarrow Ls < \dfrac{q}{q + 1}.
    \end{equation*}
    Secondly, it follows from $(i)$ of \eqref{condition:tikhonov-sequence} that
    \begin{align}
        (1+q)(\delta - \lambda) - \lambda < 0 \Longleftrightarrow \dfrac{1 + q}{2 + q} \delta < \lambda. \label{condition:K1_i_lower}
    \end{align}
    Similarly, from condition $(iii)$ of \eqref{condition:tikhonov-sequence} we obtain the following inequality
    \begin{align}
        (1 + a(1 - q)) \lambda - a \delta < 0. \label{condition:K1_i_upper}
    \end{align}
    We must solve this inequality in $\lambda$ to verify its existence.\\
    $\bullet$ If $1 + (1-q)a > 0$, i.e., $q < 1 + \dfrac{1}{a}$, then this inequality is equivalent to $\lambda < \dfrac{a\delta}{1 + a(1 - q)}$. 
    Comparing the lower bound on $\lambda$ from \eqref{condition:K1_i_lower} with its upper bound from \eqref{condition:K1_i_upper}, we find a necessary condition on $a$ and $q$ to ensure the existence of such a $\lambda$, specifically,
    \begin{align*}
        \dfrac{1 + q}{2 + q} < \dfrac{a}{1 + a(1 - q)} \Longleftrightarrow a > \dfrac{q+1}{q^2+q+1}.
    \end{align*}
    Hence, if $\left(0 < q \leq 1 \textrm{ and } a > \dfrac{q+1}{q^2+q+1}\right)$ or $\left(q > 1 \textrm{ and } \dfrac{1}{q-1} > a > \dfrac{q+1}{q^2+q+1}\right)$, then 
    \begin{align}
        \dfrac{1 + q}{2 + q} \delta < \lambda < \dfrac{a}{1 + a(1 - q)} \delta. \label{condition:K1_i}
    \end{align}
    $\bullet$ If $1 + (1-q)a \leq 0$, i.e., $q >1$ and $a \geq \dfrac{1}{q-1}$, then the inequality \eqref{condition:K1_i_upper} is true for all $\lambda$. In this case, $\lambda$ only needs to satisfy the condition \eqref{condition:K1_i_lower}. \\
    Thirdly, we derive from condition $(ii)$ in \eqref{condition:tikhonov-sequence} the following inequality
    \begin{equation}
        \dfrac{1-b}{b}\lambda^2  + \delta \lambda - 1 < 0. \label{condition:K1_ii_initial}
    \end{equation}    
    The solutions to the inequality \eqref{condition:K1_ii_initial} under the conditions of $b$ and $\delta$ are given by:
    \begin{equation} \label{condition:K1_ii}
      \lambda \in
    \begin{cases}
         \left]0, \frac{1}{\delta} \right[ & \textrm{if } b = 1 \\
         \left] 0, \lambda_{-} \right[ & \textrm{if } 0 < b < 1 \\
         ]0, +\infty[ & \textrm{if } b > 1 \textrm{ and } 0 < \delta < 2 \textrm{ and } b > \dfrac{4}{4 - \delta^2} \\
         ]0, +\infty[\backslash\{\lambda_{0} \} & \textrm{if } b > 1 \textrm{ and } 0 < \delta < 2 \textrm{ and } b = \dfrac{4}{4 - \delta^2} \\
         \left] 0, \lambda_{-} \right[ \cup ]\lambda_{+}, +\infty[ & \textrm{if } b > 1 \textrm{ and} \left( \delta \geq 2 \textrm{ or} \left( 0 < \delta < 2 \textrm{ and } b < \dfrac{4}{4 - \delta^2} \right) \right)
    \end{cases} 
    \end{equation} 
    where 
    $$\lambda_{\pm} := \frac{b}{2(b-1)} \left(\delta \pm \sqrt{\delta^2 + \frac{4(1-b)}{b}} \right) \textrm{ and } \lambda_{0} := \frac{b\delta}{2(b-1)}.$$
    We are going to verify the existence of $\lambda$ in \eqref{condition:K1_i_lower}, \eqref{condition:K1_i} and \eqref{condition:K1_ii}.\\
    $\bullet$ In case $0 < \delta < 2$, the conditions \eqref{condition:K1_i_lower} and \eqref{condition:K1_i} are still valid if $b > \frac{4}{4 - \delta^2}$.\\
    $\bullet$ When $0 < q \leq 1$ and $\delta \geq 2$, in order to guarantee the validity of $\lambda$ in \eqref{condition:K1_i} and \eqref{condition:K1_ii}, we are looking for conditions on positive parameters $a$ and $b$ such that $a > \frac{q+1}{q^2+q+1}$, $b > 1$, and
    \begin{equation}
        \lambda_{+} = \frac{b}{2(b-1)} \left(\delta + \sqrt{\delta^2 + \frac{4(1-b)}{b}} \right) < \dfrac{a}{1 + a(1 - q)} \delta. \label{condition:K1_ii_lambda_qlessthan1}
    \end{equation}
    Setting $\kappa := \frac{a}{1 + a(1 - q)}$, this inequality amounts to
    \begin{eqnarray*}
        &&\sqrt{\delta^2 + \frac{4(1-b)}{b}} < \left(\dfrac{2(b-1)\kappa}{b} -1 \right)\delta = \dfrac{b(2\kappa-1)-2\kappa}{b}\delta.
    \end{eqnarray*}
    Necessary conditions for a solution to exist are $\kappa > \frac{1}{2}$ and $b > \frac{2\kappa}{2\kappa-1}$. In this case, both sides are nonnegative and we can thus square the last inequality:
    \begin{eqnarray*}
        &&\delta^2 + \frac{4(1-b)}{b} <  \left(\dfrac{b(2\kappa-1)-2\kappa}{b}\right)^2\delta^2 \Leftrightarrow \frac{4(1-b)}{b} <  \dfrac{(b(2\kappa-1)-2\kappa)^2-b^2}{b^2}\delta^2 \\ 
        &\Leftrightarrow& 4(1-b)b <  4\kappa(b-1)(b(\kappa-1)-\kappa)\delta^2 \Leftrightarrow -b <  b\delta^2\kappa(\kappa-1)-(\kappa\delta)^2 \\ 
        &\Leftrightarrow& b(\delta^2\kappa(\kappa-1)+1) >(\kappa\delta)^2.
    \end{eqnarray*}
    Combining with necessary conditions just mentioned above and noting that $$\frac{(\kappa\delta)^2}{\delta^2\kappa(\kappa-1)+1} > \frac{2\kappa}{2\kappa-1},~\forall \kappa > \frac{1}{2}\left(1+\sqrt{1-\frac{4}{\delta^2}}\right), \forall \delta \geq 2,$$ we can choose conditions for \eqref{condition:K1_ii_lambda_qlessthan1}:
    \begin{eqnarray}
        &&\kappa > \frac{1}{2}\left(1+\sqrt{1-\frac{4}{\delta^2}}\right); b > \frac{(\kappa\delta)^2}{\delta^2\kappa(\kappa-1)+1}, \notag
        \\
        &\textrm{or equivalently},& a > \frac{1+\sqrt{1-\frac{4}{\delta^2}}}{2-\left(1+\sqrt{1-\frac{4}{\delta^2}}\right)(1-q)}; b > \frac{(a\delta)^2}{\delta^2a(aq-1)+(1+a(1-q))^2}. \label{condition:K1_ii_ab}
    \end{eqnarray}
    We observe that if $\kappa > 1$, i.e., $ a > \frac{1}{q}$, then $\frac{a}{1 + a(1 - q)} \delta = \kappa\delta > \delta$. Thus, $\lambda < \delta$. Otherwise (i.e., $a \leq \frac{1}{q}$), $\lambda < \frac{a}{1 + a(1 - q)} \delta$.
    We can easily verify $$\max\left\{\frac{q+1}{q^2+q+1},\frac{1+\sqrt{1-\frac{4}{\delta^2}}}{2-\left(1+\sqrt{1-\frac{4}{\delta^2}}\right)(1-q)}\right\} < \frac{1}{q};~ \frac{(a\delta)^2}{\delta^2a(aq-1)+(1+a(1-q))^2} > 1.$$
    $\bullet$ Similarly, when $q > 1$ and $\delta \geq 2$, we seek conditions on $b > 1$ such that $\lambda_{+} < \delta$. Using the same argument as the case $q \in ]0,1]$ but with $\kappa=1$, we obtain 
    \begin{equation}
        b > \delta^2. \label{condition:K1_ii_b}
    \end{equation}
    Finally, from \eqref{condition:K1_i_lower}, \eqref{condition:K1_i}, \eqref{condition:K1_ii}, \eqref{condition:K1_ii_ab} and \eqref{condition:K1_ii_b}, representative choices of the parameters $\lambda, q, a, b$ are summarized in Table~\ref{table:condition_necessary_parameters}.
\end{proof}

\begin{table}[htbp]
    \caption{Choices of positive parameters $q, a, b, \lambda$ for the system \eqref{condition:constant-parameter}} \label{table:condition_necessary_parameters}
    \begin{tabular*}{\textwidth}{@{\extracolsep\fill}lllll}
    \toprule
    $\delta$ & $q$ & $a$ & $b$ & $\lambda$ \\
    \midrule
    $]0,2[$& $]0,1]$& $\left]\dfrac{q+1}{q^2+q+1}, \dfrac{1}{q}\right]$& $\left]\dfrac{4}{4 - \delta^2}; +\infty\right[$& $\left]\dfrac{1 + q}{2 + q} \delta, \dfrac{a}{1 + a(1 - q)} \delta \right[$\\[0.5cm]
    $]0,2[$& $]0,1]$& $\left]\dfrac{1}{q}, +\infty\right[$& $\left]\dfrac{4}{4 - \delta^2}; +\infty\right[$& $\left]\dfrac{1 + q}{2 + q} \delta, \delta \right[$\\[0.5cm]
    $]0,2[$& $]1,+\infty[$& $\left]\dfrac{1}{q-1}, +\infty\right[$& $\left]\dfrac{4}{4 - \delta^2}; +\infty\right[$& $\left]\dfrac{1 + q}{2 + q} \delta, \delta \right[$\\[0.5cm]
    $[2,+\infty[$ & $]0,1]$ & $\left]\widehat{a}, \dfrac{1}{q}\right]$ & $\left]\widehat{b}, +\infty\right[$ & $\left]\widehat{\lambda}, \dfrac{a}{1 + a(1 - q)} \delta \right[$ \\[0.5cm]
    $[2,+\infty[$ & $]0,1]$ & $\left]\dfrac{1}{q}, +\infty\right[$ & $\left]\widehat{b}, +\infty\right[$ & $\left]\widehat{\lambda}, \delta \right[$ \\[0.5cm]
    $[2,+\infty[$ & $]1,+\infty[$ & $\left]\dfrac{1}{q-1}, +\infty\right[$ & $]\delta^2,+\infty[$ & $\left]\widehat{\lambda}, \delta \right[$ \\
    \bottomrule
    \end{tabular*}
    \footnotetext{Here $\delta$ is given; $\displaystyle \widehat{a}:=\max\left\{\frac{q+1}{q^2+q+1},\dfrac{1+\sqrt{1-\dfrac{4}{\delta^2}}}{2-\left(1+\sqrt{1-\dfrac{4}{\delta^2}}\right)(1-q)}\right\}$; \\ $\displaystyle \widehat{b}:=\frac{(a\delta)^2}{\delta^2a(aq-1)+(1+a(1-q))^2}$;    $\displaystyle \widehat{\lambda}:= \max\left\{\dfrac{1 + q}{2 + q} \delta,\dfrac{b}{2(b-1)} \left(\delta + \sqrt{\delta^2 + \frac{4(1-b)}{b}} \right)\right\}$.}
\end{table}   

\begin{remark} \label{remark:choices_parameter}
Among the parameters in \eqref{condition:constant-parameter}, $q$ affects Algorithm~\ref{algo:triga} through the admissible step-size $s$. In particular, condition $(iv)$ in \eqref{condition:constant-parameter} imposes
$$0< s < \dfrac{1}{L\left(1 + q^{-1}\right)}.$$
Thus, larger $q$ yields a larger upper bound for $s$. The choice of $q>0$ is otherwise free: by Proposition~\ref{proposition:constant-parameter} and Table~\ref{table:condition_necessary_parameters}, one can select the remaining parameters for both regimes $0<q\le 1$ and $q>1$. In the numerical section, we evaluate the performance of \ref{algo:triga} under two difference choices of $q$ for both cases.
\end{remark}

\section{Convergence rates for particular cases} \label{sec:particular_case}

\subsection{\texorpdfstring{$\varepsilon_k$}{TEXT} of order \texorpdfstring{$\frac{1}{k^p}, 0<p<2$}{TEXT}}

Attouch et al. \cite{Attouch:jde:22} studied the convergence rate of the values and the strong convergence of the trajectories of the dynamical system \eqref{trigs} when $\varepsilon(t) = \frac{1}{t^p}$ where $0<p<2$. In the case of our associated temporal discretization algorithm \ref{algo:triga}, we investigate the convergence rate of the values and the convergence rate to zero of $\|x_k - x_{\varepsilon_k}\|$ when $\varepsilon_k := \frac{1}{k^p}$ where $0<p<2$, which are stated in Theorem \ref{theorem:conv_rate_grad_algo}.
\begin{theorem} \label{theorem:conv_rate_grad_algo}
    Take $\varepsilon_k := \dfrac{1}{k^p}$, $0<p<2$. Suppose $f \colon \mathcal{H} \to \mathbb{R}$ is a convex differentiable function whose gradient $\nabla f$ is $L$-Lipschitz continuous. Let $(x_k)_k$ be a sequence generated by Algorithm \ref{algo:triga}. Then, we have the following convergence rates
    \def\labelenumi{\rm (\roman{enumi})}
    \def\theenumi{\roman{enumi}}
    \begin{enumerate}
        \item $f(x_k) - \underset{\mathcal{H}}{\min} f = \mathcal{O}\left(\dfrac{1}{k^p}\right)$ as $k \to +\infty$.
        \item $\|x_k - x_{\varepsilon_k}\|^2 = \mathcal{O} \left( \dfrac{1}{k^{\frac{2-p}{2}}} \right)$ as $k \to +\infty$. In particular, the sequence of iterates $(x_k)_k$ converges strongly to the minimum-norm solution $x^*$.
        \item $\|x_k - x_{k-1}\| = \mathcal{O}\left( \dfrac{1}{k^\frac{2+p}{4}} \right)$ as $k \to +\infty$.
    \end{enumerate}
\end{theorem}
\begin{proof}
    From Theorem \ref{theorem:grad-algo}, we recall 
    \begin{align*}
        E_{k+1} \leq \dfrac{\gamma_{k_0}}{\gamma_{k+1}}E_{k_0} + \dfrac{1}{2\gamma_{k+1}} \left( \sum_{j = k_0}^k \gamma_{j} \theta_j \right) \|x^*\|^2.
    \end{align*}
    To obtain the convergence rate of $E_{k+1}$, we study the asymptotic behaviors of the following terms
    $$ E_{k,1} = \dfrac{\gamma_{k_0}}{\gamma_{k+1}}E_{k_0}, \qquad E_{k,2} = \dfrac{1}{2\gamma_{k+1}} \left( \sum_{j = k_0}^k \gamma_{j} \theta_j \right).$$
    We now give a lower bound for $\gamma_k$. By the notation of Theorem \ref{theorem:grad-algo}, we have
    \begin{align*}
        \gamma_{k} &= \prod_{j=k_0}^k \left( \left( \dfrac{j}{j-1} \right)^{\frac{p}{2}} + \left( \dfrac{\delta}{1 - \frac{\delta \sqrt{s}}{(j-1)^{p/2}}} - \lambda \right) \dfrac{\sqrt{s}}{(j-1)^\frac{p}{2}} \right) \\
        & \geq \prod_{j=k_0}^k \left( \left( \dfrac{j}{j-1} \right)^{\frac{p}{2}} + \dfrac{(\delta - \lambda)\sqrt{s}}{(j-1)^\frac{p}{2}} \right) = \prod_{j=k_0}^k \dfrac{j^{\frac{p}{2}}-(\delta - \lambda)\sqrt{s}}{(j-1)^\frac{p}{2}}.
    \end{align*}
    Taking the logarithm of both sides, we obtain
    \begin{eqnarray*}
        && \log \gamma_k \geq \sum_{j=k_0}^k \log \dfrac{j^{\frac{p}{2}}-(\delta - \lambda)\sqrt{s}}{(j-1)^\frac{p}{2}} = \sum_{j=k_0}^k \left[\log \left( j^\frac{p}{2} + (\delta -\lambda) \sqrt{s} \right) - \log (j-1)^\frac{p}{2} \right]\\
        & =& \sum_{j=k_0}^k \left[ \log \left( 1 + \dfrac{(\delta -\lambda) \sqrt{s}}{j^\frac{p}{2}} \right)  + \log j^\frac{p}{2}  - \log (j-1)^\frac{p}{2}  \right] \\
        &=&
        \log k^\frac{p}{2} - \log (k_0 - 1)^\frac{p}{2} + \sum_{j=k_0}^k  \log \left( 1 + \dfrac{(\delta -\lambda) \sqrt{s}}{j^\frac{p}{2}} \right).
    \end{eqnarray*}
     Using the inequalities $\log(1+x) \geq \dfrac{x}{1+ x} \geq \dfrac{x}{2}$ for all $x \in [0, 1]$ and the fact that, from \eqref{condition:positive_1minusvarepsilon} and $\lambda < \delta$, the term $\frac{(\delta-\lambda)\sqrt{s}}{j^{p/2}} = (\delta-\lambda)\sqrt{s\varepsilon_j} \in [0, 1]$ for all $j \ge k_0$, we arrive at the following estimate
    \begin{align*}
        \log \left( 1 + \dfrac{(\delta -\lambda) \sqrt{s}}{j^\frac{p}{2}} \right) \geq \dfrac{\frac{(\delta -\lambda) \sqrt{s}}{j^\frac{p}{2}}}{1 + \frac{(\delta -\lambda) \sqrt{s}}{j^\frac{p}{2}}} \geq \dfrac{(\delta -\lambda) \sqrt{s}}{2 j^\frac{p}{2}}.
    \end{align*}
    Taking the sum for $j$ from $k_0$ to $k$, we obtain
    \begin{align*}
        \sum_{j=k_0}^k  \log \left( 1 + \dfrac{(\delta -\lambda) \sqrt{s}}{j^\frac{p}{2}} \right) \geq \dfrac{(\delta -\lambda) \sqrt{s}}{2 }\sum_{j=k_0}^k \dfrac{1}{j^\frac{p}{2}}.
    \end{align*}
    Applying Lemma \ref{lemma:sum-integral-inequality} with the function $g(t) = \dfrac{1}{t^\frac{p}{2}}$, we yield
    \begin{align*}
        \sum_{j=k_0}^k  \log \left( 1 + \dfrac{(\delta -\lambda) \sqrt{s}}{j^\frac{p}{2}} \right) & \geq \dfrac{(\delta -\lambda) \sqrt{s}}{2 } \int_{k_0}^{k+1} \dfrac{1}{t^\frac{p}{2}} dt = \dfrac{(\delta - \lambda)\sqrt{s}}{2-p} \left( (k+1)^\frac{2-p}{2} - k_0^\frac{2-p}{2} \right).
    \end{align*}
    Then,
    \begin{align*}
        \log \gamma_k \geq \log k^\frac{p}{2} - \log (k_0 - 1)^\frac{p}{2} + \dfrac{(\delta - \lambda)\sqrt{s}}{2-p} \left( (k+1)^\frac{2-p}{2} - k_0^\frac{2-p}{2} \right).
    \end{align*}
    Taking the exponent of both sides, we obtain
    \begin{align} \label{upper-bound-gamma}
        \gamma_k \geq C k^\frac{p}{2} \mathrm{exp} \left( \dfrac{(\delta - \lambda)\sqrt{s}}{2-p} (k+1)^\frac{2-p}{2}\right)
    \end{align}
    where $C = \left[(k_0 - 1)^\frac{p}{2} \mathrm{exp}\left( \dfrac{(\delta - \lambda)\sqrt{s}}{2-p} k_0^\frac{2-p}{2} \right)\right]^{-1}$. Then we have
    \begin{align*}
        E_{k,1} = \dfrac{\gamma_{k_0} E_{k_0}}{\gamma_{k+1}} = \mathcal{O} \left( (k+1)^{-\frac{p}{2}} \mathrm{exp} \left( - \dfrac{(\delta - \lambda)\sqrt{s}}{2-p} (k+2)^\frac{2-p}{2}\right) \right).
    \end{align*}
    
    We continue to study the behavior of $E_{k,2}$ by proving the existence of limit 
    \begin{equation*} 
        \lim\limits_{k \to +\infty} \frac{E_{k,2}}{(k+1)^{\frac{2+p}{2}}}=\lim\limits_{k \to +\infty} \frac{\frac{\sum_{j=k_0}^k \theta_j \gamma_j}{\gamma_{k+1}}}{(k+1)^{\frac{2+p}{2}}} = \lim\limits_{k\to + \infty} \frac{\sum_{j=k_0}^k  \theta_j \gamma_j}{\frac{\gamma_{k+1}}{(k+1)^\frac{2+p}{2}}}
    \end{equation*}
    with the help of Theorem \ref{theorem:stolz-cesaro} where the sequence $(a_k)_k$ and $(b_k)_k$ are defined as follows:
    \begin{equation*}
        a_k := \sum_{j=k_0}^k\theta_j \gamma_j \textrm{ and } b_k := \frac{\gamma_{k+1}}{(k+1)^\frac{2+p}{2}}.
    \end{equation*}
    
    First, we verify that the sequence $\left(b_k\right)_{k}$ is strictly increasing and unbounded from some index $N$. Indeed, it follows from \eqref{upper-bound-gamma} that $b_k$ is unbounded. To check the monotonicity, we consider the following difference
    \begin{eqnarray*}
        &&b_{k}-b_{k-1} = \dfrac{\gamma_{k+1}}{(k+1)^\frac{2+p}{2}} - \dfrac{\gamma_k}{k^\frac{2+p}{2}} = \dfrac{\gamma_{k}}{(k+1)^\frac{2+p}{2}} \left(1 + \mu_{k+1}  - \left( \dfrac{k+1}{k} \right)^\frac{2+p}{2} \right) \\
        &=& \dfrac{\gamma_{k}}{(k+1)^\frac{2+p}{2}} \left( \left( \dfrac{k+1}{k} \right)^\frac{p}{2} + \left( \dfrac{\delta}{1 - \frac{\delta\sqrt{s}}{k^\frac{p}{2}}} - \lambda \right)\dfrac{\sqrt{s}}{k^\frac{p}{2}}  - \left( \dfrac{k+1}{k} \right)^\frac{2+p}{2} \right) \\
        &=& \dfrac{\gamma_{k}}{(k+1)^\frac{2+p}{2}} \left( -\dfrac{1}{k}\left( \dfrac{k+1}{k} \right)^\frac{p}{2} + \left( \dfrac{\delta}{1 - \frac{\delta\sqrt{s}}{k^\frac{p}{2}}} - \lambda \right)\dfrac{\sqrt{s}}{k^\frac{p}{2}} \right) \\
        & > & \dfrac{\gamma_{k}}{(k+1)^\frac{2+p}{2}k^\frac{p}{2}} \left( -\dfrac{(k+1)^\frac{p}{2}}{k} + \left(\delta - \lambda \right)\sqrt{s}\right).
    \end{eqnarray*}
    Since $0<p<2$, we have $\lim\limits_{k \to+\infty}  \dfrac{(k+1)^\frac{p}{2}}{k} = 0$. Therefore, there exists an integer $N \geq k_0 + 1$ such that $(\delta - \lambda)\sqrt{s} > \dfrac{(k+1)^\frac{p}{2}}{k}, \forall k \geq N$, or equivalently, $b_k > b_{k-1}, \forall k \geq N$. Then the sequence $\left(b_k\right)_{k}$ is strictly increasing for all $k \geq N$. \\
    Next, we calculate the limit
    \begin{equation*}
        \lim\limits_{k \to +\infty} \dfrac{a_{k}-a_{k-1}}{b_{k}-b_{k-1}} = \lim\limits_{k \to +\infty} \dfrac{\theta_k \gamma_k}{\dfrac{\gamma_{k+1}}{(k+1)^{\frac{2+p}{2}}} - \dfrac{\gamma_{k}}{k^{\frac{2+p}{2}}}} = \lim\limits_{k \to +\infty} \dfrac{\theta_k k^{p+1}}{\left( \dfrac{1+ \mu_{k+1}}{(k+1)^\frac{2+p}{2}} - \dfrac{1}{k^\frac{2+p}{2}} \right)k^{p+1}}.
    \end{equation*}
    Similarly for the evaluation of $b_k-b_{k-1}$, the denominator becomes
    \begin{eqnarray*}
        &&\left( \dfrac{1 + \mu_{k+1}}{(k+1)^\frac{2+p}{2}} - \dfrac{1}{k^\frac{2+p}{2}} \right)k^{p+1}  = \dfrac{k^{p+1}}{(k+1)^\frac{2+p}{2}k^\frac{p}{2}} \left( -\dfrac{(k+1)^\frac{p}{2}}{k} + \left(\dfrac{\delta}{1 - \frac{\delta\sqrt{s}}{k^\frac{p}{2}}} - \lambda \right)\sqrt{s}\right)\\
        &=& \left( \dfrac{k}{k+1} \right)^\frac{2+p}{2} \left( -\dfrac{(k+1)^\frac{p}{2}}{k} + \left(\dfrac{\delta}{1 - \frac{\delta\sqrt{s}}{k^\frac{p}{2}}} - \lambda \right)\sqrt{s}\right) \\
        &=& - \dfrac{k^\frac{p}{2}}{k+1} + \left( \dfrac{\delta}{1 - \frac{\delta\sqrt{s}}{k^\frac{p}{2}}} -  \lambda \right) \sqrt{s} \left( \dfrac{k}{k+1} \right)^\frac{2+p}{2} \longrightarrow (\delta - \lambda)\sqrt{s}, \textrm{as } k \longrightarrow +\infty.
    \end{eqnarray*}
    Now we are going to obtain the limit of the numerator. 
    \begin{align*}
        & \lim\limits_{k \to +\infty} \theta_k k^{p+1} \\
        & = \lim\limits_{k \to+\infty} \left[ (a+b) \lambda \sqrt{s} \dfrac{(\dot{\varepsilon}_{k-1})^2}{k^\frac{p}{2}} (k-1)^{2p} k^{p+1} - \dfrac{(1+q)s \dot{\varepsilon}_{k} k^{p+1}}{(1 - \delta \sqrt{s \varepsilon_{k}})(1 -s\varepsilon_{k})}(1 + \mu_{k+1}) \right]\\
        &= \lim_{k \to +\infty} \left[ (a+b)\lambda \sqrt{s} \left( \dot{\varepsilon}_{k-1} k^{p+1} \right)^2 \left(1-\dfrac{1}{k}\right)^{2p}k^\frac{p-2}{2} -\dfrac{(1+q)s \dot{\varepsilon}_{k} k^{p+1}}{(1 - \delta \sqrt{s \varepsilon_{k}})(1 -s\varepsilon_{k})}(1 + \mu_{k+1})\right] \\
        & = (1+q)ps.
    \end{align*}
Here 
\begin{eqnarray*}
    &&\lim\limits_{k \to +\infty} (1 + \mu_{k+1}) = \lim\limits_{k \to +\infty} \left[ \left( \dfrac{k+1}{k} \right)^\frac{p}{2}  + \left( \dfrac{\delta}{1 - \delta\sqrt{s \varepsilon_k}} - \lambda \right)\sqrt{s \varepsilon_k} \right]= 1, \\
    && \lim\limits_{k \to+\infty} \dot{\varepsilon}_{k-1} k^{p+1} = \lim\limits_{k \to+\infty} \left(\dfrac{1}{k^p} - \dfrac{1}{(k-1)^{p}}\right) k^{p+1} = \lim\limits_{k \to+\infty} \left(1-\left(\frac{k}{k-1}\right)^p\right)k\\
    &=& \lim\limits_{k \to+\infty} \dfrac{1-\left(1-\frac{1}{k}\right)^{-p}}{\frac{1}{k}}= \lim\limits_{k \to+\infty} -p\left(1-\frac{1}{k}\right)^{-p-1} = -p ~\textrm{(using L'Hôpital's rule)},\\
    && \lim\limits_{k \to+\infty} \dot{\varepsilon}_{k} k^{p+1} = \lim\limits_{k \to+\infty} \left(\dfrac{1}{{k+1}^p} - \dfrac{1}{k^{p}}\right) k^{p+1} = \lim\limits_{k \to+\infty} \left(\left(\frac{k}{k+1}\right)^p-1\right)k\\
    &=& \lim\limits_{k \to+\infty} \dfrac{\left(1+\frac{1}{k}\right)^{-p}-1}{\frac{1}{k}}= \lim\limits_{k \to+\infty} -p\left(1+\frac{1}{k}\right)^{-p-1} = -p.\\
\end{eqnarray*}
Consequently, the limit $\lim\limits_{k \to +\infty} \dfrac{a_{k}-a_{k-1}}{b_{k}-b_{k-1}} = \dfrac{(1+q)p\sqrt{s}}{\delta - \lambda}.$ \\
Finally, according to Theorem \ref{theorem:stolz-cesaro}, we get
    \begin{align*}
        \lim\limits_{k \to +\infty} \dfrac{a_{k}}{b_{k}} = \lim\limits_{k\to + \infty} \dfrac{\sum_{j=k_0}^k  \theta_j \gamma_j}{\frac{\gamma_{k+1}}{(k+1)^\frac{2+p}{2}}} = \dfrac{(1+q)p\sqrt{s}}{\delta - \lambda}.
    \end{align*}
    Then, by Lemma \ref{lemma:big-o}, we obtain $\dfrac{\sum_{j=k_0}^k \theta_j \gamma_j}{\gamma_{k+1}} = \mathcal{O}\left( \dfrac{1}{k^\frac{2+p}{2}}\right)$ which means that $E_{k,2} = \mathcal{O}\left( \dfrac{1}{k^\frac{2+p}{2}}\right)$. Since the first term $E_{k,1}$ decays to zero exponentially, we deduce that 
    \begin{align*}
        E_k = \mathcal{O}\left( \dfrac{1}{k^\frac{2+p}{2}} \right) \text{ as } k \to +\infty.
    \end{align*}

      \def\labelenumi{\rm (\roman{enumi})}
\def\theenumi{\roman{enumi}}
    \begin{enumerate}
        \item \textbf{Convergence rate of the values:} Thanks to Lemma \ref{lemma:convergence_rate_grad_algo}.(i), we have
        \begin{align*}
            f(x_k) - \underset{\mathcal{H}}{\min} f \leq \dfrac{E_k}{\alpha_k} + \dfrac{1}{2k^p}\|x^*\|^2.
        \end{align*}
        Since $\alpha_k = \dfrac{(1 + q)s}{(1 - \delta \sqrt{s \varepsilon_k})(1 - s\varepsilon_k)} > s$ for all $k \geq 0$, we obtain
        \begin{align*}
            \dfrac{E_k}{\alpha_k} < \dfrac{E_k}{s} = \mathcal{O} \left( \dfrac{1}{k^\frac{2+p}{2}} \right).
        \end{align*}
        Since $0<p<2$, we have $p < \dfrac{2+p}{2}$. We conclude that
        \begin{align*}
            f(x_k) - \underset{\mathcal{H}}{\min} f = \mathcal{O}\left( \dfrac{1}{k^p} \right) \text{ as } k \to +\infty.
        \end{align*}

        \item \textbf{Convergence rate to zero of $\|x_k - x_{\varepsilon_k}\|$}: Using Lemma \ref{lemma:convergence_rate_grad_algo}.(i), we have 
        \begin{align*}
            \|x_k - x_{\varepsilon_k}\|^2 \leq \dfrac{2E_k}{\alpha_k \varepsilon_k} < \dfrac{2k^p E_k}{s}.
        \end{align*}
        Since $E_k = \mathcal{O}\left( \dfrac{1}{k^\frac{2+p}{2}} \right)$, we obtain
        \begin{align*}
            \|x_k - x_{\varepsilon_k}\|^2 = \mathcal{O} \left( \dfrac{1}{k^\frac{2-p}{2}} \right).
        \end{align*}
        Since $\lim\limits_{k \to +\infty} \|x_{\varepsilon_k} - x^*\| = 0$ (see Lemma~\ref{lemma:visco-curve}(ii)), we deduce that
        \begin{align*}
            0 \leq \lim\limits_{k \to +\infty} \|x_k - x^*\| \leq \lim\limits_{k \to +\infty} ( \|x_k - x_{\varepsilon_k}\| + \|x_{\varepsilon_k} - x^*\| ) = 0.
        \end{align*}
        Then $(x_k)_k$ strongly converges to the minimum-norm solution $x^*$.

        \item \textbf{Convergence rate to zero of the velocity $\|x_k - x_{k-1}\|$}: We get from Lemma \ref{lemma:convergence_rate_grad_algo}.(ii)
        \begin{align*}
            \|x_k - x_{k-1}\|^2 \leq \dfrac{2E_k}{\beta_k}.
        \end{align*}
        By the definition of $\beta_k = q ( 1- \delta \sqrt{s \varepsilon_k})$, there exists $N \in \mathbb{N}$ such that $\beta_k \geq \dfrac{q}{2}$ for all $k \geq N$. Therefore, for all $k \geq N$,
        \begin{align*}
            \|x_k - x_{k-1}\|^2 \leq \dfrac{4E_k}{q} = \mathcal{O}\left( \dfrac{1}{k^\frac{2+p}{2}} \right).
        \end{align*}
        We deduce that $\|x_k - x_{k-1}\| = \mathcal{O}\left( \dfrac{1}{k^\frac{2+p}{4}} \right)$ as $k \to +\infty$. This completes the proof.
    \end{enumerate}
\end{proof}

\begin{remark}~ \label{remark:particular_case_plessthan2}
\def\labelenumi{\rm (\roman{enumi})}
\def\theenumi{\roman{enumi}}
\begin{enumerate}
    \item The rates in Theorem~\ref{theorem:conv_rate_grad_algo} for our discretization are consistent with the convergence results for the continuous dynamics \eqref{trigs} proved in \cite[Theorem~5]{Attouch:jde:22}.

    \item The rates of Algorithm~\ref{algo:triga} are comparable to those obtained for Algorithm~\ref{algo:nadtr} in \cite[Theorem~9]{Karapetyants:amo:24}. In contrast to Algorithm~\ref{algo:nadtr}, our scheme guarantees strong convergence of $(x_k)_k$ while using only one Tikhonov regularization term. This may reduce both computational time and the number of iterations, as illustrated in Section~\ref{sec:numerical}.

    \item We obtain strong convergence of $(x_k)_k$ to the minimum-norm solution $x^*$. However, we only derive that $\|x_k-x_{\varepsilon_k}\|\to 0$. To obtain an explicit rate for $\|x_k-x^*\|$, one needs quantitative information on the convergence of the discrete viscosity curve $(x_{\varepsilon_k})_k$ toward $x^*$.

    \item As in the continuous setting, the choice $p=\frac{2}{3}$ (i.e., $\frac{p}{2}=\frac{2+p}{4}$) provides a balanced trade-off between the decay of the values $f(x_k)-\underset{\mathcal{H}}{\min} f$ and the strong convergence of $(x_k)_k$. Numerical results in Section~\ref{sec:numerical} further inform practical choices of $p$.

    \item In the proof of Theorem~\ref{theorem:conv_rate_grad_algo}, we use $(\delta-\lambda)\sqrt{s}>\frac{(k+1)^{p/2}}{k}$ for $k\ge N\ge 1$, which implies $\delta>\frac{(k+1)^{p/2}}{k\sqrt{s}}$ for $k\ge N\ge 1$. Since $\left(\frac{(k+1)^{p/2}}{k\sqrt{s}}\right)_k$ is strictly decreasing, one may take in the numerical experiments
    \[
        \delta:=\frac{(1+1)^{p/2}}{1\sqrt{s}}=\frac{2^{p/2}}{\sqrt{s}}.
    \]
\end{enumerate}
\end{remark}

\subsection{\texorpdfstring{$\varepsilon_k$}{TEXT} of order \texorpdfstring{$\dfrac{1}{k^2}$}{TEXT}}

Let $\varepsilon_k := \dfrac{c}{k^2}$. We choose $c$ so that \eqref{condition:tikhonov-sequence} holds. Fix $a,b,q,\lambda,s$ satisfying \eqref{condition:constant-parameter}. Then, for any $c>0$, conditions \eqref{condition:tikhonov-sequence}$(ii)$ and \eqref{condition:tikhonov-sequence}$(iv)$ hold for all sufficiently large $k$.

Condition \eqref{condition:tikhonov-sequence}$(i)$ reads
\begin{eqnarray}
    && (1+q) \left( \dfrac{k+1}{\sqrt{sc}} - \dfrac{k}{\sqrt{sc}} + \left( \dfrac{\delta}{1 - \frac{\delta\sqrt{sc}}{k}} - \lambda \right) \right) - \lambda \dfrac{k}{k+1} \leq 0 \notag \\
    &\Longleftrightarrow& \dfrac{1}{\sqrt{sc}} \leq \left( 1 + \dfrac{k}{(1+q)(k+1)} \right) \lambda - \dfrac{\delta}{1 - \frac{\delta\sqrt{sc}}{k}}. \label{ineq:bound_sc}
\end{eqnarray}
Since
$\lim\limits_{k \to +\infty} \left( \left( 1 + \dfrac{k}{(1+q)(k+1)} \right) \lambda - \dfrac{\delta}{1 - \frac{\delta\sqrt{sc}}{k}} \right) \lambda - \delta  =  \dfrac{2 + q}{1 + q} \lambda - \delta > 0$ (see \eqref{condition:constant-parameter}$(i)$), it suffices to impose
\begin{align*}
    \dfrac{1}{\sqrt{sc}} < \dfrac{2 + q}{1 + q} \lambda - \delta.
\end{align*}
Then there exists $N\in\mathbb{N}$ such that \eqref{ineq:bound_sc} holds for all $k\ge N$.

For \eqref{condition:tikhonov-sequence}$(iii)$, we obtain
\begin{align*}
    & \dfrac{(1 + a(1-q)) \lambda}{a} - \delta+ (1+q) \left( \dfrac{k+1}{\sqrt{sc}} - \dfrac{k}{\sqrt{sc}} \right) + \dfrac{q\delta^2 \frac{\sqrt{sc}}{k+1}}{1 - \frac{\delta \sqrt{sc}}{k}} \leq 0 \\
    \Longleftrightarrow & 
    \dfrac{(1 + a(1-q)) \lambda}{a} - \delta + \dfrac{1+q}{\sqrt{sc}} + \dfrac{q\delta^2 \sqrt{sc}}{(k+1)(1 - \frac{\delta \sqrt{sc}}{k})} \leq 0 \\
    \Longleftrightarrow & \dfrac{1+q}{\sqrt{sc}} \leq \delta - \dfrac{(1 + a(1-q)) \lambda}{a}  - \dfrac{q\delta^2 \sqrt{sc}}{(k+1)(1 - \frac{\delta \sqrt{sc}}{k})}.
\end{align*}
Since $\lim_{k \to +\infty} \dfrac{q\delta^2 \sqrt{sc}}{(k+1)(1 - \frac{\delta \sqrt{sc}}{k})} = 0$, it suffices to choose $c$ such that
$$\dfrac{1+q}{\sqrt{sc}} < \delta - \dfrac{(1 + a(1-q)) \lambda}{a} \Leftrightarrow \dfrac{1}{\sqrt{sc}} < \dfrac{1}{1 + q} \left( \delta - \dfrac{(1 + a(1-q)) \lambda}{a} \right).$$

\begin{theorem} \label{theorem:conv_rate_grad_algo_k^2}
    Take $\varepsilon_k := \dfrac{c}{k^2}$ where $c$ satisfies
    \begin{align*}
        \dfrac{1}{\sqrt{sc}} < \min \left\{ \dfrac{2+q}{1+q} \lambda -\delta, \dfrac{1}{1 + q} \left( \delta - \dfrac{(1 + a(1-q)) \lambda}{a} \right), \delta - \lambda \right\}.
    \end{align*}
    Suppose $f \colon \mathcal{H} \to \mathbb{R}$ is a convex differentiable function whose gradient $\nabla f$ is $L$-Lipschitz continuous. Let $(x_k)_k$ be a sequence generated by Algorithm \ref{algo:triga}. We have the following convergence rate of the values and the discrete velocity:
    \begin{align*}
        f(x_k) - \underset{\mathcal{H}}{\min} f = \mathcal{O}\left( \dfrac{1}{k^2} \right), ~ \|x_k - x_{k-1}\| = \mathcal{O} \left( \dfrac{1}{k} \right).
    \end{align*}
\end{theorem}
\begin{proof}
    Similarly to the proof of Theorem \ref{theorem:conv_rate_grad_algo}, firstly, we estimate the lower bound of $\log\gamma_k$.
    From now on, we assume $k_0\ge 2$ (so that all denominators involving $(j-1)$ are well-defined), and this does not affect the asymptotic estimates.
    \begin{align*}
        \log \gamma_k &\geq \sum_{j=k_0}^k \log \left( 1 + \dfrac{1 + (\delta - \lambda)\sqrt{sc}}{j - 1} \right) \geq \sum_{j=k_0}^k \dfrac{\frac{1 + (\delta - \lambda)\sqrt{sc}}{j - 1}}{1 + \frac{1 + (\delta - \lambda)\sqrt{sc}}{j - 1}} = \sum_{j=k_0}^k \dfrac{1 + (\delta - \lambda)\sqrt{sc}}{j +  (\delta - \lambda)\sqrt{sc}}.
    \end{align*}
    Using Lemma \ref{lemma:sum-integral-inequality} for the function $g(t) = \dfrac{1}{t + (\delta - \lambda)\sqrt{sc}}$, we obtain
    \begin{align*}
        \log \gamma_k \geq (1 + (\delta - \lambda)\sqrt{sc}) \left( \log(k+1 + (\delta - \lambda)\sqrt{sc}) - \log(k_0 + (\delta - \lambda)\sqrt{sc})\right).
    \end{align*}
    By taking the exponent of both sides and recalling from the choice of $c$ that $(\delta - \lambda) \sqrt{sc} > 1$, it follows that
    \begin{align*}
        \gamma_k \geq C_1 (k+1 + (\delta - \lambda)\sqrt{sc})^{1 + (\delta - \lambda)\sqrt{sc}} > C_1 k^2
    \end{align*}
    where $C_1 := (k_0 + (\delta - \lambda)\sqrt{sc})^{-(1 + (\delta - \lambda)\sqrt{sc})}$. 
    Consequently, $$\dfrac{\gamma_{k_0}E_{k_0}}{\gamma_{k+1}} \leq \dfrac{\gamma_{k_0}E_{k_0}}{C_1(k+1)^2}= \mathcal{O}\left( \dfrac{1}{k^2} \right).$$   
    We now aim to show that the limit $\lim\limits_{k \to +\infty} \dfrac{\sum_{j=k_0}^k \theta_j \gamma_j}{\dfrac{\gamma_{k+1}}{(k+1)^2}}$ exists. Since $(\delta - \lambda)\sqrt{sc} > 1$, we have
    \begin{align*}
        \dfrac{\gamma_k}{k^2} \geq C_1 (1+1/k + (\delta - \lambda)\sqrt{sc}/k)^{1 + (\delta - \lambda)\sqrt{sc}}k^{(\delta - \lambda)\sqrt{sc}-1}  \to +\infty \text{ as } k \to +\infty.
    \end{align*}
    Next, we calculate the difference:
    \begin{align*}
        \dfrac{\gamma_{k+1}}{(k+1)^2} - \dfrac{\gamma_{k}}{k^2} > \dfrac{\gamma_{k}}{(k+1)^2k} \left(\left(\delta - \lambda \right)\sqrt{sc} -\dfrac{(k+1)}{k} \right) = \dfrac{\gamma_{k}}{(k+1)^2k} \left( \left(\delta - \lambda \right)\sqrt{sc} - 1 -\dfrac{1}{k}\right).
    \end{align*}
    There exists $N \in \mathbb{N}$ such that 
    \begin{align*}
        (\delta - \lambda) \sqrt{sc} - 1 \geq \dfrac{1}{k}, \forall k \geq N.
    \end{align*}
    Therefore, we deduce that the sequence $\left( \dfrac{\gamma_k}{k^2} \right)_k$ is strictly increasing for all $k \geq N$. \\
    Applying Theorem \ref{theorem:stolz-cesaro}, we establish that the limit $\lim\limits_{k \to +\infty} \dfrac{\sum_{j=k_0}^k \theta_j \gamma_j}{\dfrac{\gamma_{k+1}}{(k+1)^2}}$ exists. Combined with Lemma \ref{lemma:big-o}, this existence ensures the asymptotic bound
    \begin{align*}
        \dfrac{\sum_{j=k_0}^k \theta_j \gamma_j}{\gamma_{k+1}} = \mathcal{O} \left( \dfrac{1}{k^2} \right).
    \end{align*}
    We invoke Theorem \ref{theorem:grad-algo} to obtain
    \begin{align*}
        E_{k+1} \leq \dfrac{\gamma_{k_0}E_{k_0}}{\gamma_{k+1}} + \dfrac{\sum_{j=k_0}^k \gamma_j \theta_j}{2 \gamma_{k+1}} \|x^*\|^2 = \mathcal{O}\left( \dfrac{1}{k^2} \right).
    \end{align*}
    Consequently, we conclude that
    \begin{align*}
        f(x_k) - \underset{\mathcal{H}}{\min} f = \mathcal{O}\left( \dfrac{1}{k^2} \right), ~ \|x_k - x_{k-1}\| = \mathcal{O} \left( \dfrac{1}{k} \right).
    \end{align*}
\end{proof}

\begin{remark}
    In the case $\varepsilon_k = \dfrac{c}{k^2}$, the convergence rate of the values $f(x_k) - \underset{\mathcal{H}}{\min} f$ and the discrete velocity $\|x_k - x_{k-1}\|$ is as good as those obtained by Nesterov's accelerated gradient method. However, unlike the case with $p \in ]0,2[$, the strong convergence of the iterates toward the minimum-norm solution is not guaranteed. 
\end{remark}

\section{Numerical experiments} \label{sec:numerical}

Our numerical experiments aim to illustrate the performance of our proposed Algorithm \ref{algo:triga} and compare with Algorithm \ref{algo:nadtr} in the very-related work \cite{Karapetyants:amo:24}. Particularly, we carry out a series of experiments on three test problems: a simple quadratic problem from \cite{Attouch:jde:22}, the linear least-squares problem, and the logistic regression.  All computations are executed in MATLAB 2024b on a MacBook Air running macOS Sequoia $15.7$, equipped with an Apple M2 processor and 16GB of memory.

\bigskip

\noindent \textbf{Computational settings}: Stopping criteria of the comparable algorithms are defined as either reaching a maximum of $10^5$ iterations or attaining a gradient norm of the objective function $\|\nabla f(x_k)\|$ smaller than $10^{-6}$, unless stated otherwise. Since both algorithms \ref{algo:triga} and \ref{algo:nadtr} have the same step of updating $x_{k+1}$, for a fair comparison, we set the same choices of the step-size $s$. As discussed in Remark~\ref{remark:choices_parameter}, we choose here two values of $s$, namely $s = \frac{1}{1.1L}$ and $s = \frac{1}{2.1L}$, which correspond to $q > 10 > 1$ and $q < \frac{10}{11} < 1$, respectively. As for the Tikhonov regularization parameters in \ref{algo:nadtr}, we follow the set-up in \cite{Karapetyants:amo:24} as follows: $a=c=1$, $q = 0.99$, $p \in \{0.3, 0.6, 0.9, 1.2, 1.5, 1.95\}$, meanwhile for our algorithm, we choose $p \in \{0.3, 0.6, 0.9, 1.2, 1.5, 1.95, 1.99\}$. According to Remark~\ref{remark:particular_case_plessthan2}, we set $\delta := \frac{2^\frac{p}{2}}{\sqrt{s}}$.

\bigskip

\noindent \textbf{Comparison criteria}: We compare the performance of these algorithms using the performance profiles, developed by Dolan and Moré \cite{Dolan:mp:01}. For each value of $t \in \mathbb{R}$, the performance profile reports the proportion $\rho_s(t)$ of problems for which the performance ratio of the solver $s$ is not greater than the factor $t$. In particular, given a set of problems $P := \{1,\ldots,n_p\}$ and a set of solvers $S := \{1,\ldots,n_s\}$, the performance ratio of solver $s \in S$ for problem $p \in P$ is defined by
$$r_{p,s} := \dfrac{t_{p,s}}{\min\{ t_{p,s'} \colon s' \in S\}},$$
where $t_{p,s}$ is the performance measure of solver $s$ for problem $p$ (for example, CPU time or number of iterations). 
Then, the performance of solver $s \in S$ is given by
$$\rho_{s}(t) := \dfrac{1}{n_p}\mathrm{size} \{p \in P \colon \log_2(r_{p,s}) \leq t \}.$$
For a given value of the factor $t$, a solver $s$ with a higher value of $\rho_{s}(t)$ is considered more efficient, as it successfully solves a larger number of problems within the same factor compared to the other solvers. The performance $\rho_{s}(t) = 1$ means that solver $s$ successfully handles all tested problems within the given performance factor $t$. The comparable algorithms are evaluated throughout our experiments according to the following criteria: objective function value $f(x_k) - \min_{\mathcal{H}} f$,  gradient norm $\Vert \nabla f(x_k) \Vert$, strong convergence measure $\Vert x_k - x^* \Vert$, discrete velocity $\Vert x_k - x_{k-1}\Vert$, the number of iterations, and CPU time (in seconds).

\subsection{Problem 1 : Simple quadratic problem} \label{sec:test_problem_simple}

We consider a simple quadratic problem with the aim of  examining the strong convergence behavior of Algorithm \ref{algo:triga} and facilitating a fair comparison to algorithm \ref{algo:nadtr} based on the aforementioned performance criteria. In particular, we test on the problem of minimizing the following quadratic function $f \colon \mathbb{R}^{2n} \to \mathbb{R}$ introduced in \cite{Attouch:jde:22}:
$$
 f(x_1, \ldots, x_{2n}) = \dfrac{1}{2} \sum_{i=1}^{n} (x_{2i-1} + x_{2i}-1)^2.
$$
 The function $f$ is convex, but not strongly convex. The solution set $\mathcal{S}$ is the affine subspace $\{x \in \mathbb{R}^{2n} \colon x_{2i-1} + x_{2i} - 1 = 0, \text{ for all } i=1, \ldots, n \}$ and its unique minimum-norm element is $x^* = \left( \frac{1}{2}, \ldots, \frac{1}{2} \right)$. 
  
In order to illustrate the strong convergence of the sequence $(x_k)_k$ generated by \ref{algo:triga} to the minimum-norm solution $x^*$, we plot its trajectory in the $2$-dimensional space $\mathbb{R}^2$ ($n=1$). In addition, we also display the trajectory of \ref{algo:nag} for comparison. It should be noted that \ref{algo:nag} does not guarantee both the fast convergence rate of the values and the strong convergence of the iterates toward the minimum-norm solution $x^*$. Both algorithms initialized from the same starting point. Figure \ref{fig:trajectory-2-d} shows the trajectories of \ref{algo:nag} and \ref{algo:triga}. Unsurprisingly, \ref{algo:nag} needs $2$ iterations to reach an arbitrary solution in $\mathcal{S}$ which is totally far from $x^*$. On the other hand, \ref{algo:triga} require more iterations, however, all generated sequences $(x_k)_k$ converge toward the point $x^*$. It is therefore reasonable to not compare \ref{algo:triga} and \ref{algo:nadtr} with \ref{algo:nag} in the remaining tests. Among difference values of $p$, our algorithm works better with $p \in \{1.95, 1.99\}$ when approaching faster to $x^*$. 

\begin{figure}[htbp]
\centering
\includegraphics[width=0.6\linewidth]{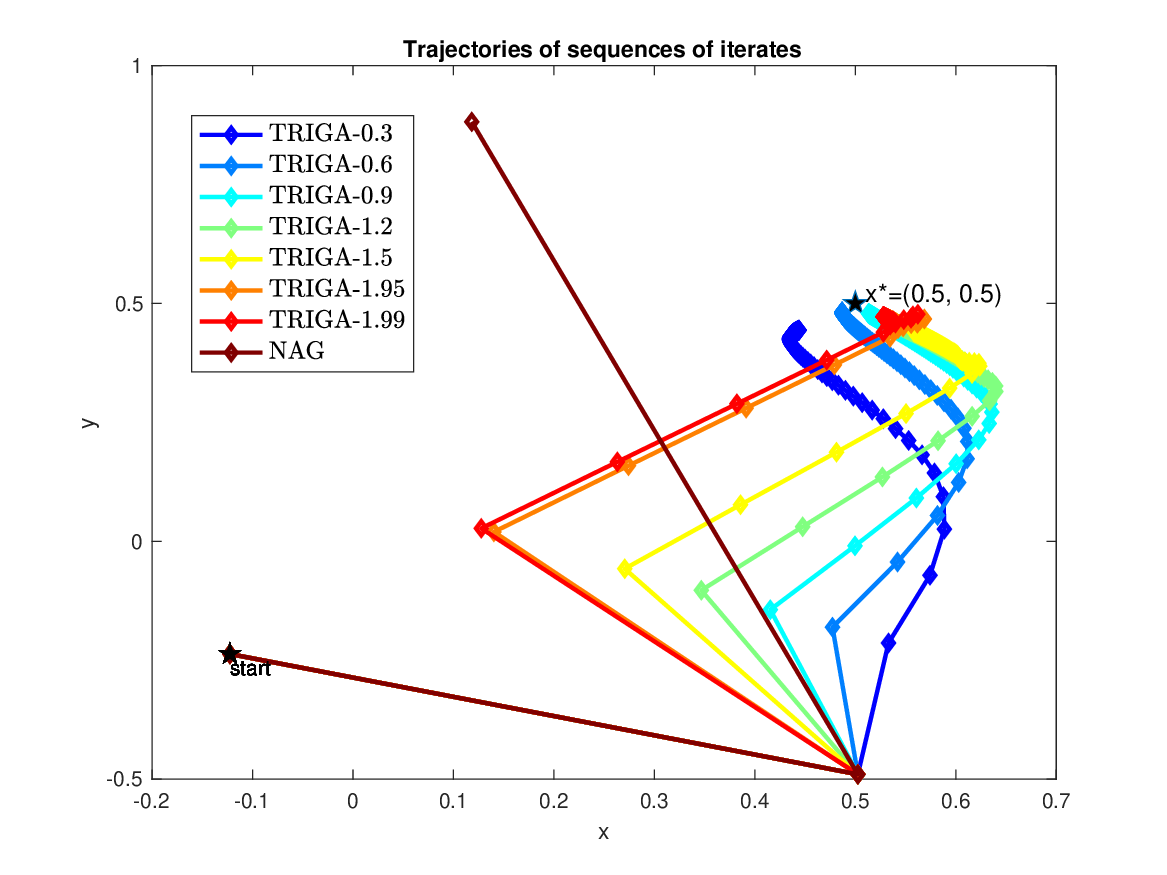}
\caption{Strong convergence to the minimum-norm solution $x^*=(0.5,0.5)$ in $\mathbb{R}^2$}
\label{fig:trajectory-2-d}
\end{figure}

We run \ref{algo:triga} and \ref{algo:nadtr} on $\mathbb{R}^{20}$ ($n=10$) with different choices of $p$ and the maximum of 100 iterations. The results in objective function value $f(x_k) - \min_{\mathcal{H}} f$,  gradient norm $\Vert \nabla f(x_k) \Vert$, strong convergence measure $\Vert x_k - x^* \Vert$, and discrete velocity $\Vert x_k - x_{k-1}\Vert$ with the step-size $s=\frac{1}{1.1L}$ (resp. $s=\frac{1}{2.1L}$) are shown in Figure \ref{fig:simple_large_s} (resp. Figure \ref{fig:simple_small_s}). 

\begin{figure}[htbp]
\centering
\includegraphics[width=0.46\linewidth]{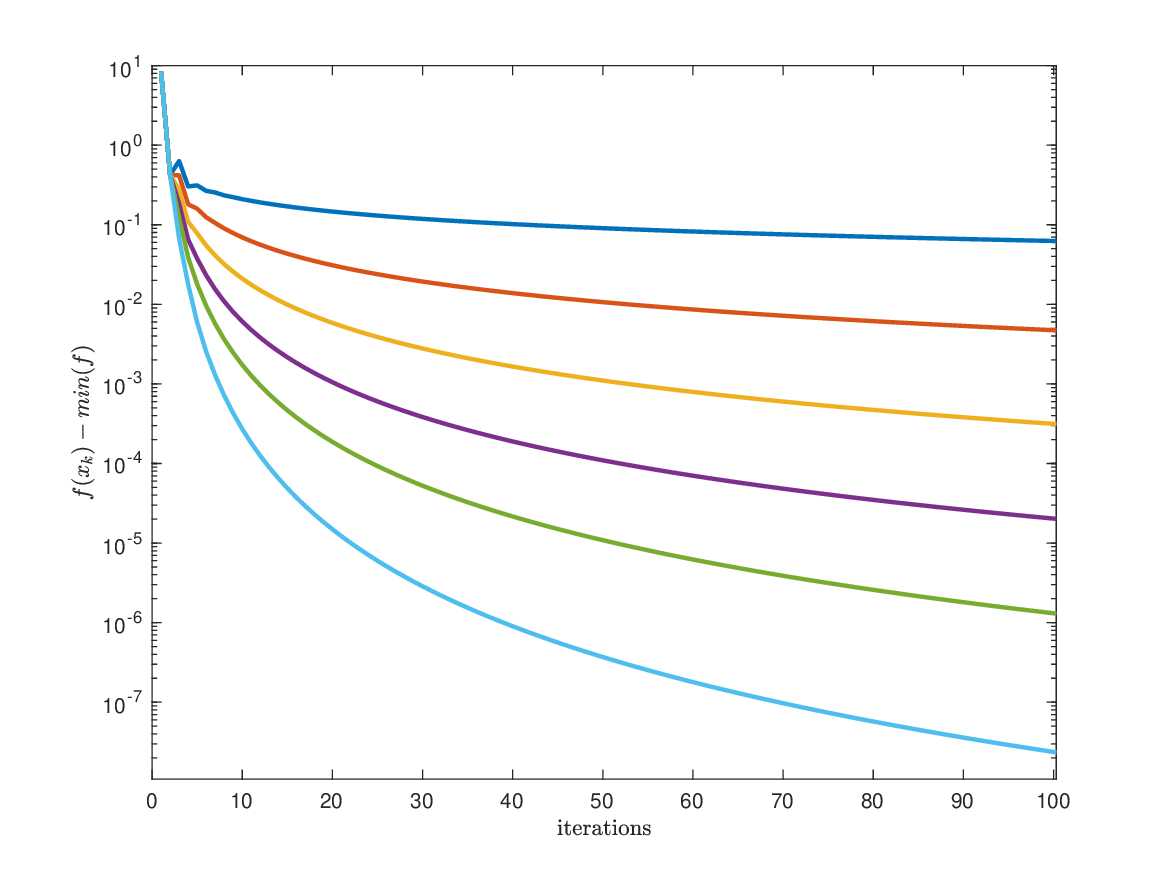}
\includegraphics[width=0.46\linewidth]{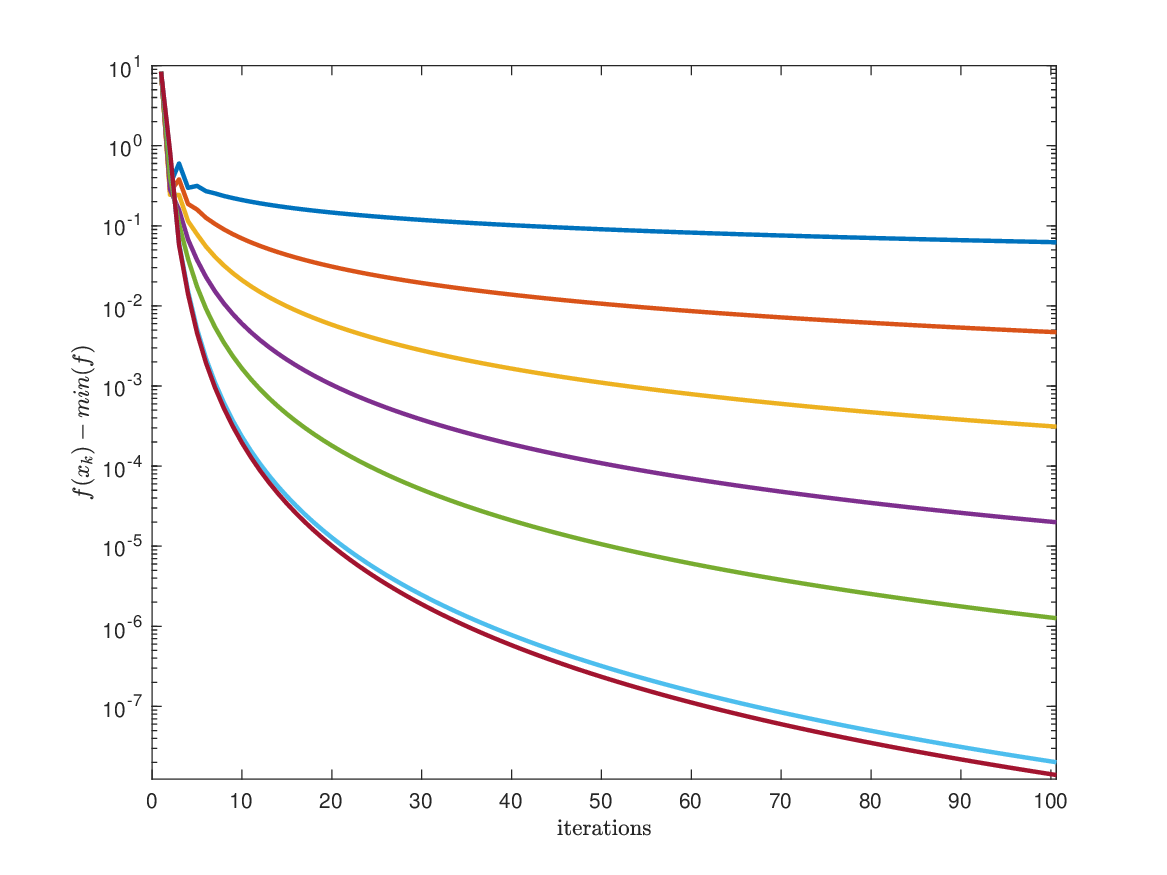}
\includegraphics[width=0.46\linewidth]{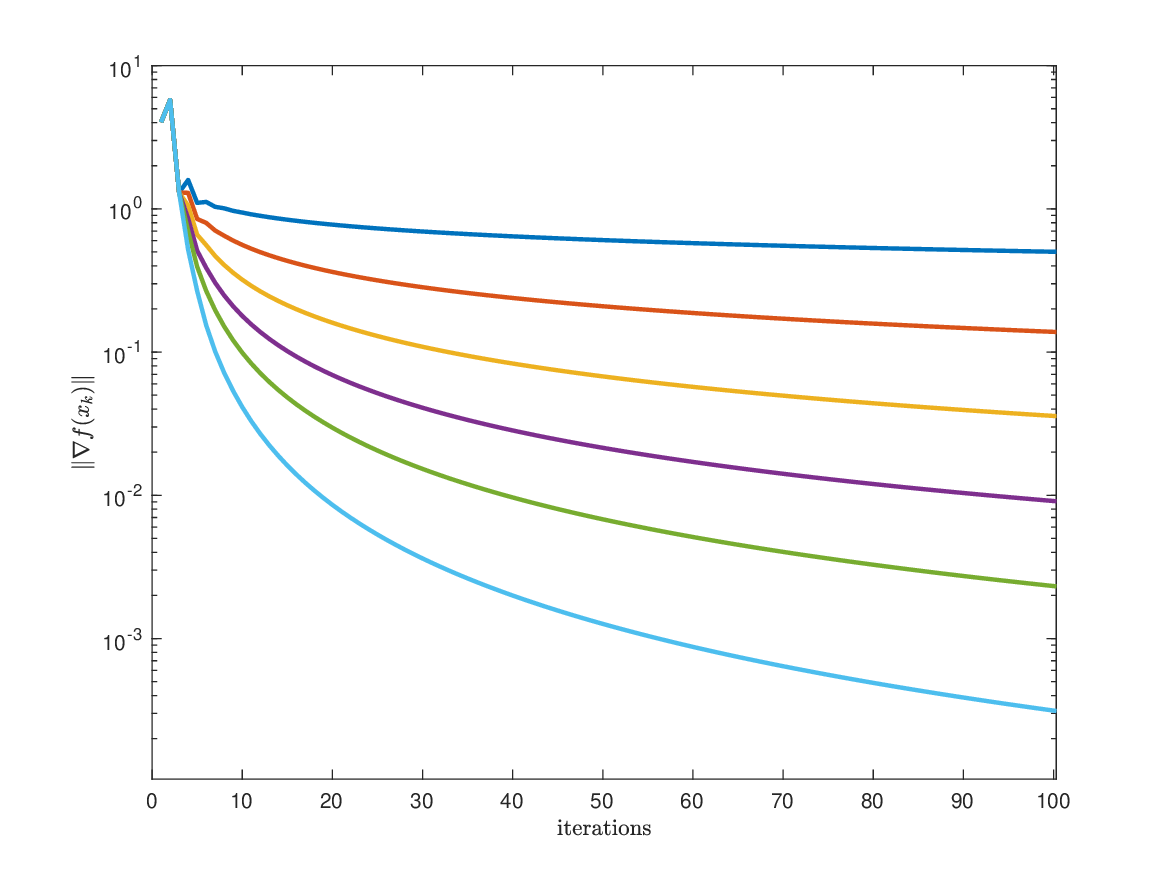}
\includegraphics[width=0.46\linewidth]{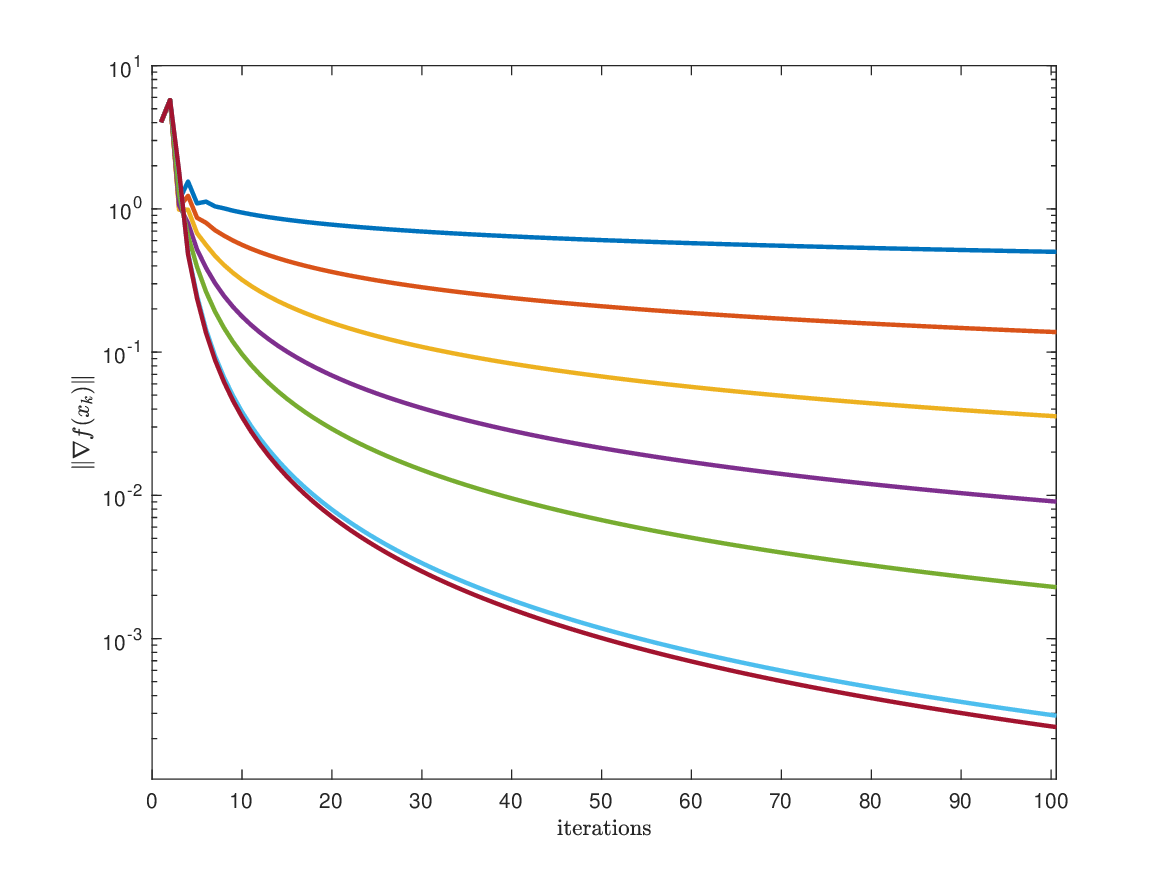}
\includegraphics[width=0.46\linewidth]{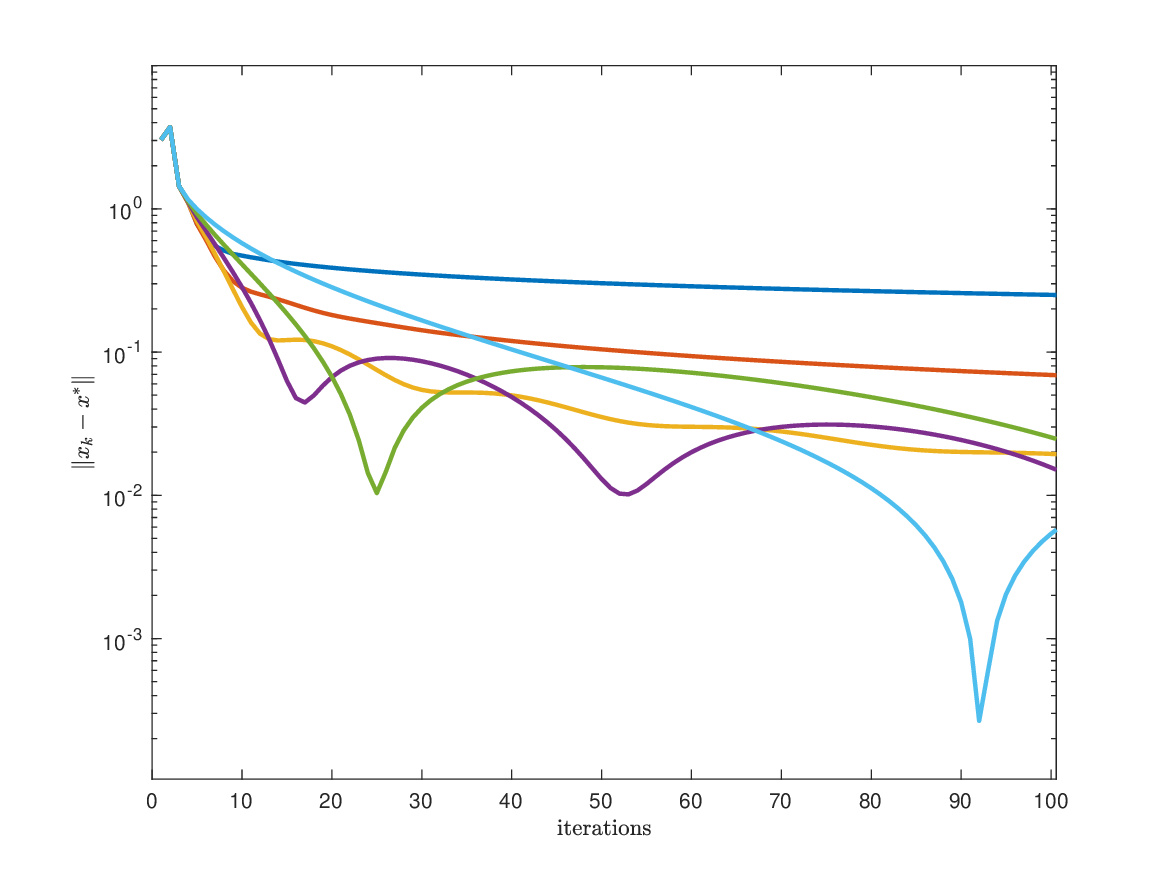}
\includegraphics[width=0.46\linewidth]{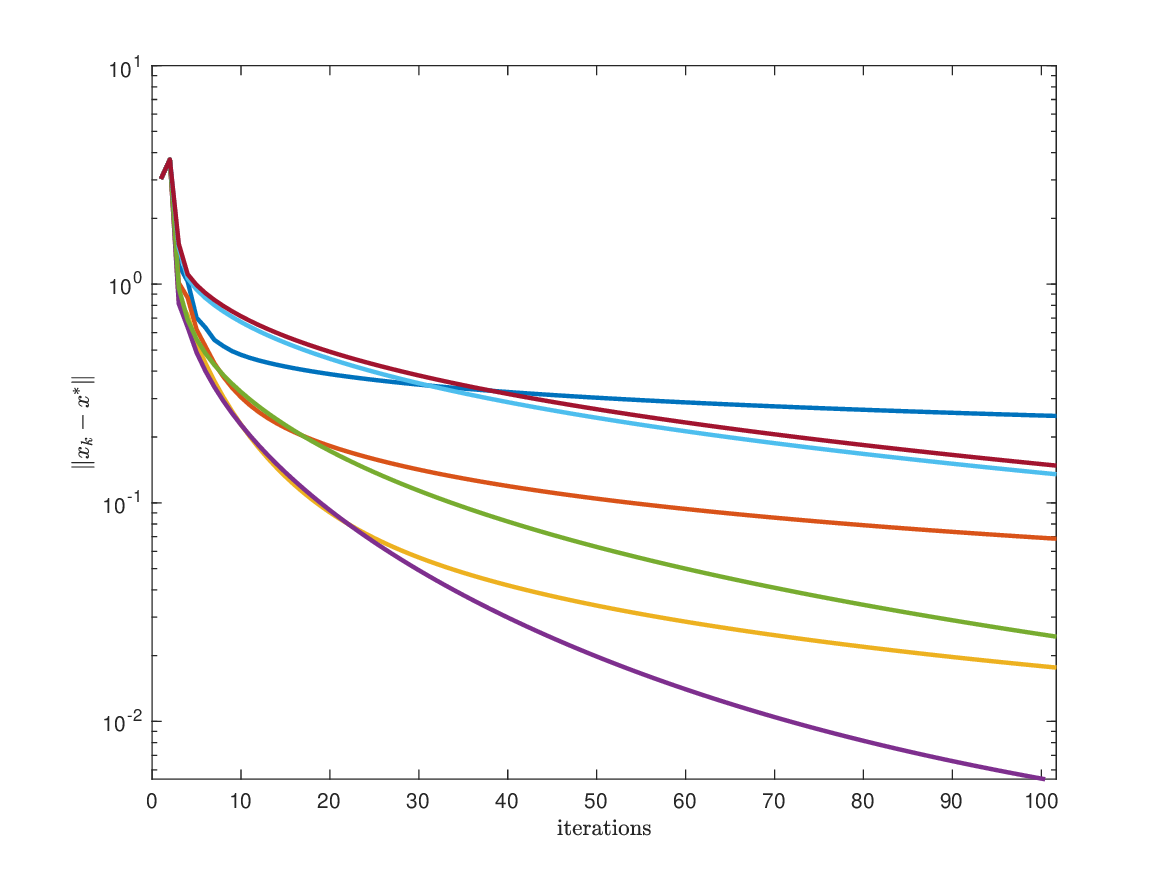}
\includegraphics[width=0.46\linewidth]{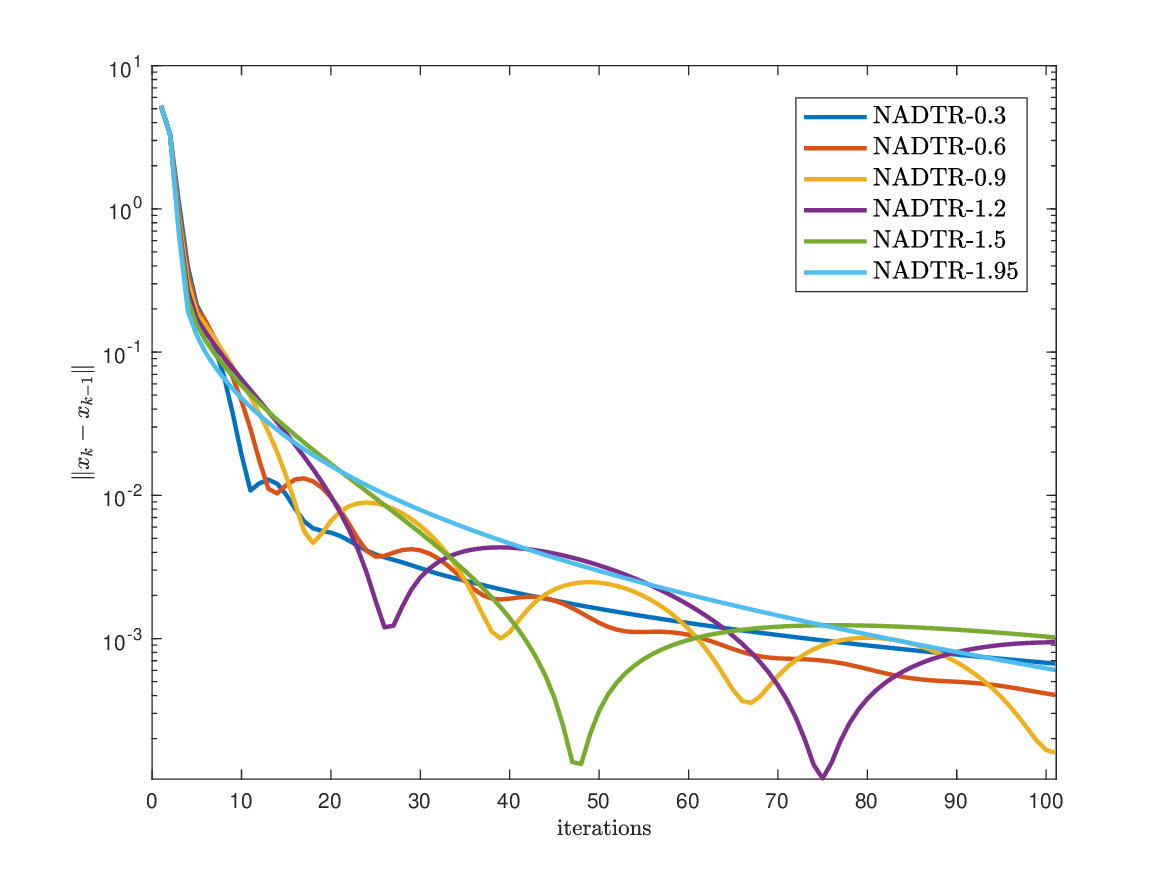}
\includegraphics[width=0.46\linewidth]{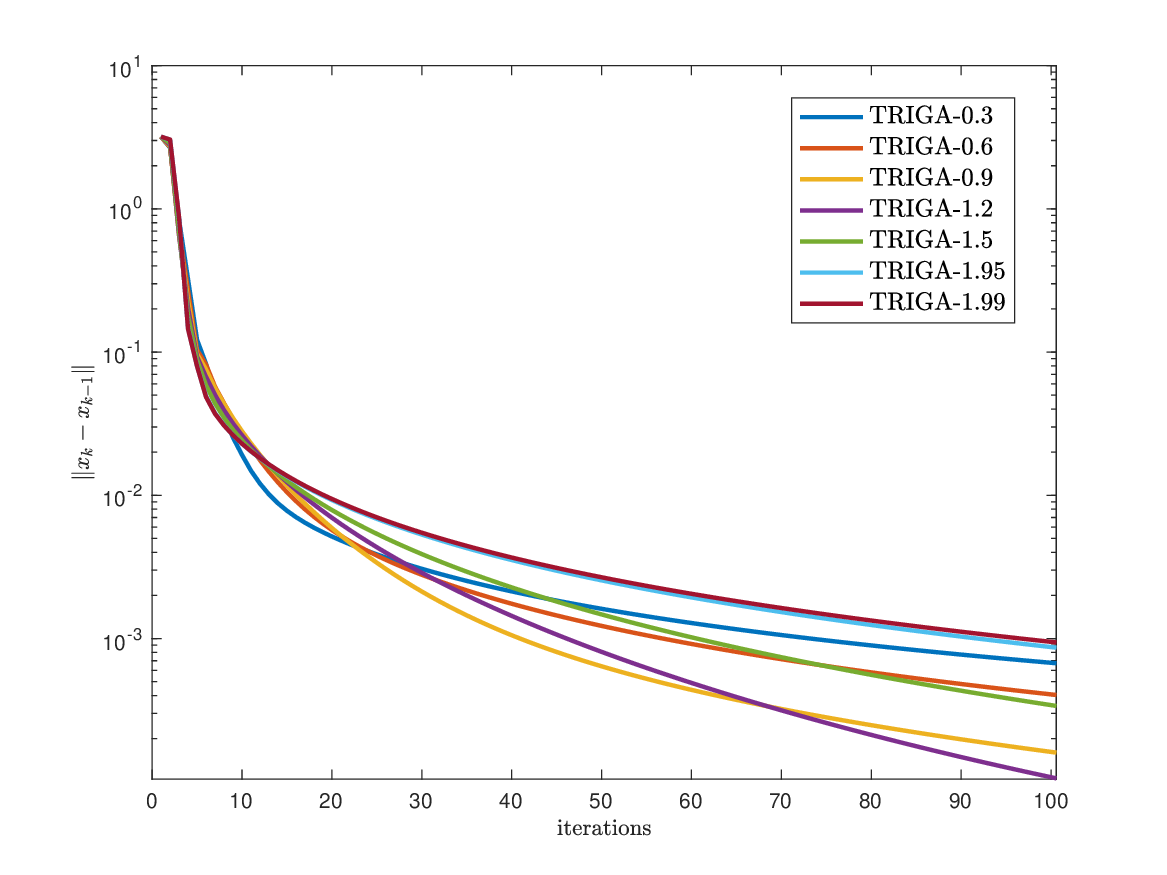}
\caption{Performance in iterations of \ref{algo:nadtr} (left) and \ref{algo:triga} (right) with $s = \dfrac{1}{1.1L}$ in terms of the criteria (from top to bottom): $f(x_k) - \min_\mathcal{H} f$, $\Vert \nabla f(x_k) \Vert$, $\Vert x_k - x^* \Vert$ and $\Vert x_k - x_{k-1}\Vert$}
\label{fig:simple_large_s}
\end{figure}

\begin{figure}[htbp]
\centering
\includegraphics[width=0.46\linewidth]{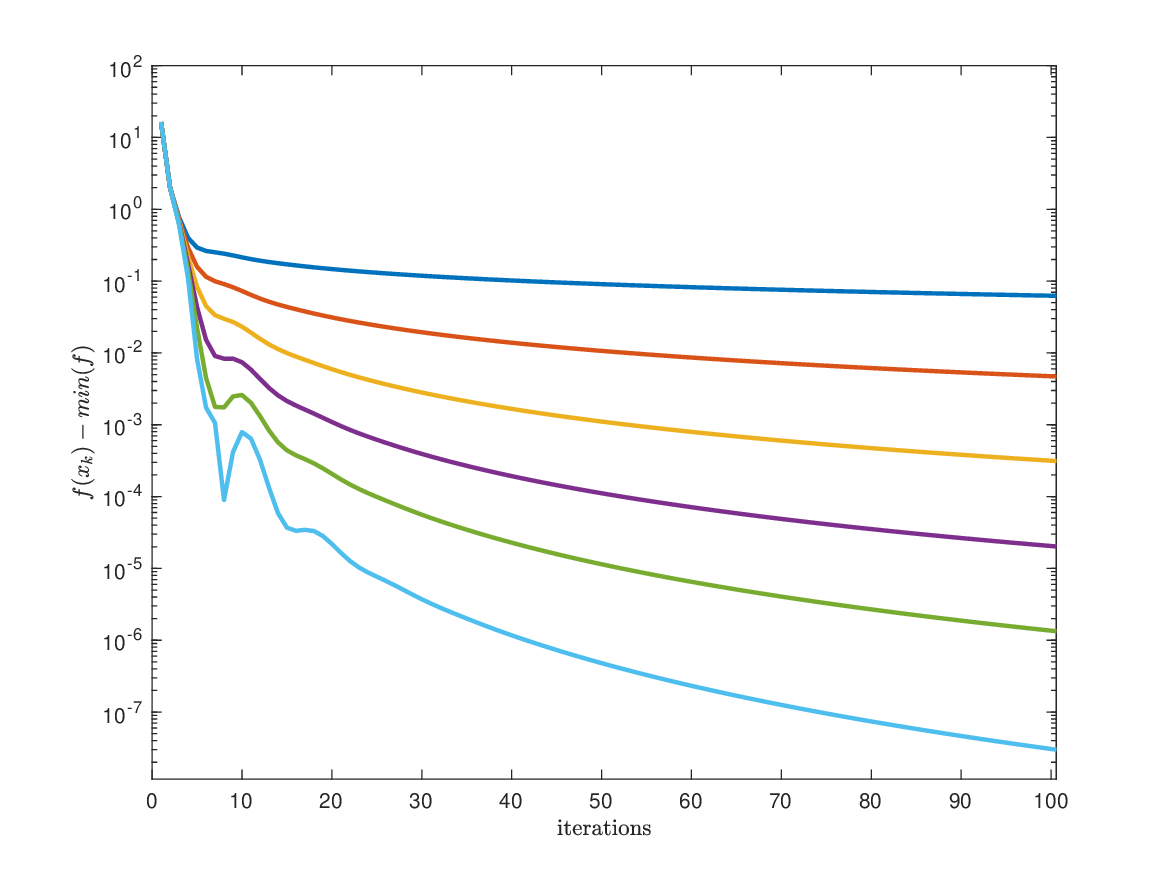}
\includegraphics[width=0.46\linewidth]{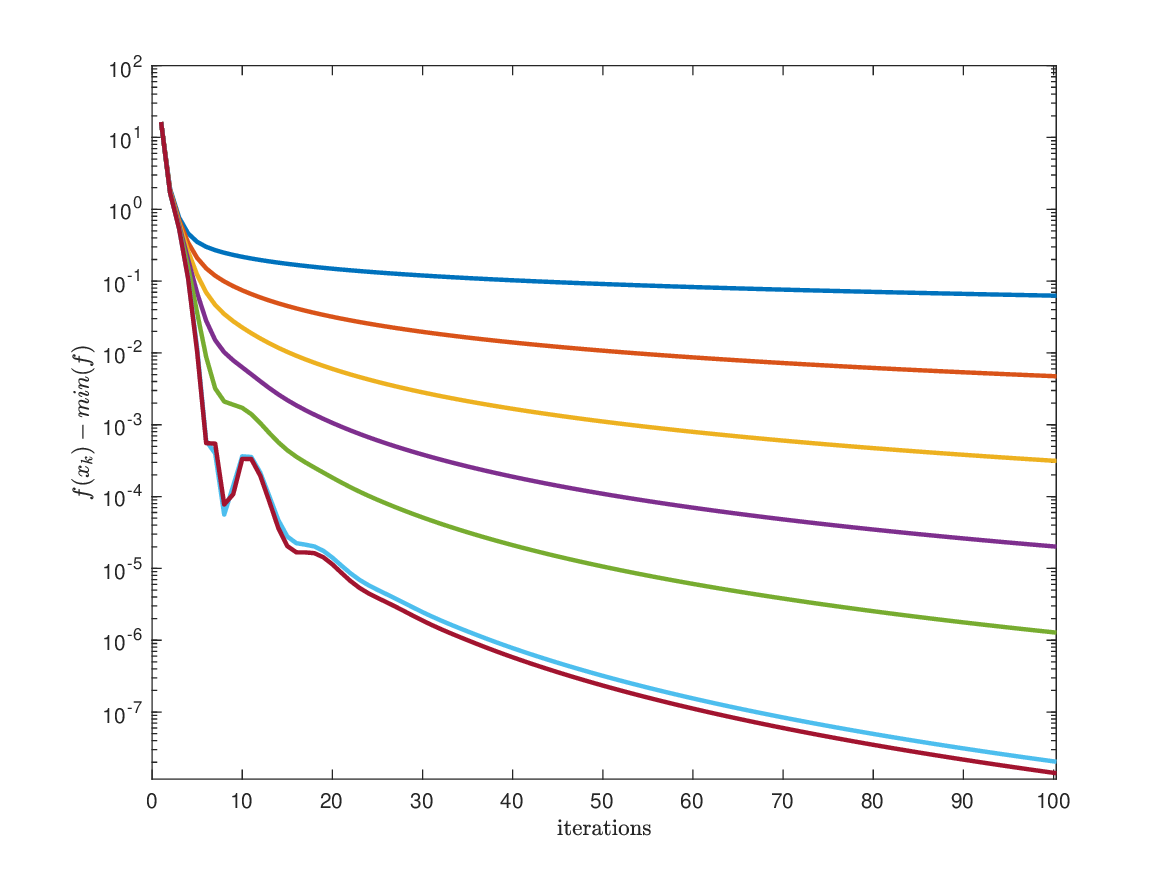}
\includegraphics[width=0.46\linewidth]{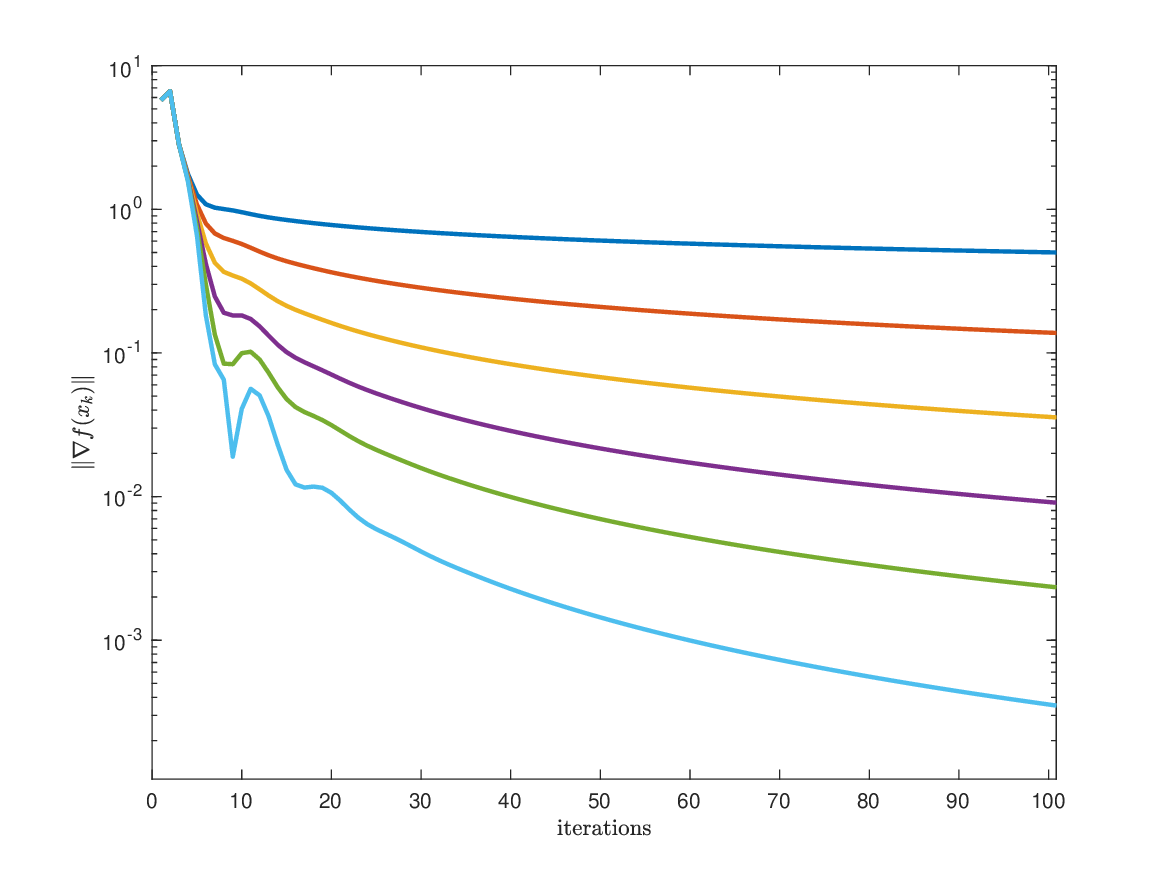}
\includegraphics[width=0.46\linewidth]{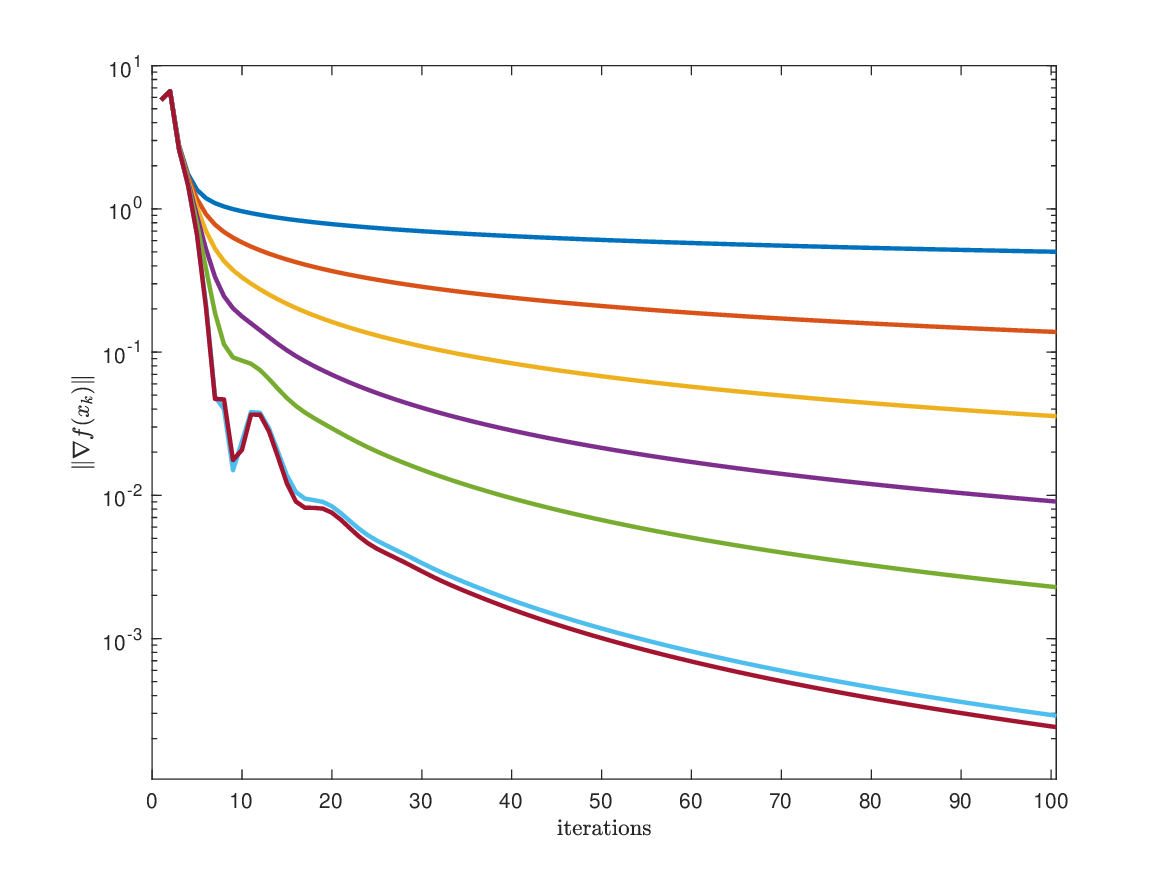}
\includegraphics[width=0.46\linewidth]{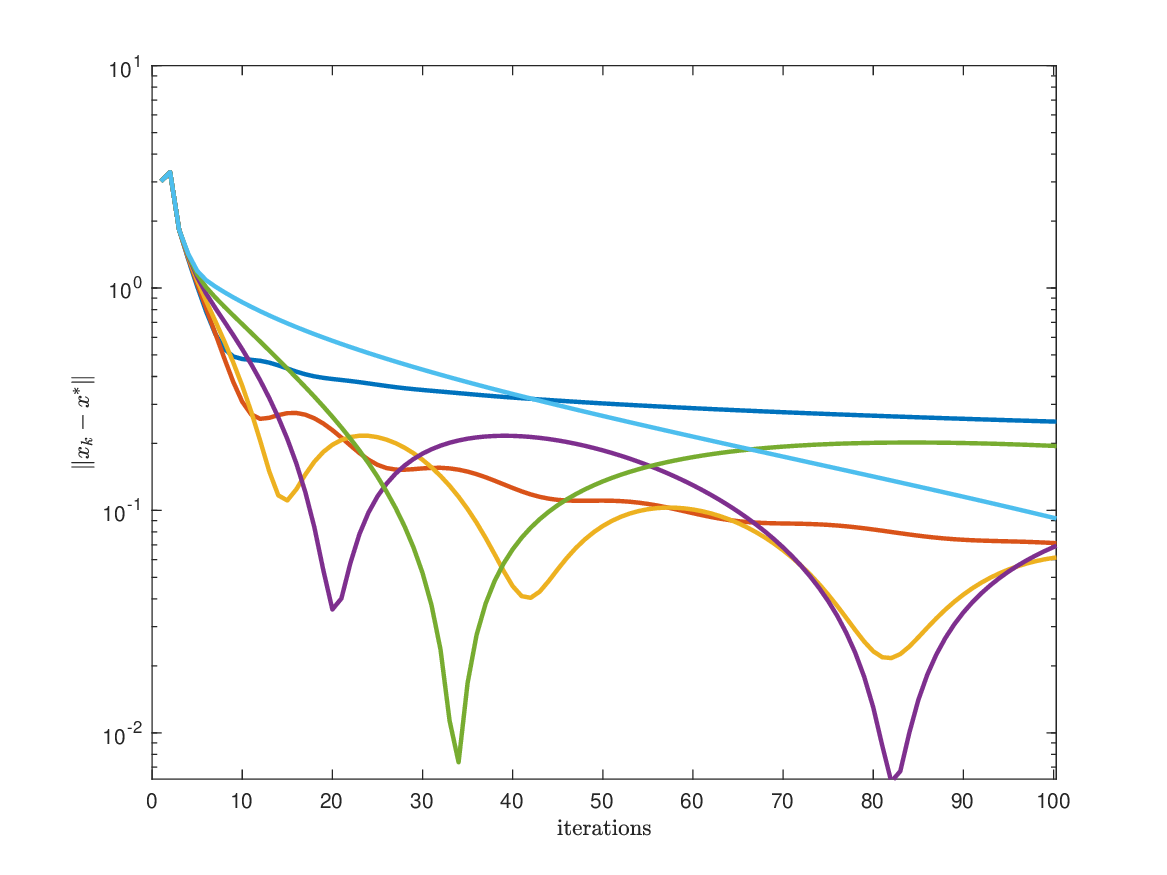}
\includegraphics[width=0.46\linewidth]{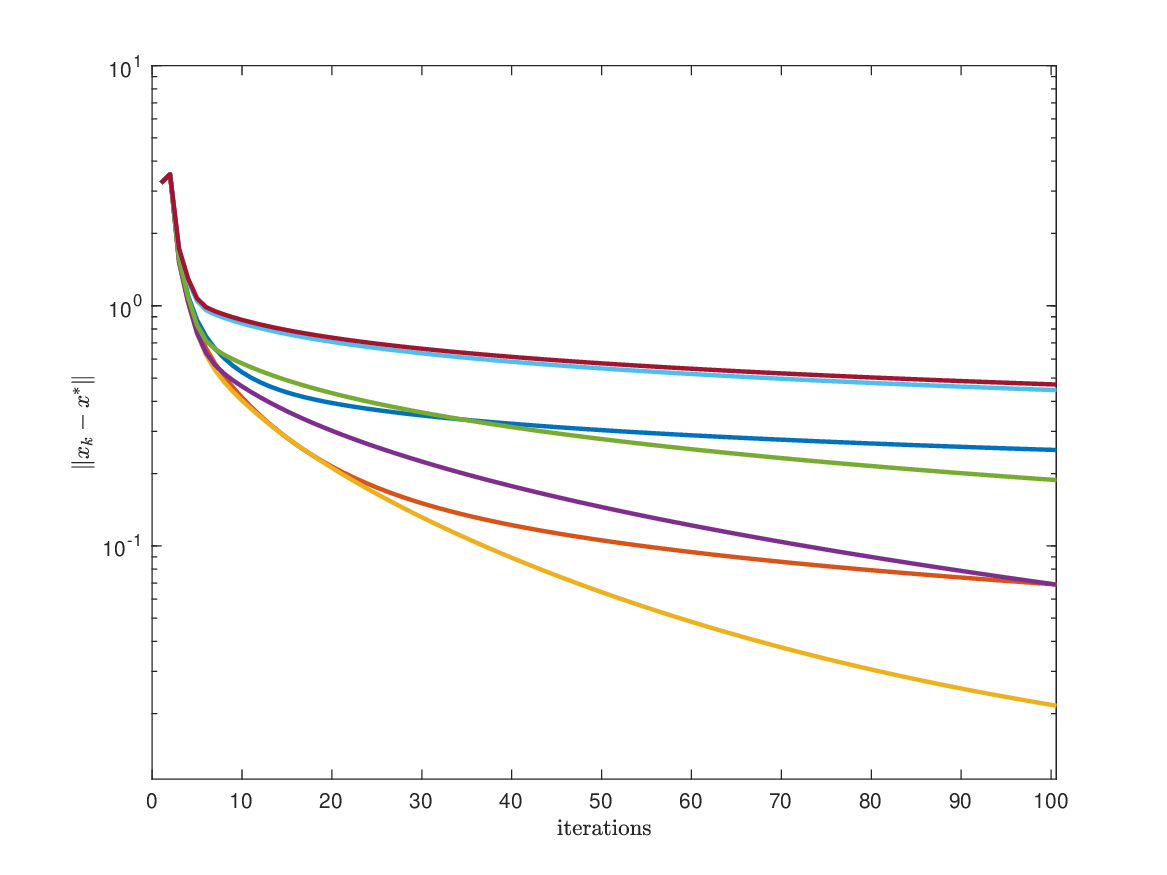}
\includegraphics[width=0.46\linewidth]{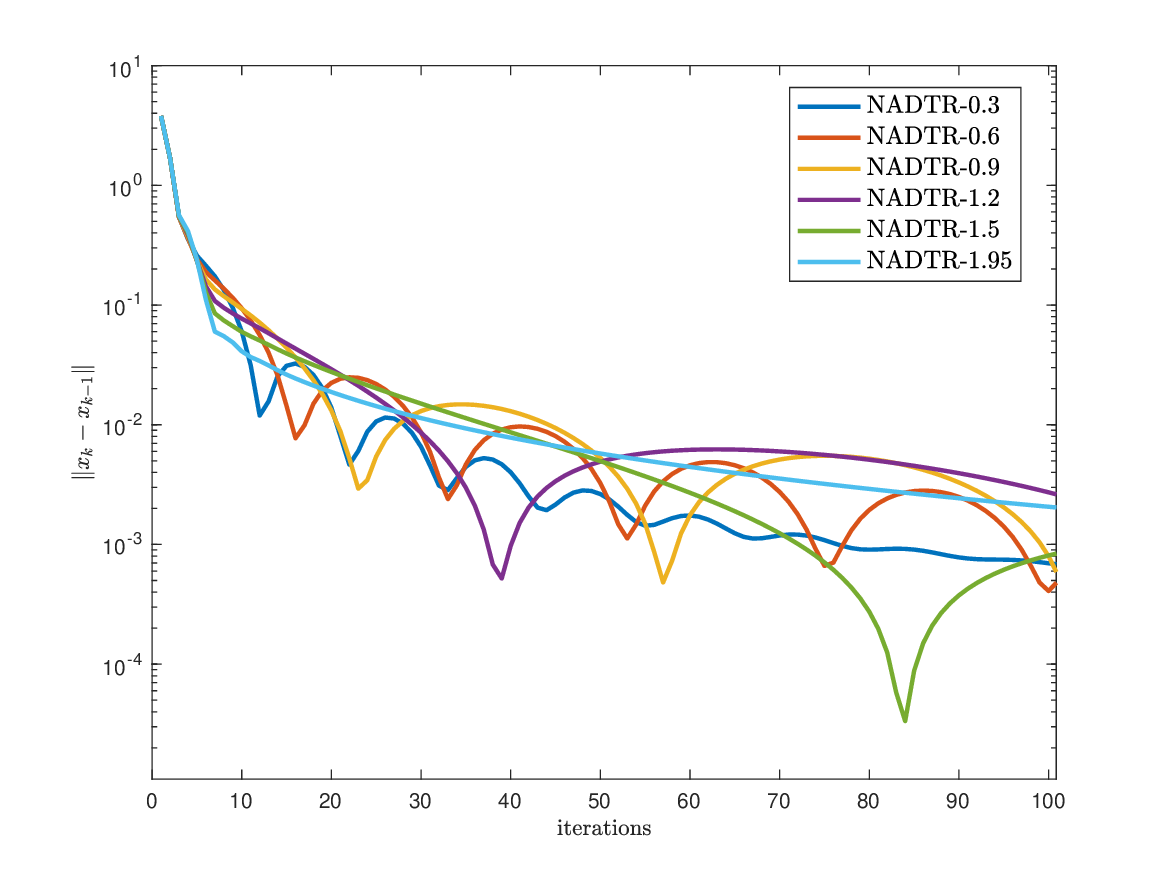}
\includegraphics[width=0.46\linewidth]{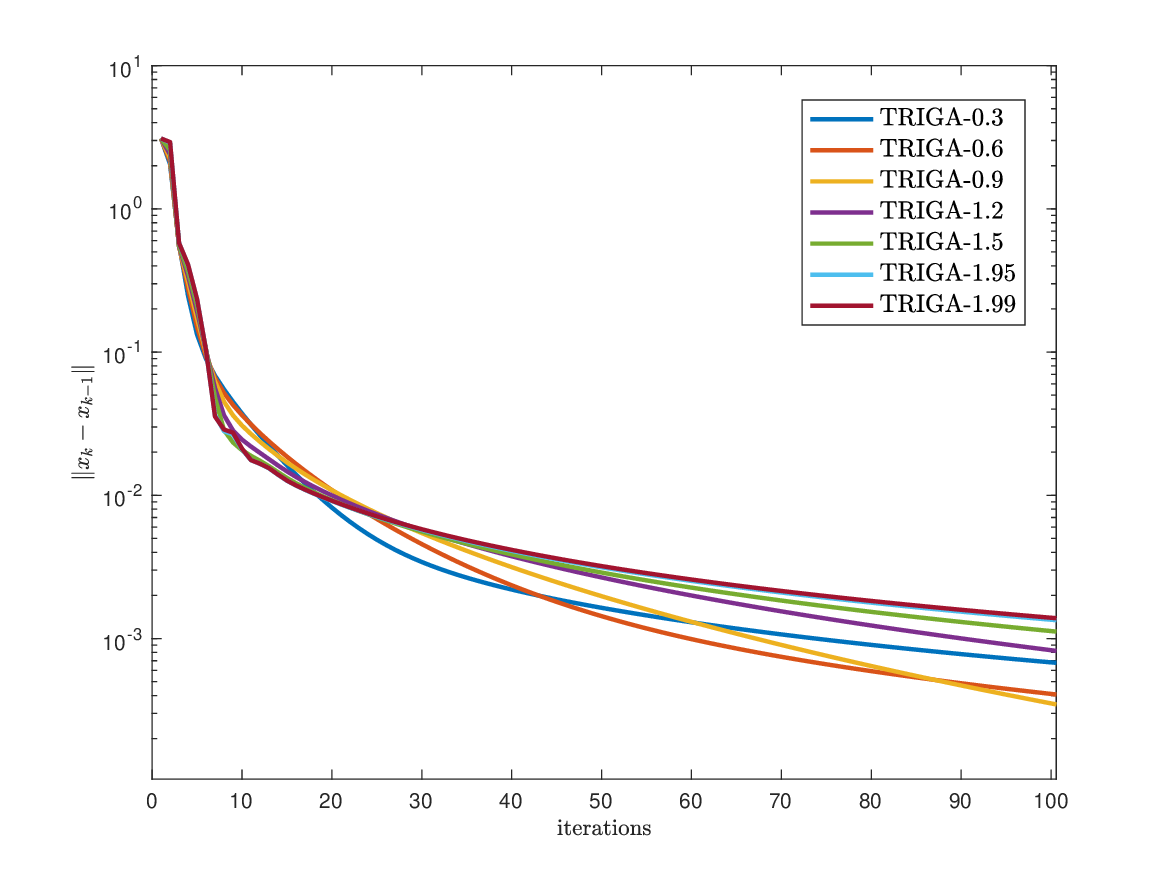}
\caption{Performance in iterations of \ref{algo:nadtr} (left) and \ref{algo:triga} (right) with $s = \dfrac{1}{2.1 L}$ in terms of the criteria (from top to bottom): $f(x_k) - \min_\mathcal{H} f$, $\Vert \nabla f(x_k) \Vert$, $\Vert x_k - x^* \Vert$ and $\Vert x_k - x_{k-1}\Vert$}
\label{fig:simple_small_s}
\end{figure}

In terms of objective function values and norm of the gradient, \ref{algo:triga} and \ref{algo:nadtr} have a comparable performance for each value of $p$. We observe that the higher the value of $p$ is, the more the performance of both algorithms enhances. The values of $p$ close to $2$ seem to be good choices. Concerning the strong convergence $\|x_k - x^*\|$, \ref{algo:nadtr} is similar as (resp. better than) \ref{algo:triga} with $p \leq 1.5$ (resp. $p\in \{1.95, 1.99\}$),
where $x^*$ denotes the minimum-norm solution of $\min f$ for this test problem (available in closed form), so that $\|x_k-x^*\|$ provides a standard proxy for strong convergence,
however \ref{algo:nadtr} exhibits more oscillation, especially for smaller step-size such as $s=\frac{1}{2.1L}$. Regarding the discrete velocity $\|x_k - x_{k-1}\|$, both algorithms \ref{algo:nadtr} and \ref{algo:triga} share quite similar rates in all values of $p$ but with more oscillatory behaviors in the case of \ref{algo:nadtr}. Selecting the larger step-size $s=\frac{1}{1.1L}$ makes the algorithms more stable. In conclusion for this test problem, both algorithms demonstrate comparable performance in most of four criteria; the difference is that \ref{algo:nadtr} yields better strong convergence $\|x_k -x^*\|$ in some tests, whereas our algorithm significantly reduces the oscillation effects.

\subsection{Problem 2: Linear least-squares problem}
The linear least-squares optimization problem is defined as
\begin{equation*}
    \min\limits_{x \in \mathbb{R}^n} f(x) = \dfrac{1}{2} \|Ax - b\|^2,
\end{equation*}
where $A \in \mathbb{R}^{m \times n}$ and $b \in \mathbb{R}^{m}$.
Two sets of these problems are used in this study: one based on real benchmark data and the other on synthetically generated instances. Specifically, for the benchmark data, we choose matrices $A$ from the SuiteSparse Matrix Collection\footnote{https://sparse.tamu.edu/} and draw $b$ randomly. Our set of problems consists of $38$ different problems with the size $m$ and $n$ varying from $9$ to $3600$. Since the chosen matrices $A$ are not full-rank, the objective function $f$ is convex, but not strongly convex. As for the synthetic data, we test on $40$ different problems where square matrices $A \in \mathbb{R}^{n \times n}$ and vectors $b \in \mathbb{R}^{n}$ are randomly generated with $n$ ranging from $5$ to $14$. All synthetic instances are generated with a fixed random seed, using standard Gaussian entries for $(A,b)$, and no additional tuning is performed across instances.

On various datasets, we evaluate the computational efficiency of \ref{algo:nadtr} and \ref{algo:triga} using two key criteria: the number of iterations and the CPU time (in seconds) which measures the computational time required by each algorithm for solving the problems. To choose the best value of $p$ for comparison, we execute both algorithms on benchmark datasets with different Tikhonov regularization terms. Then, we compare their performance with chosen best value of $p$ on all datasets of each type. The performance profiles of \ref{algo:triga} (resp. \ref{algo:nadtr}) with different $p$ in terms of CPU time and number of iterations are reported in Figure \ref{fig:best_p_triga} (resp. Figure \ref{fig:best_p_nadtr}). Figure \ref{fig:performance_lasso_benchmark} (resp. Figure \ref{fig:performance_lasso_synthetic}) represents the performance profiles of \ref{algo:triga} and \ref{algo:nadtr} with the chosen best value of $p$ and the same random starting points.
  \begin{figure}[htbp]
  \centering
  \includegraphics[width=0.46\linewidth]{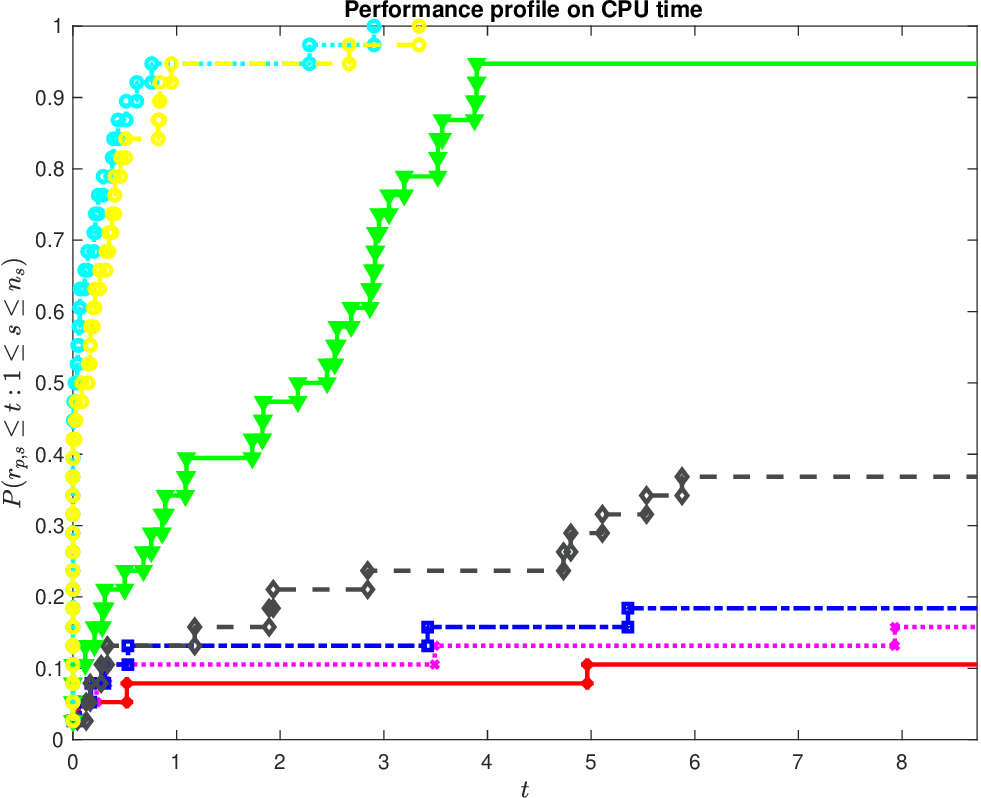}
  \includegraphics[width=0.46\linewidth]{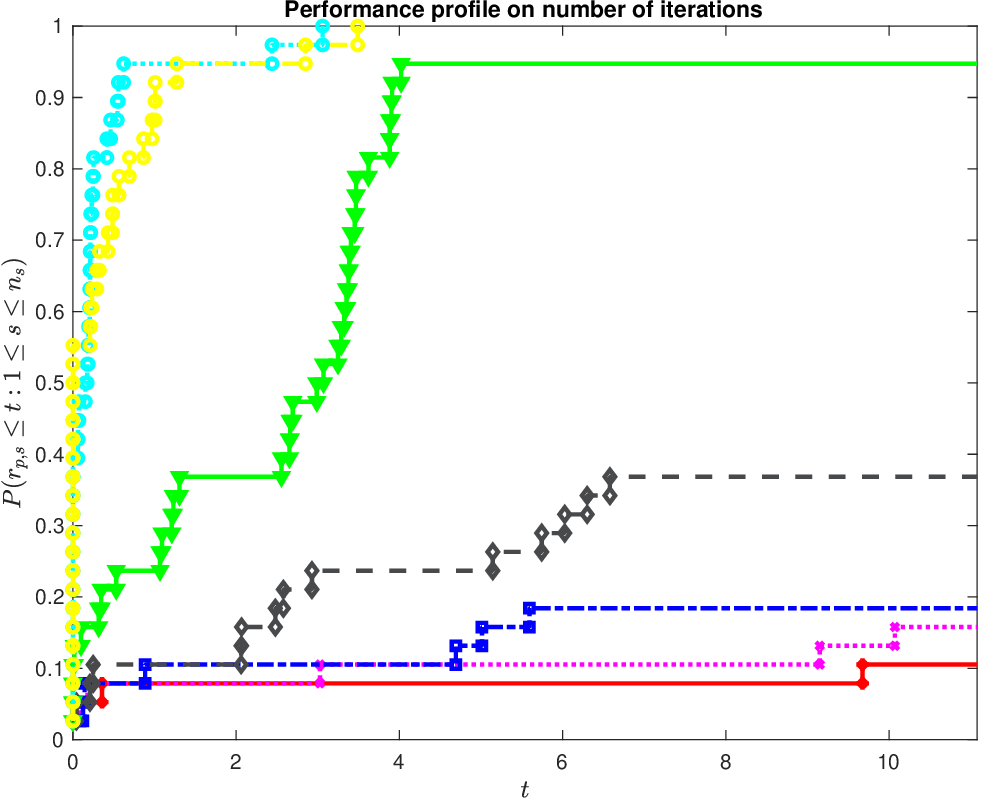}
  \includegraphics[width=0.46\linewidth]{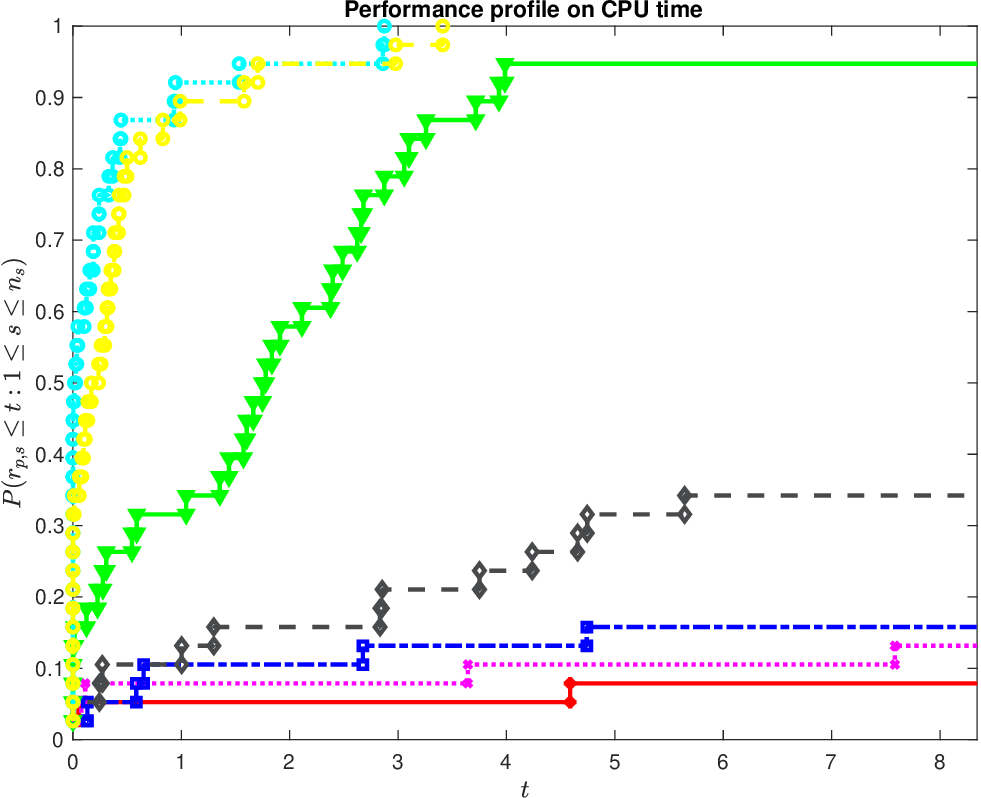}
  \includegraphics[width=0.46\linewidth]{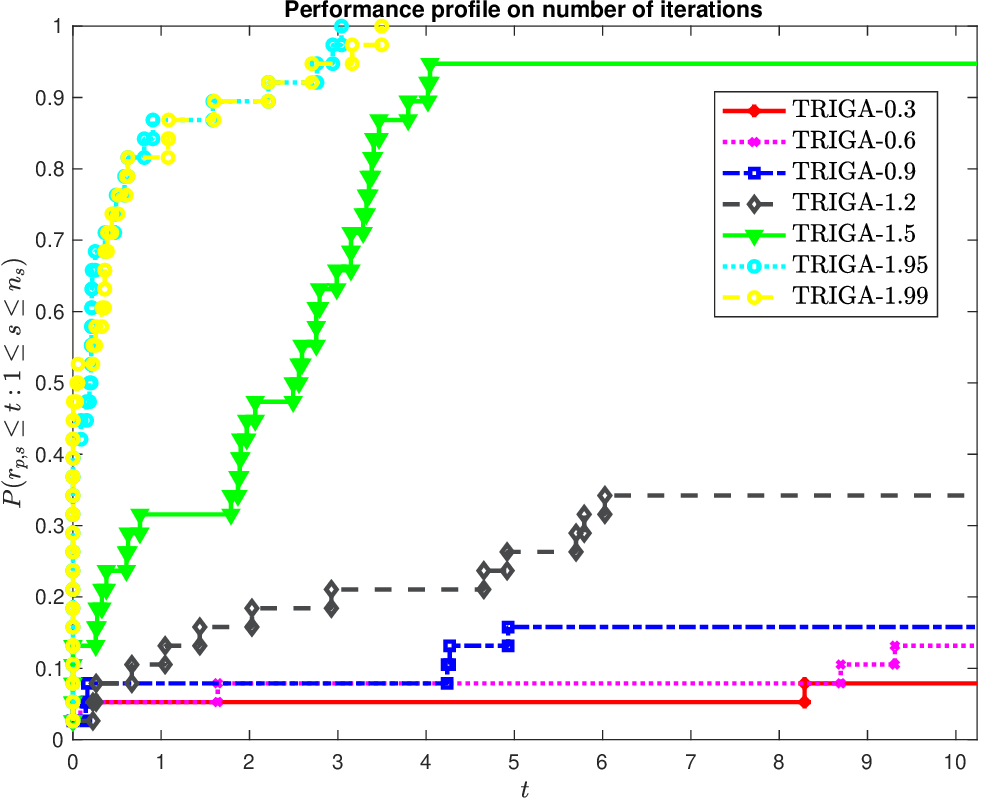}
  \caption{Performance profiles of \ref{algo:triga} for different values of $p$ when $s = \frac{1}{1.1L}$ (first row) and $s = \frac{1}{2.1L}$ (second row) in terms of CPU time (in seconds) (left) and number of iterations (right) on benchmark datasets of linear least-squares problems}  \label{fig:best_p_triga}
  \end{figure}

  \begin{figure}[htbp]
  \centering
  \includegraphics[width=0.46\linewidth]{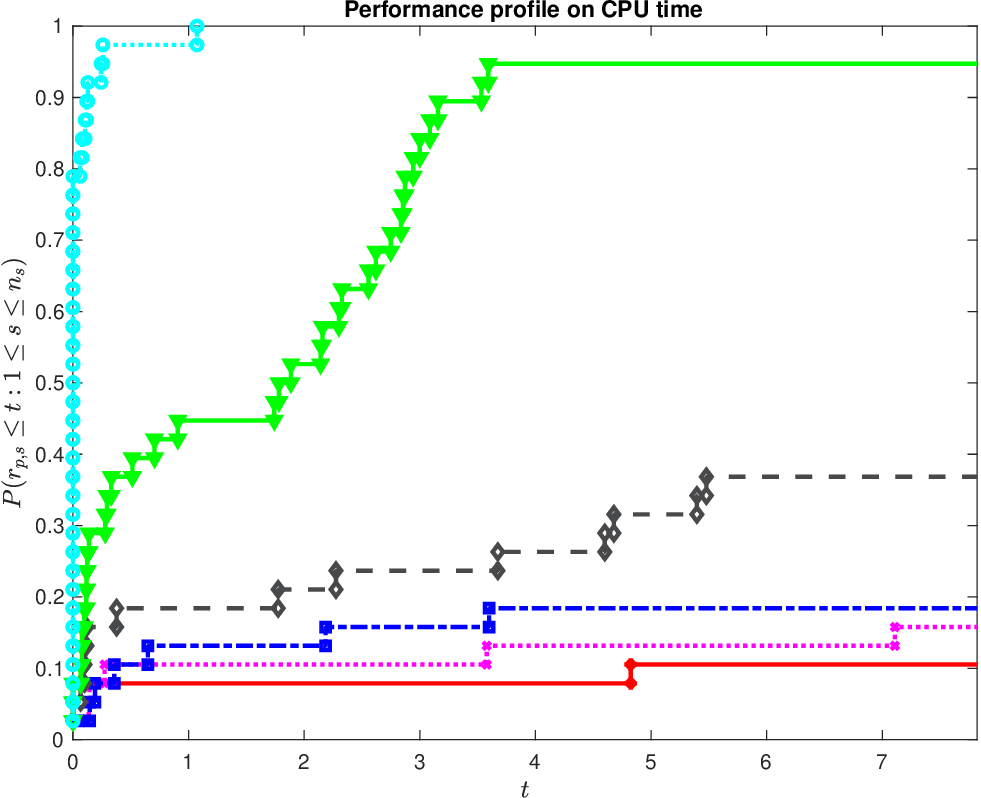}
  \includegraphics[width=0.46\linewidth]{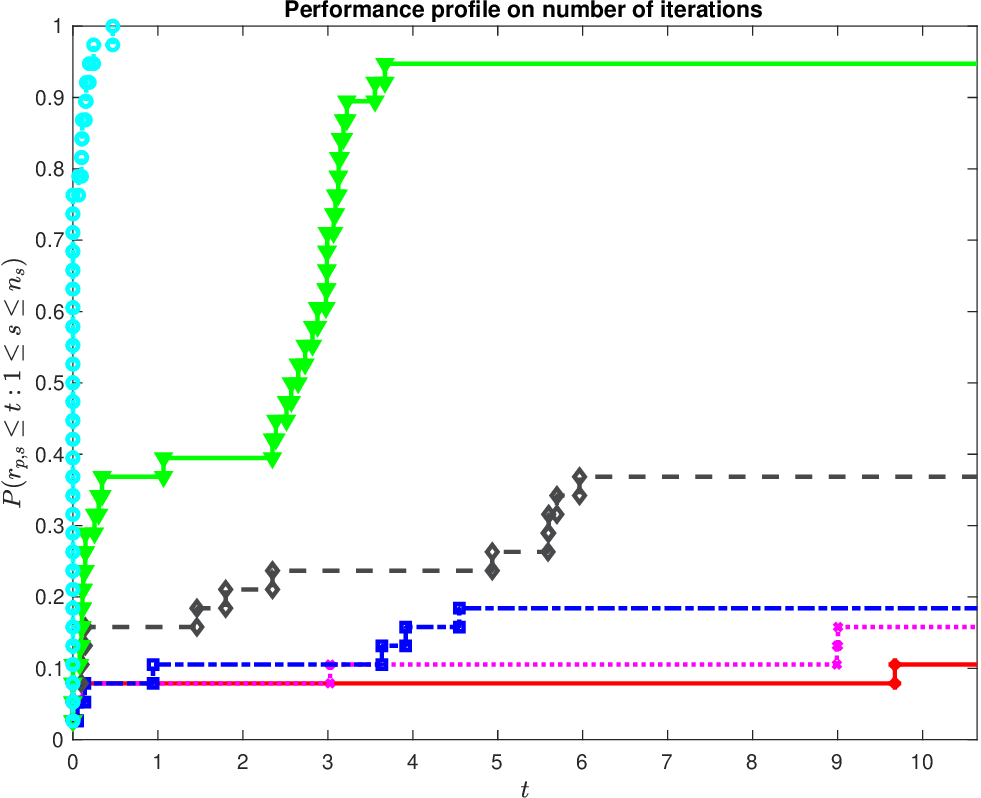}
  \includegraphics[width=0.46\linewidth]{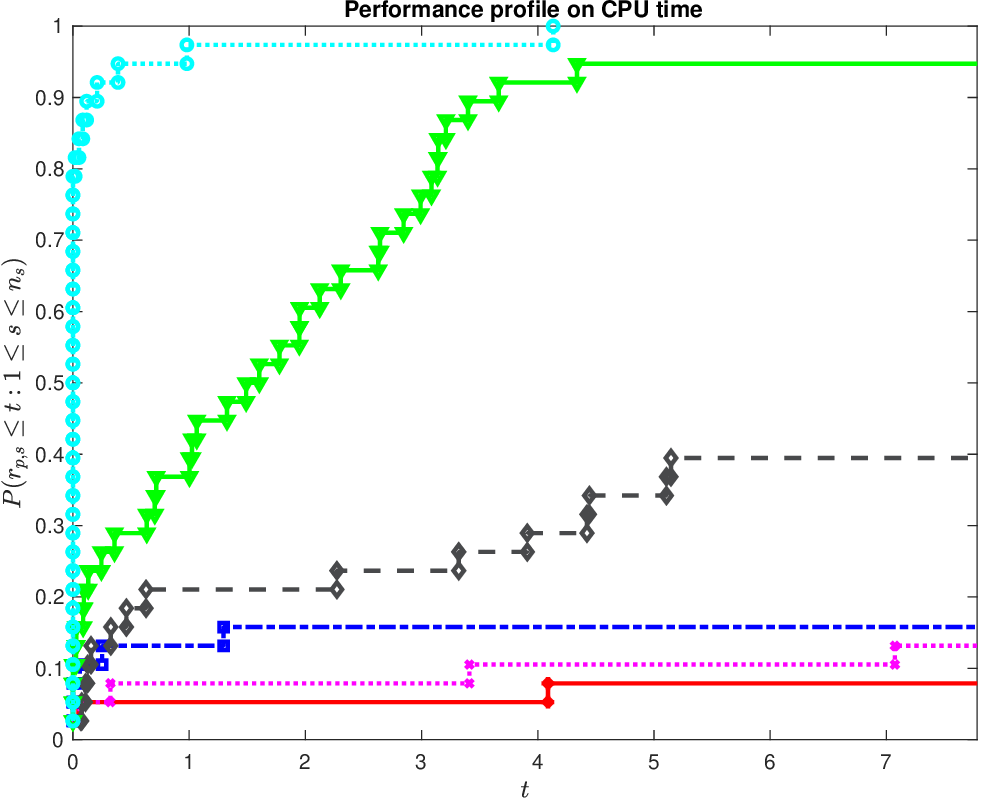}
  \includegraphics[width=0.46\linewidth]{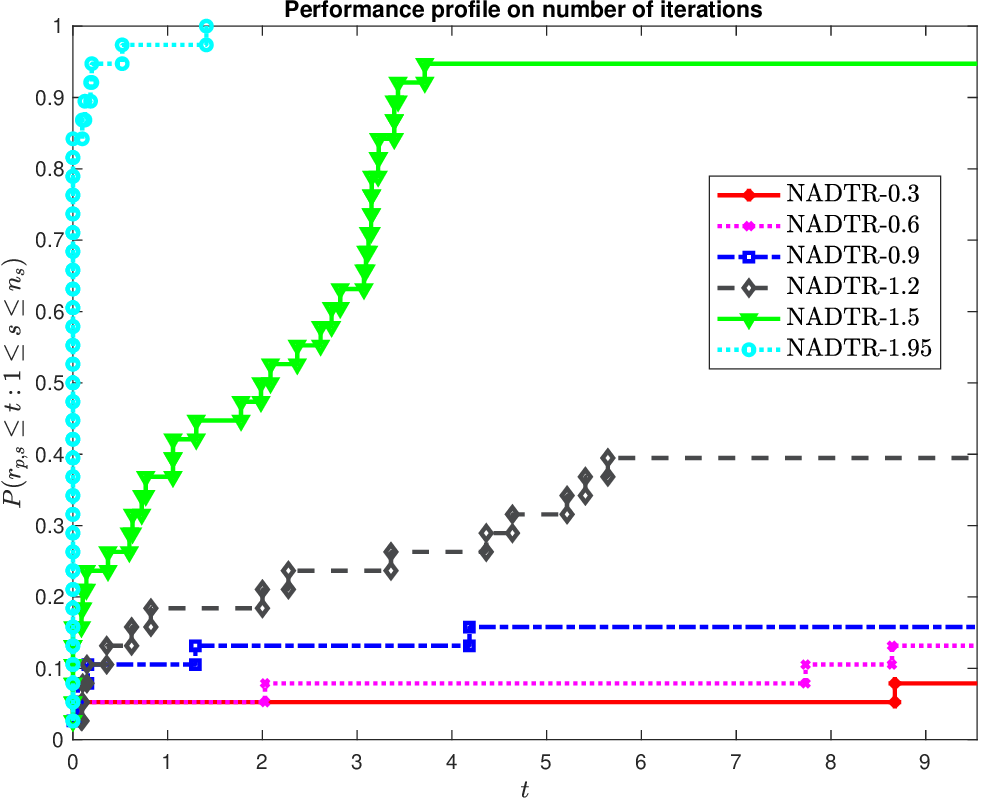}
  \caption{Performance profiles of \ref{algo:nadtr} for different values of $p$ when $s = \frac{1}{1.1L}$ (first row) and $s = \frac{1}{2.1L}$ (second row) in terms of CPU time (in seconds) (left) and number of iterations (right) on benchmark datasets of linear least-squares problems} \label{fig:best_p_nadtr}
  \end{figure}

  \begin{figure}[htbp]
  \centering
  \includegraphics[width=0.46\linewidth]{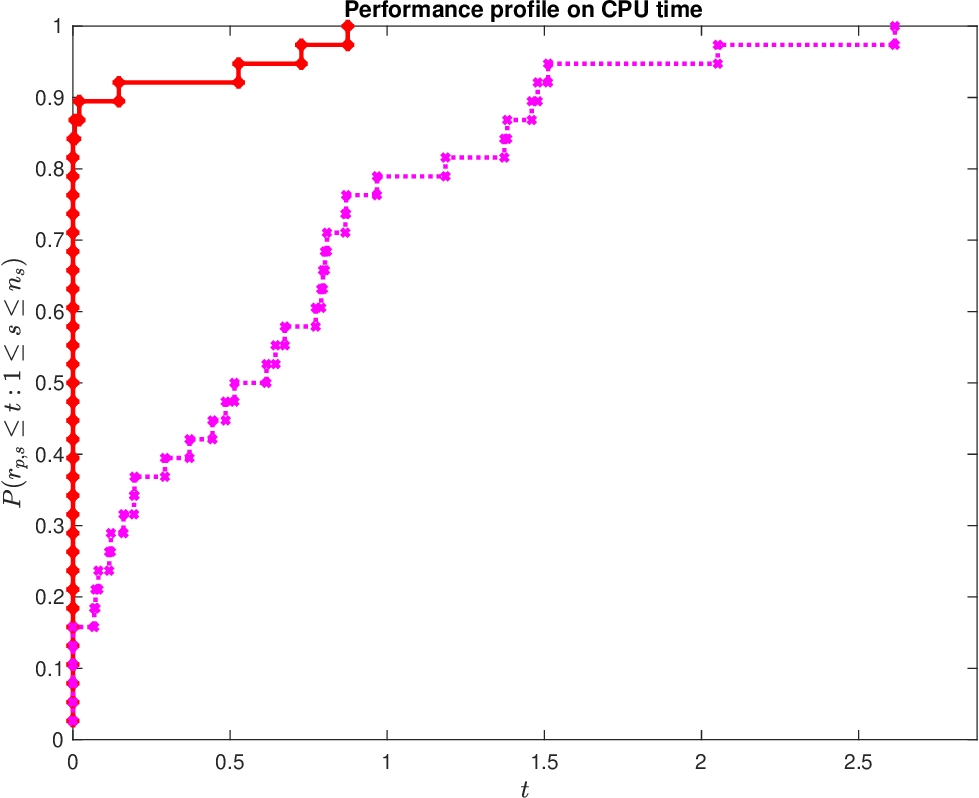}
  \includegraphics[width=0.46\linewidth]{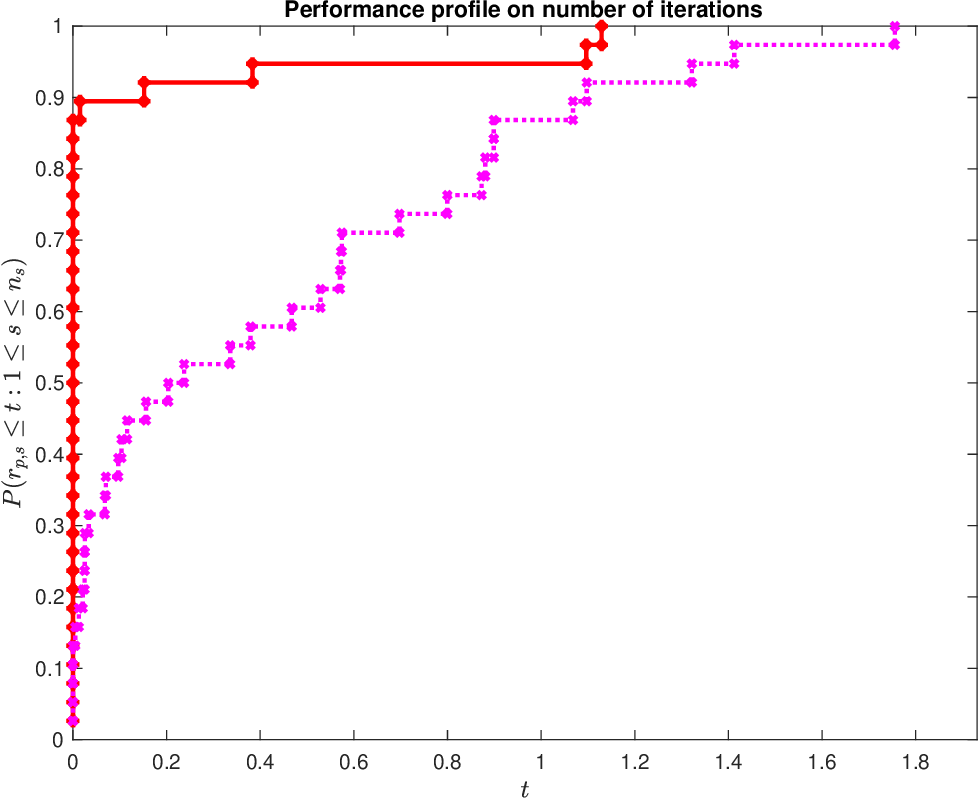}
  \includegraphics[width=0.46\linewidth]{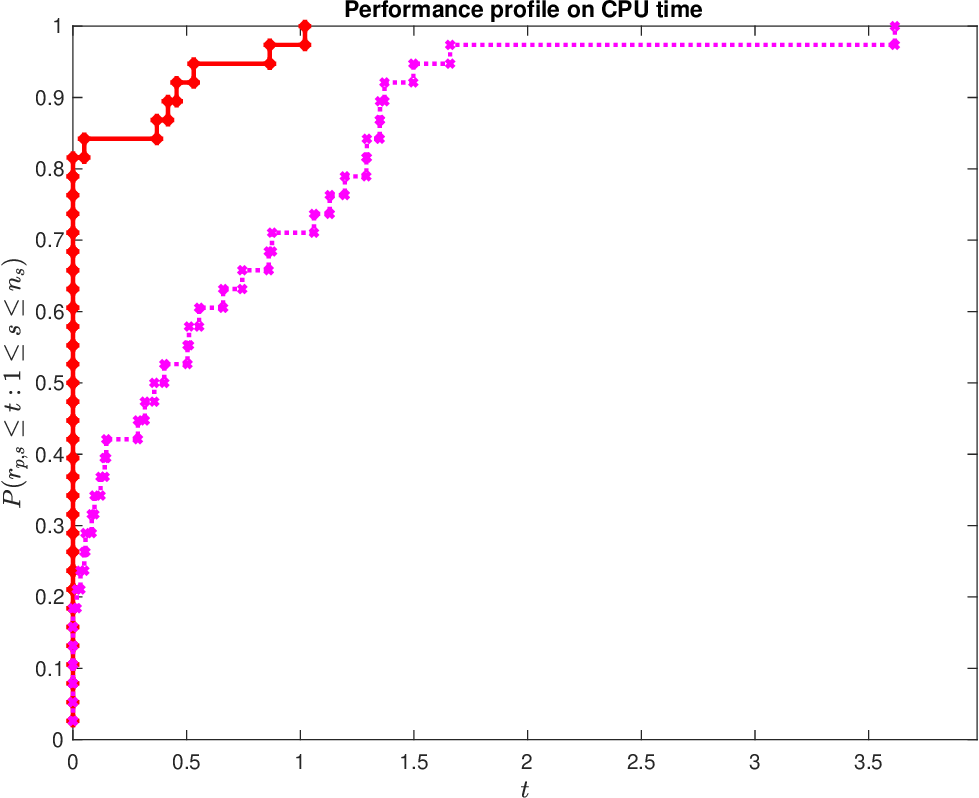}
  \includegraphics[width=0.46\linewidth]{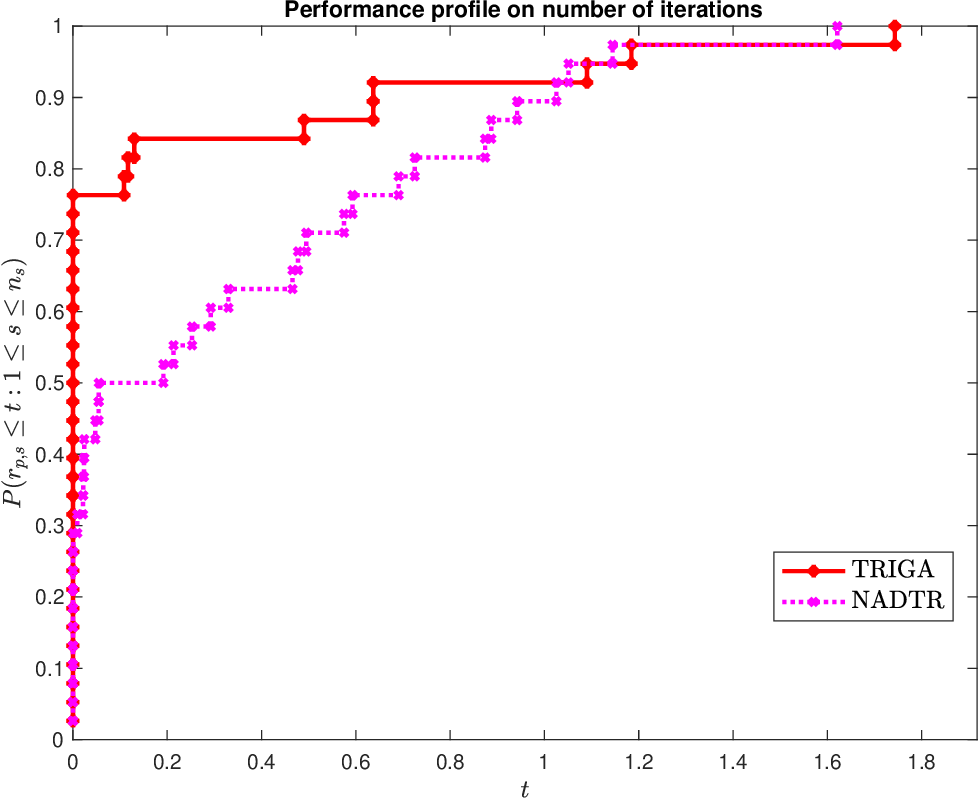}
  \caption{Performance profiles of \ref{algo:triga} and \ref{algo:nadtr} with $p=1.95$ when $s = \frac{1}{1.1L}$ (first row) and $s = \frac{1}{2.1L}$ (second row) in terms of CPU time (in seconds) (left) and number of iterations (right) on benchmark datasets of linear least-squares problems}
  \label{fig:performance_lasso_benchmark}
  \end{figure}

  \begin{figure}[htbp]
  \centering
  \includegraphics[width=0.46\linewidth]{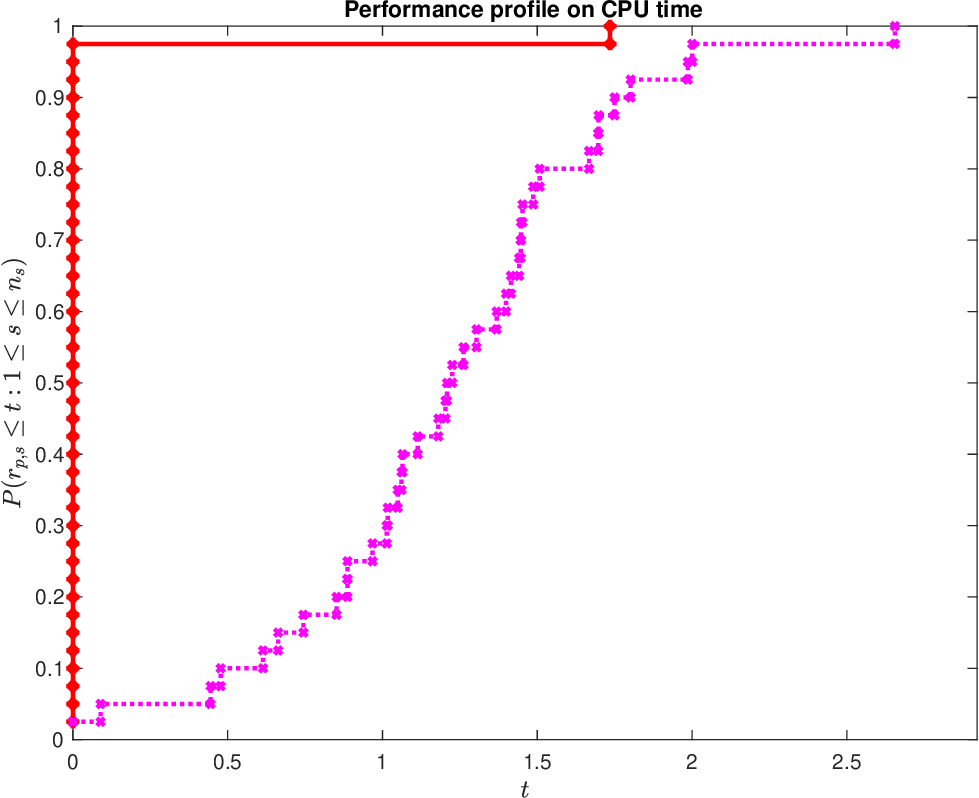}
  \includegraphics[width=0.46\linewidth]{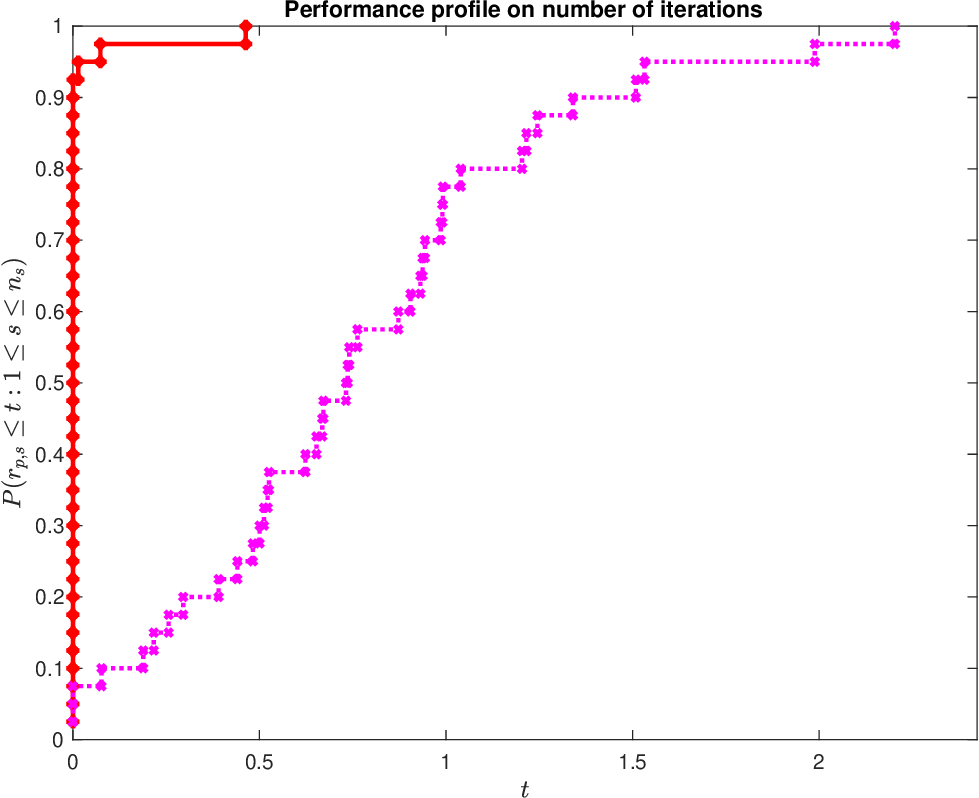}
  \includegraphics[width=0.46\linewidth]{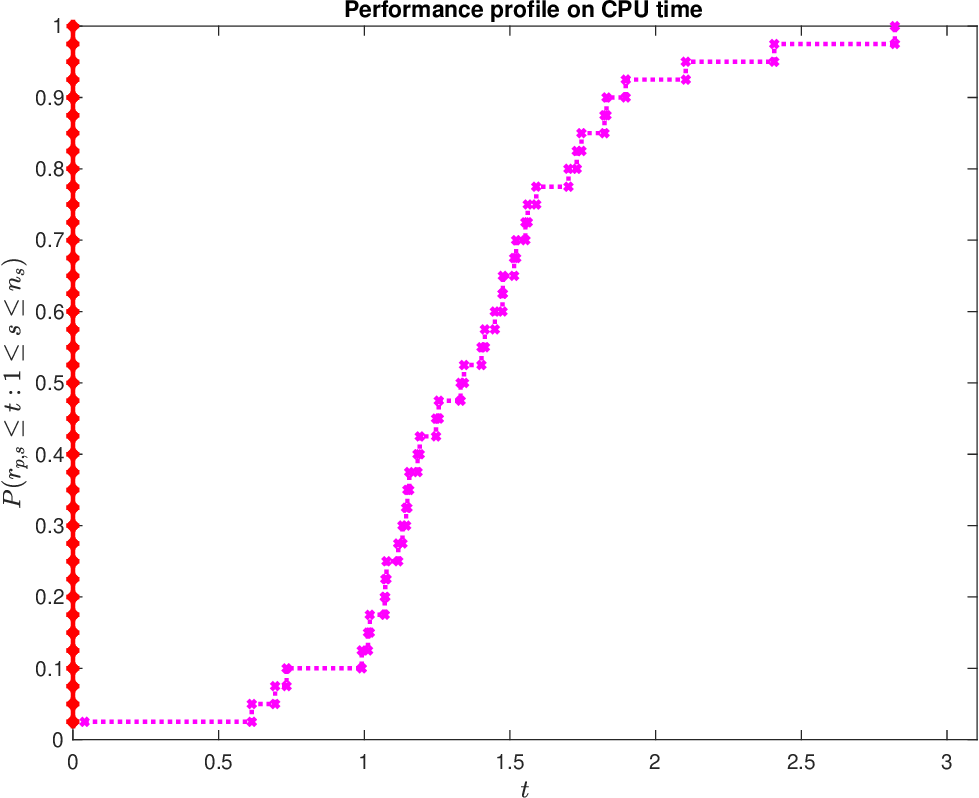}
  \includegraphics[width=0.46\linewidth]{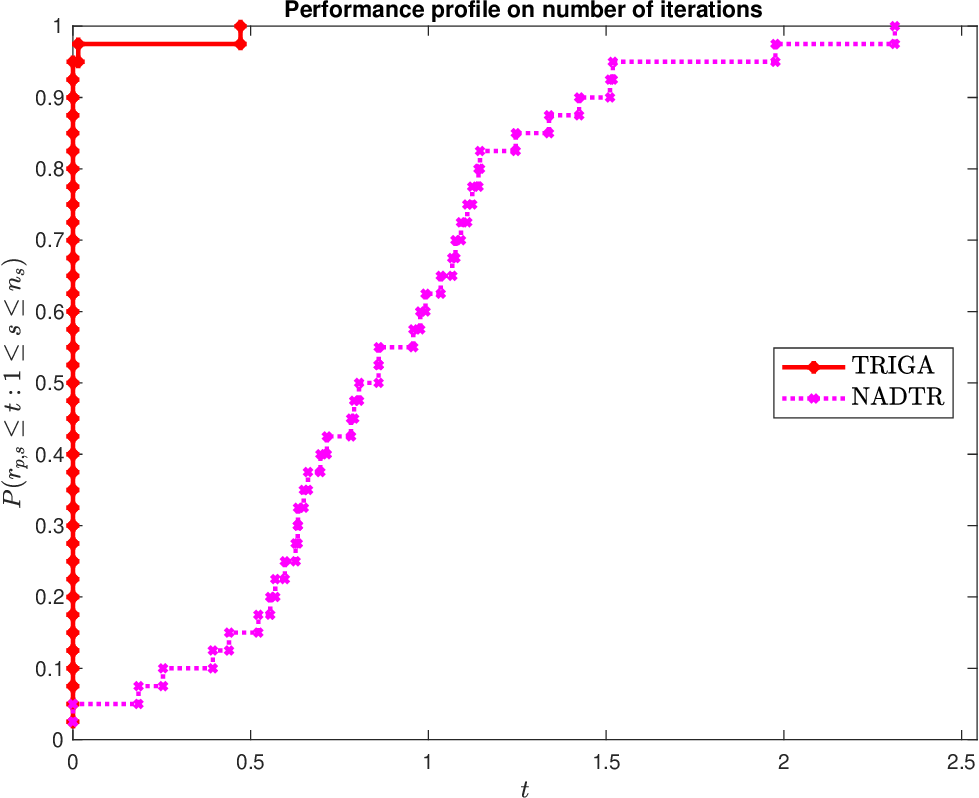}  
  \caption{Performance profiles of \ref{algo:triga} and \ref{algo:nadtr} with $p=1.95$ when $s = \frac{1}{1.1L}$ (first row) and $s = \frac{1}{2.1L}$ (second row) in terms of CPU time (in seconds) (left) and number of iterations (right) on synthetic datasets of linear least-squares problems}
  \label{fig:performance_lasso_synthetic}
  \end{figure}

From Figure~\ref{fig:best_p_triga}, we observe that Algorithm \ref{algo:triga} achieves a performance $\rho_{s}(t)$ exceeding $0.9$ even for relatively small values of $t$. It works very well for $p \geq 1.5$ and the best  $p = 1.95$. In contrast, for smaller values of $p$, the algorithm fails to successfully solve most problems, as indicated by $\rho_{s}(t) \leq 0.1$, in the limit of iterations. Once again, the larger the step-size is, the more the performance of \ref{algo:triga} is improved. This tendency is similarly observed  for Algorithm \ref{algo:nadtr}, as shown in Figure~\ref{fig:best_p_nadtr}. The results of test problems in Subsection~\ref{sec:test_problem_simple} and in this subsection demonstrate that $p=1.95$ is the good choice in Tikhonov regularization term $\varepsilon_k = \frac{1}{k^{p}}$ for both algorithms. We emphasize that this choice of $p$ is purely empirical and based on the present benchmark suite. In addition, Figure~\ref{fig:performance_lasso_benchmark} and Figure~\ref{fig:performance_lasso_synthetic} both reveal that \ref{algo:triga} consistently outperforms \ref{algo:nadtr}, as explained by its higher-lying performance curve for both types of datasets. The results of \ref{algo:triga} with $s = 1/(1.1L)$ are slightly better than (resp. as good as) those with $s = 1/(2.1L)$ on benchmark (resp. synthetic) datasets. As for CPU time (resp. number of iterations), \ref{algo:triga} solves more than $90\%$ of benchmark problems within the factor $t \geq 0.15$ (resp. $t \geq 0.15$) of the optimal possible result while \ref{algo:nadtr} is within $t \geq 1.48$ (resp. $t \geq 1.10$). For the synthetic datasets, our algorithm runs faster than \ref{algo:nadtr} on $39$ out of $40$ problems ($\rho_s = 97.5\%$) and requires fewer iterations on $37$ out of $40$ problems ($\rho_s = 92.5\%$). This can be explained in Remark~\ref{remark:particular_case_plessthan2}. The gap of performance $\rho_{s}$ between the two algorithms on the benchmark datasets in terms of CPU time (resp. number of iterations) varies from $21\%$ to $74\%$ (resp. from $5\%$ to $71\%$). This gap becomes significantly larger for the synthetic datasets, ranging from $10\%$ to $95\%$ (resp. from $73\%$ to $88\%$). 

\subsection{Problem 3: Logistic regression problem}

We finally address the logistic regression problem for binary classiﬁcation. Given a set of $m$ points $\{(a^{(i)}, y^{(i)})\}_{1 \leq i\leq m}$, the considered problem is to minimize the objective function $f : \mathbb{R}^n \to \mathbb{R}$ given by
\begin{align*}
  f(x) = \dfrac{1}{m} \sum_{i=1}^m \log(1+ e^{ - y^{(i)} \langle a^{(i)}, x \rangle}),
\end{align*}
where $a^{(i)} \in \mathbb{R}^n$ is the $n$-dimensional feature vector and $y^{(i)} \in \{-1,1\}$ is the binary label. Our experiment is conducted on a set of $12$ different real datasets drawn from LIBSVM Data\footnote{https://www.csie.ntu.edu.tw/~cjlin/libsvmtools/datasets/} where the number of samples $m$ (resp. the number of features $n$) varies from $62$ to $32561$ (resp. from $8$ to $2000$). Here, following the previous experiments, we set $\varepsilon_k = \frac{1}{k^{1.95}}$ as the best Tikhonov regularization parameter for both \ref{algo:triga} and \ref{algo:nadtr}. Figure \ref{fig:performance_logistic} displays the performance profiles of \ref{algo:triga} and \ref{algo:nadtr} with respect to CPU time and number of iterations. 

\begin{figure}[htbp]
\centering
\includegraphics[width=0.46\linewidth]{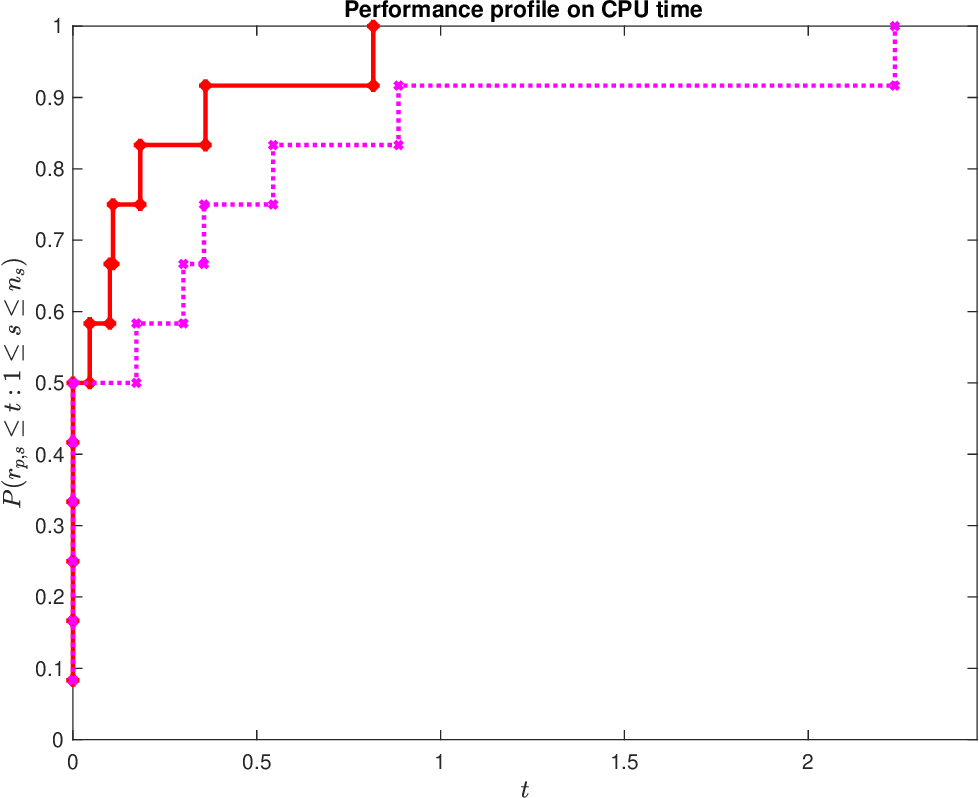}
\includegraphics[width=0.46\linewidth]{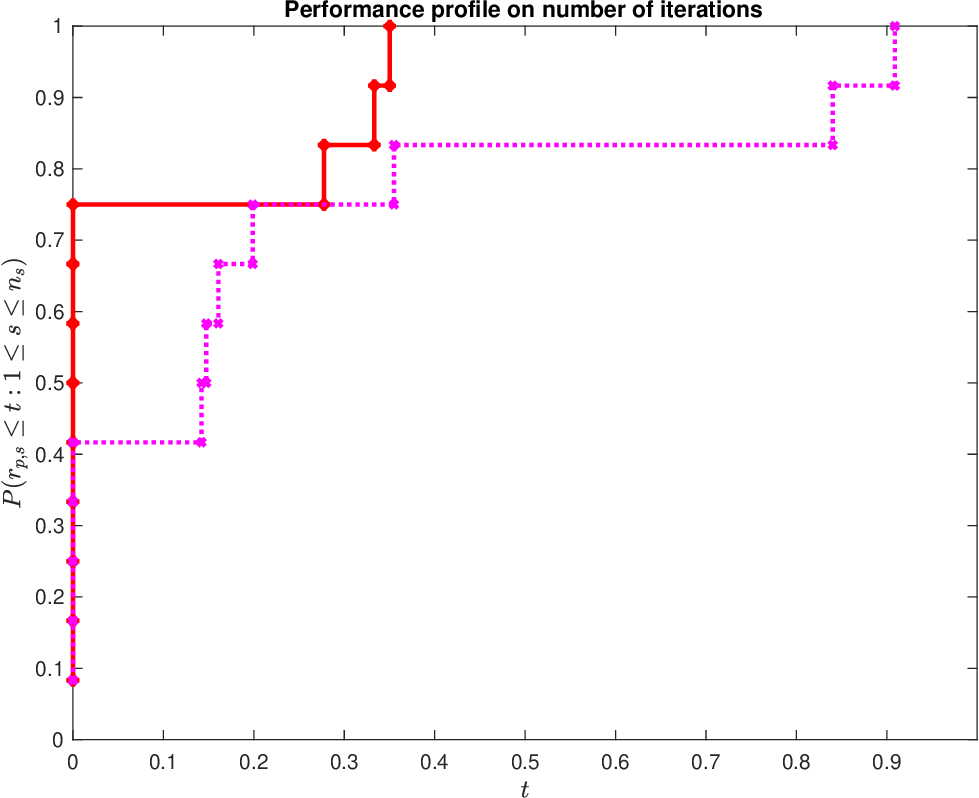}
\includegraphics[width=0.46\linewidth]{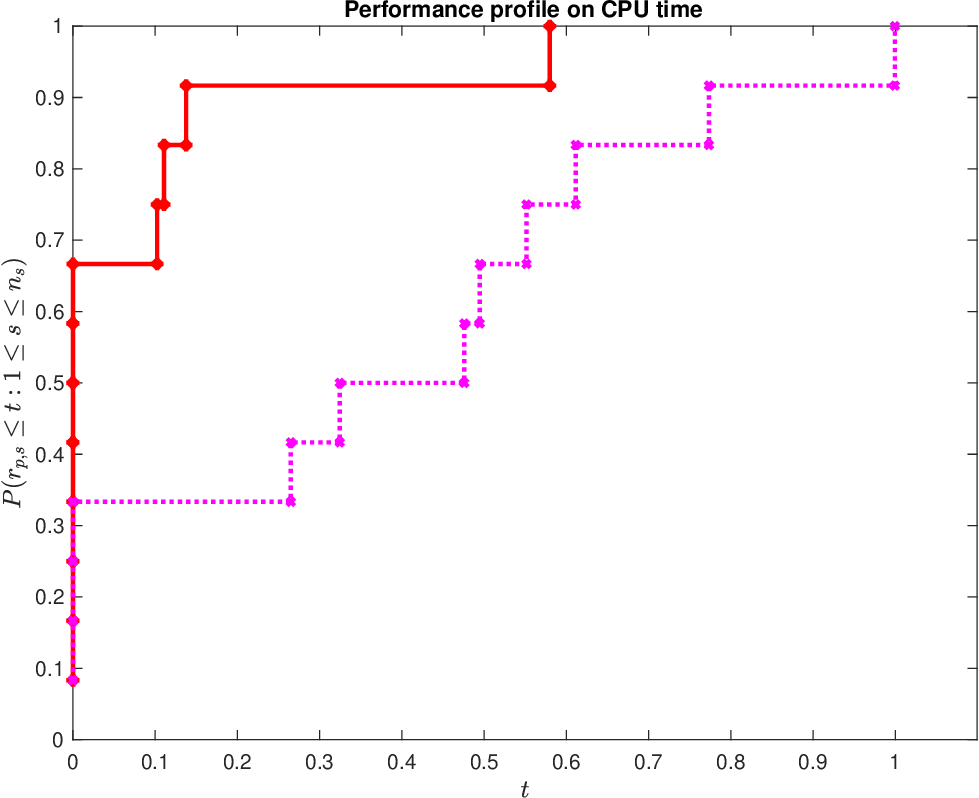}
\includegraphics[width=0.46\linewidth]{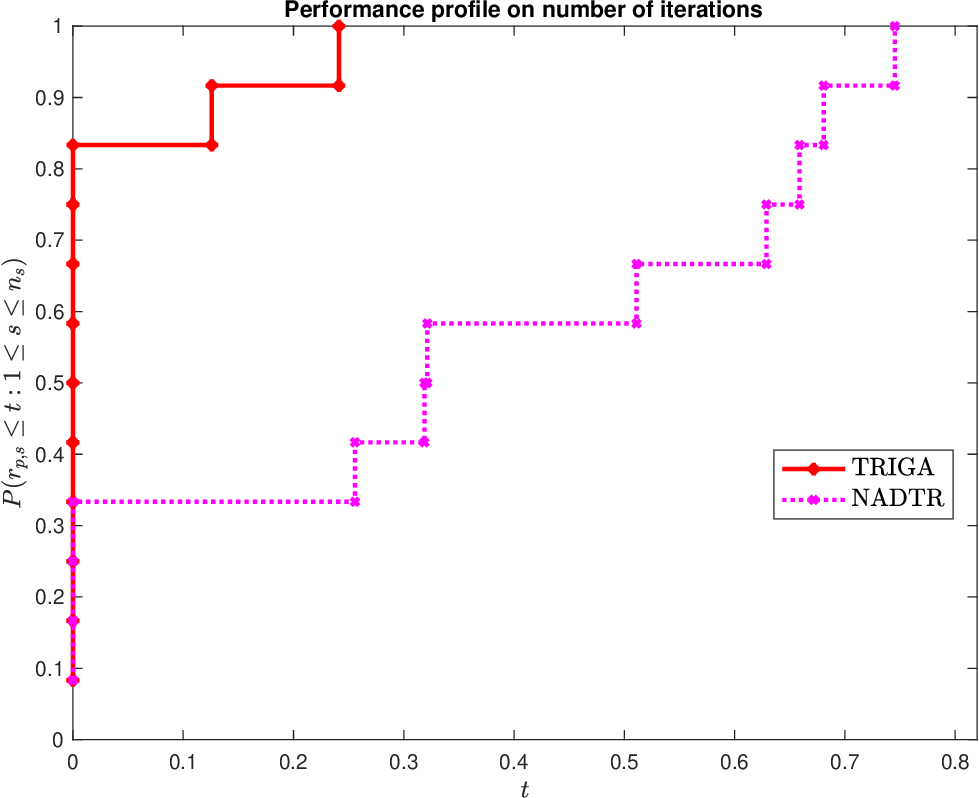}  
\caption{Performance profiles of \ref{algo:triga} and \ref{algo:nadtr} when $s = \frac{1}{1.1L}$ (first row) and $s = \frac{1}{2.1L}$ (second row) in terms of CPU time (in seconds) (left) and number of iterations (right) on real datasets of logistic regression problems}
\label{fig:performance_logistic}
\end{figure}

The performance curves in Figure~\ref{fig:performance_logistic} highlight the efficiency of \ref{algo:triga} in comparison with \ref{algo:nadtr}. Algorithm \ref{algo:triga} efficiently solves over all test problems, reaching at least the factor $t \geq 0.36$ of the best CPU time and $t \geq 0.33$ of the fewest iterations. In comparison, Algorithm \ref{algo:nadtr} achieves the same within $t \geq 0.89$ and $t \geq 0.84$, respectively. The performance gap of \ref{algo:triga} versus \ref{algo:nadtr} ranges from $8\%$ to $25\%$ (resp. from $8\%$ to $33\%$) in CPU time (resp. number of iterations). 

To sum up, the Tikhonov regularization parameter $\varepsilon_k = \frac{1}{k^{1.95}}$ and the step-size $s = \frac{1}{1.1L}$ (corresponding to the case $q > 1$) are the good choices for Algorithm~\ref{algo:triga} in these numerical experiments. Tikhonov regularization in our approach leads to the reduction in oscillations. The obtained results show that our algorithm runs faster and converges in fewer iterations, and performs more robustly than \ref{algo:nadtr} across all test problems.  

\section{Conclusion, perspective} \label{sec:conclusion}

We have introduced \ref{algo:triga}, a simple inertial gradient scheme with a Tikhonov regularization term, for solving convex optimization problems in Hilbert spaces. The method is obtained by time discretization of a second-order dynamical system with vanishing Tikhonov regularization. Our analysis establishes convergence rates for the objective values and the velocity sequence, together with strong convergence of the iterates. In particular, for $\varepsilon_k = 1/k^p$ with $0<p<2$, \ref{algo:triga} converges to the minimum-norm solution, while in the critical case $p=2$ it achieves fast rates on the objective values, comparable to those of Nesterov-type methods. Numerical experiments further indicate that \ref{algo:triga} performs very competitively on the considered benchmark, synthetic, and large-scale datasets, both in terms of iteration counts and practical running time.

Several directions deserve further investigation. First, beyond the prescribed decay laws studied here, it would be valuable to design adaptive strategies for choosing $(\varepsilon_k)_k$ (and possibly the inertial parameters), while preserving the selection mechanism and the theoretical guarantees. Second, motivated by large-scale learning, extending \ref{algo:triga} to stochastic or incremental variants (mini-batch gradients) raises nontrivial questions on the interplay between inertia, Tikhonov regularization, and noise. Third, it is natural to broaden the scope from smooth minimization to structured composite problems \eqref{prob:composite_fg}, by developing splitting counterparts of \ref{algo:triga} for objectives of the form ``smooth $+$ nonsmooth''. In parallel, following the Hessian-driven damping ideas of \cite{Attouch:amo:23}, one may also combine such damping mechanisms with Tikhonov regularization within a unified Lyapunov framework, with the aim of obtaining sharper guarantees and improved performance. Finally, it would be highly interesting to extend the present analysis beyond convexity, for instance to nonconvex settings under suitable regularity assumptions (e.g., Kurdyka-{\L}ojasiewicz type properties or prox-regularity), where the interaction between inertia and a vanishing Tikhonov term may still induce meaningful selection mechanisms. This direction is outside the scope of the present manuscript and will be addressed in future work.

\section*{Statements and Declarations}

The authors have no competing interests to declare that are relevant to the content of this article.

\begin{appendices}

\section{} \label{sec:appendix}

\begin{lemma} \label{lemma:sum-integral-inequality}
    Let $g \colon (0, +\infty) \to \mathbb{R}$ be a continuous and decreasing function. Then for all positive integer numbers $k_0$, $k$ such that $k_0 < k$, we have
    \begin{align*}
        \sum_{j=k_0}^{k-1} g(j+1) \leq \int_{k_0}^k g(t)dt \leq \sum_{j=k_0}^{k-1} g(j).
    \end{align*}
\end{lemma}
\begin{proof}
    By mean value theorem, for all $k_0 \leq j \leq k$, there exists $c \in ]j, j+1[$ such that 
    \begin{align*}
        \int_{j}^{j+1} g(t) dt = g(c).
    \end{align*}
    Since $g$ is continuous and decreasing on $[j, j+1]$, we have
    \begin{align*}
        g(j+1) \leq \int_{j}^{j+1} g(t) dt \leq g(j).
    \end{align*}
    Taking the sum for $j$ from $k_0$ to $k-1$, we obtain
    \begin{align*}
        \sum_{j=k_0}^{k-1} g(j+1) \leq \int_{k_0}^k g(t)dt \leq \sum_{j=k_0}^{k-1} g(j).
    \end{align*}
\end{proof}

\begin{theorem}[Stolz-Cesaro theorem] \cite[Theorem 1.22]{Muresan:springer:09} \label{theorem:stolz-cesaro}
    Let $(a_n)_{n\geq 1}$ and $(b_n)_{n \geq 1}$ be two sequences of real numbers. Suppose that the following conditions are satisfied
      \def\labelenumi{\rm (\roman{enumi})}
\def\theenumi{\roman{enumi}}
    \begin{enumerate} 
        \item $(b_n)_{n\geq 1}$ strictly increases to $+\infty$,

        \item the limit $\lim\limits_{n \to +\infty} \dfrac{a_{n+1} - a_n}{b_{n+1} - b_n} = l$ exists.
    \end{enumerate}
    Then we have $\lim\limits_{n \to + \infty} \dfrac{a_n}{b_n}= l$.
\end{theorem}

\begin{definition}[Big-O notation] \cite[Section 3.2]{Cormen:mit:22}
    Given two functions $f \colon \mathbb{R} \to \mathbb{R}$ and $g \colon \mathbb{R} \to \mathbb{R}^+$, we say that $f(x) = \mathcal{O}(g(x))$ if there exist a positive constant $c$ and a real number $x_0$ such that $|f(x)| \leq c g(x)$ for all $x \geq x_0$.
\end{definition}
   
\begin{lemma} \label{lemma:big-o}
     Given two functions $f \colon \mathbb{R} \to \mathbb{R}$ and $g \colon \mathbb{R} \to \mathbb{R}^+$, if $\lim\limits_{x \to + \infty} \dfrac{|f(x)|}{g(x)} < \infty$, then $f(x) = \mathcal{O}(g(x))$.
\end{lemma}
\begin{proof}
    Let $l \in \mathbb{R}$ be the limit $\lim\limits_{x \to +\infty} \dfrac{|f(x)|}{g(x)}$. There exist $\epsilon > 0$ and $x_0 \in \mathbb{R}$ such that
    \begin{align*}
        \dfrac{|f(x)|}{g(x)} \leq l + \epsilon, ~ \forall x \geq x_0.
    \end{align*}
    Therefore, $|f(x)| \leq (l + \epsilon) g(x)$ for all $x \geq x_0$ which implies that $f(x) = \mathcal{O}(g(x))$.
\end{proof}

\begin{lemma}[Descent lemma] \cite[Lemma 4.22]{Beck:siam:17} \label{lemma:descent-lemma}Let $g \colon \mathcal{H} \to \mathbb{R}$ be a differentiable convex function whose gradient $\nabla g$ is $L$-Lipschitz continuous. For all $x,y \in \mathcal{H}$, we have
\begin{align*}
    g(y) \leq g(x) + \langle \nabla g(x), y - x \rangle + \dfrac{L}{2} \| x - y \|^2.
\end{align*}    
\end{lemma}

\begin{lemma}[Extended descent lemma] \label{lemma:extended-descent-lemma}
    Let $g \colon \mathcal{H} \to \mathbb{R}$ be a $\mu$-strongly convex, continuously differentiable function whose gradient $\nabla g$ is $L$-Lipschitz continuous. For all $x, y \in \mathcal{H}$ and all $s>0$, we have the following inequality
    \begin{equation} \label{extended-descent-lemma-inequality}
        g(y - s\nabla g(y)) \leq g(x) + \langle \nabla g(y), y - x \rangle + \left( \dfrac{Ls^2}{2} - s \right) \|\nabla g(y)\|^2 - \dfrac{\mu}{2} \|x - y\|^2.
    \end{equation}
\end{lemma}

\begin{proof}
    By Lemma \ref{lemma:descent-lemma}, we have for all $x, y \in \mathcal{H}$ and $s>0$
    \begin{align*}
        g(y - s\nabla g(y)) &\leq g(y) + \langle \nabla g(y), -s\nabla g(y) \rangle + \dfrac{L}{2} \| s\nabla g(y)\|^2 \\
        & = g(y) + \left( \dfrac{Ls^2}{2} - s \right) \|\nabla g(y)\|^2.
    \end{align*}
    By the strong convexity of $g$, we have
    \begin{align*}
        g(y) \leq g(x) + \langle \nabla g(y), y -x \rangle - \dfrac{\mu}{2}\|x-y\|^2.
    \end{align*}
    From the above inequalities, we obtain \eqref{extended-descent-lemma-inequality}.
\end{proof}

\end{appendices}

\end{document}